\documentclass[12pt]{article}
\usepackage{amsmath,amssymb, mathtools} 
\usepackage{amsthm} 
\usepackage{dsfont} 
\usepackage{MnSymbol} 

\usepackage[pdftex]{graphicx} 
\usepackage{color}
\usepackage[colorlinks,citecolor=blue,linkcolor=blue,urlcolor=blue]{hyperref} 
\usepackage{cleveref}
\usepackage{tikz-cd} 

\usepackage[utf8]{inputenc} 
\usepackage[english]{babel} 

\usepackage{geometry} 
\usepackage{tikz} 
\usepackage{cite} 
\usepackage{enumitem} 
\usepackage{comment} 
\usepackage[bottom]{footmisc}

\newtheorem{theorem}{Theorem}[section]
\newtheorem{proposition}[theorem]{Proposition}
\newtheorem{lemma}[theorem]{Lemma}
\newtheorem{corollary}[theorem]{Corollary}
\newtheorem{question}[theorem]{Question}
\newtheorem{example}[theorem]{Example}

\newtheorem{definition}[theorem]{Definition}
\newtheorem*{remark}{Remark}

\newtheorem*{fakedefinition}{Definition}
\newtheorem{faketheorem}{Theorem}

\newtheorem{fakeprop}[faketheorem]{Proposition}
\newtheorem{fakequestion}[faketheorem]{Question}

\numberwithin{theorem}{section} 
\numberwithin{equation}{section} 
\numberwithin{figure}{section} 

\newcommand{\C}{\mathbb{C}}
\newcommand{\R}{\mathbb{R}}
\newcommand*{\g}{\mathfrak{g}}

\renewcommand*{\l}{\lambda}
\newcommand*{\mf}{\mathfrak}

\newcommand{\del}{\partial}
\newcommand{\delbar}{\bar{\partial}}

\DeclareMathOperator{\tr}{tr}
\DeclareMathOperator{\rk}{rk}
\DeclareMathOperator{\SL}{SL}
\DeclareMathOperator{\PSL}{PSL}
\DeclareMathOperator{\Hom}{Hom}
\DeclareMathOperator{\ad}{ad}
\DeclareMathOperator{\Ad}{Ad}

\DeclareMathOperator{\Aut}{Aut}
\DeclareMathOperator{\End}{End}
\DeclareMathOperator{\Hit}{Hit}

\title{Fock bundles and Hitchin components}

\author{Georgios Kydonakis, Charles Reid, Alexander Thomas}

\date{}

\begin{document}
\maketitle
\begin{abstract}
    We introduce the concept of a Fock bundle, a smooth principal bundle over a surface equipped with a special kind of adjoint-valued 1-form, as a new tool for studying character varieties of surface groups. Although similar to Higgs bundles, the crucial difference  is that no complex structure is fixed on the underlying surface. Fock bundles are the gauge-theoretic realization of higher complex structures. 
    We construct a canonical connection to a Fock bundle equipped with compatible symmetric pairing and hermitian structure. The space of flat Fock bundles maps to the character variety of the split real form. Determining the hermitian structure such that this connection is flat gives a non-linear PDE similar to Hitchin's equation.
    We explicitly construct solutions for Fock bundles in the Fuchsian locus. Ellipticity of the relevant linear operator provides a map from a neighborhood of the Fuchsian locus in the space of higher complex structures modulo higher diffeomorphisms to a neighborhood of the Fuchsian locus in the Hitchin component.
\end{abstract}

\tableofcontents

\section{Introduction and results}

\subsection{Context}

Let $S$ be a smooth closed orientable surface of genus at least 2.
The moduli space of complex structures on $S$ modulo diffeomorphisms isotopic to the identity is famously known as the \emph{Teichm\"{u}ller space}. The Poincaré--Koebe uniformization theorem implies that the Teichm\"{u}ller space also describes all marked hyperbolic structures on $S$. The holonomy representation of a hyperbolic structure leads to a discrete and faithful representation $\pi_1 S \to \mathrm{PSL}_2(\mathbb{R})$, thus allowing for an algebraic realization of the Teichm\"{u}ller space of $S$ as a connected component of the character variety $\mathcal{X}(\pi_1 S,\mathrm{PSL}_2(\mathbb{R}))$.
Generalizing the uniformization theorem to higher rank Lie groups is the main motivation of this work.
 
Representations of fundamental groups and their links to geometric structures are captured by the character variety. For a Lie group $G$, the $G$-\emph{character variety} $$\mathcal{X}(\pi_1 S,G)=\mathrm{Hom}(\pi_1 S,G)/G$$ is defined as the space of isomorphism classes of completely reducible representations, where $G$ acts by conjugation.

Higher Teichm\"{u}ller theory concerns the study of connected components of character varieties  for more general Lie groups than $\mathrm{PSL}_2(\mathbb{R})$ which share analogous properties to the classical Teichm\"{u}ller space. The first step in this direction was taken by Hitchin in \cite{Hit92}, where he found a connected component of $\mathcal{X}(\pi_1 S,\mathrm{PSL}_n(\mathbb{R}))$ (and more generally for any adjoint group of the split real form of a complex simple Lie group) which is contractible, and naturally contains a copy of Teichm\"uller space. Representations parametrized by these components, now most often called \emph{Hitchin components}, were later shown to be discrete and faithful \cite{Labourie06, FG}.

The \emph{mapping class group} $\mathrm{Mod}(S)$ of the surface $S$, that is, the group of all orien\-tation-preserving diffeomorphisms of $S$ modulo the ones which are isotopic to the identity, acts naturally on the Teichm\"{u}ller space by changing the marking. This action is properly discontinuous and the resulting quotient is the moduli space of Riemann surfaces with the same topological type as $S$. Similarly, $\mathrm{Mod}(S)$ acts on any character variety by precomposition.

Fixing a complex structure on $S$, the nonabelian Hodge correspondence identifies $\chi(\pi_1 S,G)$, for any reductive Lie group $G$, with the moduli space of polystable $G$-Higgs bundles. The nonabelian Hodge correspondence is very effective in providing complex analytic descriptions of character varieties $\mathcal{X}(\pi_1 S,G)$, their Hitchin components, and other Teichm\"uller space-like components, now called \emph{positive} components. However, a major drawback in this correspondence is that one has to fix a complex structure on $S$, and so $\mathrm{Mod}(S)$ does not act canonically on the moduli of Higgs bundles.  

In \cite{Labourie}, Labourie proposed a remedy which, if generally applicable, could provide a canonical way to associate a complex structure on $S$ to a given point in the Hitchin component. This analytic method involving harmonic maps and minimal surfaces has received considerable attention and has proven to be effective in several cases of higher Teichm\"{u}ller spaces but only for groups of rank 2; we refer to \cite{Lab17} and \cite{CTT} for general accounts on split real simple Lie groups of rank 2 and Hermitian Lie groups of rank 2 respectively. Yet, it recently became clear that this approach via minimal surfaces towards establishing a canonical choice of complex structure on $S$ and thus obtaining a mapping class group equivariant parametrization of the associated higher Teichm\"{u}ller spaces in terms of holomorphic data, fails for groups of rank greater than 2. In \cite{SaSm}, Sagman and Smillie demonstrated the existence of numerous equivariant minimal surfaces for certain Hitchin representations $\rho: \pi_1(S) \to G$, when $G$ is a split real semisimple Lie group of rank at least 3, building on the breakthrough accomplished by Markovi\'{c} \cite{Mar}, who applied his New Main Inequality to get an analogous statement in the case of $\mathrm{PSL}_2(\mathbb{R})^3$. In conclusion, the minimal surface approach gives a well-defined map on a neighborhood of the Fuchsian
locus, yet the existence of Hitchin representations with several equivariant minimal surfaces prevents it from being injective.

A new approach towards realizing an equivariant parametrization of the Hitchin components was proposed by Vladimir Fock and the third author \cite{FT} in terms of newly introduced geometric structures called \emph{higher complex structures}. These were originally defined using the punctual Hilbert scheme of the plane and a moduli space of higher complex structures was defined in \cite{FT} as a quotient by Hamiltonian diffeomorphisms of the cotangent bundle of the surface preserving the zero section $S\subset T^*S$ setwise. Such diffeomorphisms were called \emph{higher diffeomorphisms} and the resulting quotient space, called the \emph{geometric Hitchin space} and denoted by $\hat{\mathcal{T}}^n(S)$, recovers the classical Teichm\"{u}ller space when $n=2$. Moreover, this space is a manifold of complex dimension equal to that of the Hitchin component for $G=\mathrm{PSL}_2(\mathbb{R})$ \cite[Theorem 2]{FT}. Its cotangent bundle can be parametrized by a set of tensors on the underlying surface, satisfying a certain compatibility condition called \emph{$\mu$-holomorphicity} \cite[Theorem 3]{FT}. The main conjecture suggested in \cite{FT} was that the geometric Hitchin space is canonically diffeomorphic to the Hitchin component.   

In \cite{Nolte}, Nolte introduced the notion of \emph{harmonic} representatives for higher complex structures and studied in detail the group of higher diffeomorphisms. 
There, the geometric Hitchin space is denoted by $\mathcal{T}^n(S)$ and is called the degree-n Fock–Thomas space. Moreover, a canonical diffeomorphism from $\mathcal{T}^3(S)$ to the $\mathrm{PSL}_3(\mathbb{R})$-Hitchin component was constructed. This construction is, in fact, $\mathrm{Mod}(S)$-equivariant. More generally, the space $\mathcal{T}^n(S)$ is diffeomorphic to a ball of complex dimension $(g-1)(n^2-1)$, meaning that it is abstractly diffeomorphic to the Hitchin component; see \cite[Theorem 2]{FT} and \cite[Corollary 1.7]{Nolte}. Nolte's method, however, for getting this canonical $\mathrm{Mod}(S)$-equivariant diffeomorphism uses the affine spheres perspective on the $\mathrm{PSL}_3(\mathbb{R})$-Hitchin component by Labourie \cite{Lab07} and Loftin \cite{Loftin}, and so the technique does not directly apply for higher rank. 

\medskip
In this article, we introduce a gauge-theoretic realization of higher complex structures which allows us to make progress on their link to character varieties, in particular to Hitchin components.
In addition, central associated notions such as higher diffeomorphisms and the $\mu$-holomorphicity condition become more transparent. 

We introduce a new type of object $(P,\Phi,\sigma)$, that we call \emph{$G$-Fock bundle}, consisting of a smooth principal $G$-bundle $P$ together with an adjoint-valued 1-form $\Phi$, called a \emph{Fock field}, that satisfies certain conditions. In this triple, $\sigma$ is an involution on $P$ generalizing the notion of a symmetric pairing on a vector bundle. This notion is similar to a $G$-Higgs bundle, with the crucial difference that it does not involve fixing a priori any complex structure on the underlying surface $S$. Isomorphism classes of $\PSL_n(\C)$-Fock bundles are equivalent to higher complex structures of order $n$. A $G$-Fock field $\Phi$ as above naturally induces a complex structure on $S$. Varying the Fock field also varies this complex structure. 

We then set to establish a passage from $G$-Fock bundles to connections analogously to the nonabelian Hodge correspondence. 
Our theory associates to a Fock bundle $(P,\Phi,\sigma)$ equipped with a compatible hermitian structure $\rho$, a connection $\nabla=\Phi+d_A+\Phi^{*}$, where $\Phi^{*}=-\rho(\Phi)$ is the hermitian conjugate of $\Phi$ and $d_A$ is the unique unitary $\sigma$-invariant connection satisfying $d_A\Phi=0$. We ask under which conditions there exists a hermitian structure such that $\nabla$ is flat and give a partial answer. 

The flatness equation for $\nabla$ is similar to Hitchin's self-duality equations over a Riemann surface \cite{Hit87} and, in fact, coincides with Hitchin's equation in the most trivial examples, the so-called Fuchsian locus (Fock bundles coming from uniformizing Higgs bundles). 
Moreover in a neighborhood of the Fuchsian locus, we prove existence of a hermitian structure giving a flat connection.
For the family of flat $G$-Fock bundles, we show that the monodromy of the connection is always in the split real form of $G$. 

Rudiments of the approach towards establishing a passage to a canonical family of flat connections have appeared in the previous articles \cite{thomas-flat-conn}, \cite{ThoWKB} and \cite{Tho22} by the third author. In particular, it was shown in \cite{thomas-flat-conn} that the cotangent bundle of the moduli space of higher complex structures can be included into a 1-parameter family of spaces whose sections are flat formal connections.  The theory of $\mathfrak{g}$-complex structures was introduced in \cite{Tho22}, extending the case of $\mathfrak{sl}_n(\C)$ to a general complex simple Lie algebra $\g$, and constitutes the starting point for our definition of $G$-Fock bundles.

The introduction of Fock bundles and their relationship to families of connections points towards a broader framework to study character varieties, which includes both the theory of Higgs and Fock bundles.
Fixing a complex structure on $S$, one may consider a $G$-principal bundle $P$ on $S$, together with a connection $d_A$ and a field $\Phi\in \Omega^1(S,\mathfrak{g}_P)$ satisfying $[\Phi\wedge\Phi]=0$, $d_A(\Phi)=0$ and with fixed conjugacy class of $\Phi^{0,1}$. Here we used the complex structure to decompose $\Phi$ into Hodge types. Under this setup, then whenever $\Phi^{0,1}=0$, the field $\Phi$ is holomorphic with respect to the holomorphic structure on $P$ given by $d_A^{0,1}$, thus leading to the theory of Higgs bundles. On the other hand, whenever $\Phi^{0,1}$ is principal nilpotent, this leads to our theory of Fock bundles.
We hope that to such data, one can always associate a flat connection, in analogy to the nonabelian Hodge correspondence. 
Alternative conditions for the conjugacy class of $\Phi^{0,1}$ can possibly lead to alternative approaches to character varieties, yet the case when $\Phi^{0,1}$ is principal nilpotent describes the only possible conjugacy class which stays invariant under coordinate change on $S$. In other words, \emph{the approach via Fock bundles within this broader framework is the only one which is independent of the complex structure on $S$}.

\subsection{Results and structure}

We now make the statements appearing in the article more precise. Throughout the whole paper we denote by $S$ a smooth closed connected orientable surface of genus at least 2 and by $G$ a complex simple Lie group with Lie algebra $\g$. 
The main new object we introduce is the notion of a $G$-Fock bundle. We restrict to the case $G=\mathrm{SL}_n(\mathbb{C})$ here for simplicity and give the full definition in Section \ref{Sec:Fock-bundles}.

\begin{fakedefinition}[Definition \ref{defn_Fock-bundle}]
     An \emph{$\mathrm{SL}_n(\mathbb{C})$-Fock bundle} over $S$ is a triple $(E,\Phi,g)$ where $E$ is a complex vector bundle over $S$ with fixed volume form, non-degenerate symmetric pairing $g:E\times E\to \underline{\mathbb{C}}$, where $\underline{\mathbb{C}}$ denotes the trivial line bundle, and $\Phi\in \Omega^1(S,\mf{sl}(E))$ satisfying
    \begin{enumerate}
        \item $\Phi\wedge\Phi = 0$,
        \item $\Phi(v)(z)$ is principal nilpotent for all $z\in S$ and all non-zero vectors $v\in T_zS$.
        \item $\Phi$ is $g$-self-adjoint. 
    \end{enumerate}
We refer to $\Phi$ as the \emph{Fock field}.
\end{fakedefinition}

This generalizes easily to general $G$: a $G$-Fock bundle is a principal $G$-bundle $P$ with a certain involution $\sigma$ playing the role of the symmetric pairing, and an adjoint-valued 1-form $\Phi$ satisfying the analogous three properties above. An important example of $G$-Fock bundles, the so-called \emph{Fuchsian locus}, arises from the underlying smooth principal bundle to the uniformizing $G$-Higgs bundle when we equip $S$ with a complex structure (see Section \ref{Sec:fuchsian-locus}). The link to higher complex structures is given by the following:
 
\begin{fakeprop}[Proposition \ref{Prop:link-to-g-C-str}]
An $\mathrm{SL}_n(\mathbb{C})$-Fock bundle induces a higher complex structure of order $n$ on $S$. Isomorphism classes of $\mathrm{PSL}_n(\mathbb{C})$-Fock bundles are equivalent to higher complex structures. 
\end{fakeprop}
 
In Subsection \ref{Sec:var-Fock-fields} we describe variations of Fock bundles via the cohomology of a certain chain complex. In particular, we prove that a Fock bundle has no infinitesimal automorphisms, hence is stable in that sense.

The main result of Section \ref{Sec:can-connection} is the construction of a canonical connection associated to a Fock bundle equipped with a compatible positive hermitian structure. A hermitian structure $\rho$ is an involution on $P$ associated to the compact real form of $G$. It is said to be \emph{compatible} with a Fock bundle if $\rho$ and $\sigma$ commute. The property of being \emph{positive} is a certain open condition, see Definition \ref{Def:pos-Higgs-field} for the precise statement.
\begin{faketheorem}[Theorem \ref{Thm-filling-in}]
    For a $G$-Fock bundle $(P,\Phi,\sigma)$ equipped with a compatible, positive hermitian structure $\rho$, there is a unique unitary, $\sigma$-invariant connection $d_A$ satisfying $d_A\Phi=0$.
\end{faketheorem} 
This result is analogous to the existence of a Chern connection on a holomorphic bundle induced by a hermitian structure.
Thus, to the data $(P,\Phi,\sigma,\rho)$ as above we can associate a connection $\Phi+d_A+\Phi^{*}$, where $\Phi^*=-\rho(\Phi)$, which preserves a split-real structure, and has curvature 
\begin{equation}\label{Eq-intro:hitchin-like}
F(A)+[\Phi\wedge\Phi^{*}].
\end{equation}

Our main open question, which we partially address in this paper, is the following:
\begin{fakequestion}[Question \ref{main-conj}]\label{Conj-intro}
    For which $G$-Fock bundles $(P,\Phi,\sigma)$, does there exist a unique compatible positive hermitian structure $\rho$ such that the associated connection $\Phi+d_A+\Phi^{*}$ is flat, in other words, such that $F(A)+[\Phi\wedge\Phi^{*}]=0$?
\end{fakequestion} 
If there is always a unique solution, then this gives a map from the space of Fock bundles to the space of Hitchin representations. For a Fock bundle from the Fuchsian locus, we show existence by the nonabelian Hodge correspondence (for the uniformizing Higgs bundle). 

Section \ref{Sec:ellipticity} analyzes Question \ref{Conj-intro} near the Fuchsian locus. The following result shows that the Hitchin-like equation we have to solve is elliptic:
\begin{faketheorem}[Theorem \ref{Thm:ellipticity}]
Let $(P,\Phi,\sigma)$ be a $G$-Fock bundle equipped with a compatible positive hermitian structure $\rho$. The derivative of the map from the $\Aut(P,\sigma)$-conjugacy class of $\Phi$ to the curvature of the resulting connection $F(A)+[\Phi\wedge\Phi^{*}]$ is an elliptic isomorphism.
\end{faketheorem}

We then use the implicit function theorem in Section \ref{Sec:impl-fct-thm} to conclude that the space of solutions to Equation (\ref{Eq-intro:hitchin-like}) is a Banach manifold which maps locally diffeomorphically to the space of Fock bundles. In particular, we get a Mapping Class Group equivariant map from a neighborhood of the Fuchsian locus in the space of higher complex structures to the Hitchin component. In \cite{Nolte} it is shown that on each higher diffeomorphism orbit there is a harmonic representative unique up to isotopy. By restricting our map to this slice, we get the following result, analogously to the minimal surface approach:

\begin{faketheorem}[Theorem \ref{Thm:neighborhood}]
    There is an open neighborhood of the Fuchsian locus in the moduli space $\mathcal{T}^n(S)$ of higher complex structures of order $n$, which has a canonical map to the $\PSL_n(\R)$-Hitchin component, via solution of Equation (\ref{Eq-intro:hitchin-like}).
\end{faketheorem}

It is expected that the restriction to harmonic higher complex structures is not actually necessary because our map to the character variety should be constant along higher diffeomorphism orbits. In Section \ref{Sec:higher-diffeos}, we analyse the action of special $\lambda$-dependent gauge transformations on a family of flat connections of the form $\lambda^{-1}\Phi+d_A+\lambda \Phi^{*}$ where $\lambda\in\mathbb{C}^*$ is a parameter.
In the case of $\mathrm{SL}_n(\C)$, these special gauge transformations relate to higher diffeomorphisms:
\begin{faketheorem}[Theorem \ref{hamiltonian-first-variations-coincide}]
    The variation on an $\mathrm{SL}_n(\C)$-Fock field $\Phi$ induced by an infinitesimal gauge transformation $\lambda^{-1} \eta$ with $\eta=\Phi(v_1)\cdots \Phi(v_k)$ is equivalent to the infinitesimal action of the Hamiltonian $H=v_1\cdots v_k$ on the higher complex structure induced by $\Phi$.
\end{faketheorem}
This theorem gives a clear gauge-theoretic meaning to infinitesimal higher diffeomorphisms, which in the theory of higher complex structures are Hamiltonian diffeomorphisms of $T^*S$ preserving the zero-section $S\subset T^*S$ setwise. 
This realization of higher diffeomorphisms as gauge transformations, together with our ellipticity result gives the following:

\begin{fakeprop}[Proposition \ref{prop:constancy}]
A differentiable family $\Phi_t$ of solutions to Equation (\ref{Eq-intro:hitchin-like}) for $G=\SL_n(\C)$ which induces a family of higher diffeomorphic higher complex structures maps to a constant path of representations.
\end{fakeprop}

In the final Section \ref{Sec:covectors} we consider $G$-Fock bundles $(P,\Phi,\sigma,\rho)$ equipped with a hermitian structure $\rho$ not necessarily commuting with $\sigma$. We parameterize the space of unitary connections $d_A$ satisfying $d_A\Phi=0$ by so-called \emph{covectors}. In the case of $\mathrm{SL}_n(\C)$, covectors can be identified with cotangent vectors to the space of higher complex structures. We expect that the equation $F(A)+[\Phi\wedge\Phi^{*}]=0$ can still be solved for small and $\mu$-holomorphic covectors and that the monodromies of the flat connections $\Phi+d_A+\Phi^{*}$ describe a tubular neighborhood of the $\mathrm{SL}_n(\mathbb{C})$-Hitchin component inside the complex character variety. This picture generalizes the work of Donaldson \cite{Donaldson} and Trautwein \cite{Traut} on the space of almost-Fuchsian representations in the $\mathrm{SL}_2(\C)$-case.
The $\mu$-holomorphicity condition from the theory of higher complex structures gets a precise gauge-theoretical interpretation:
\begin{faketheorem}[Theorem \ref{Thm:mu-holo}]
    For an $\mathrm{SL}_n(\C)$-Fock bundle, the condition $F(A)\in\mathrm{Im}(\ad_\Phi)\subset\Omega^2(S,\g_P)$ is equivalent to the $\mu$-holomorphicity condition \eqref{Eq:mu-holo-cond-hcs}.
\end{faketheorem}

\bigskip
\noindent\textbf{\textit{Acknowledgements.}}
We warmly thank Abdelmalek Abdesselam, Andrea Bianchi, Jeff Danciger, Vladimir Fock, Dan Freed, Andy Neitzke and Alex Nolte for many fruitful discussions and insights. We are also very grateful to the referee for a careful reading of
the manuscript and for their comments. We acknowledge support from the Alexander von Humboldt Foundation, the National Science Foundation (grants DMS-1937215, and DMS-1945493), the SPP 2026 \textit{Geometry at Infinity} for a travel grant, the European Research Council (ERC-Advanced Grant 101018839), the Deutsche Forschungsgemeinschaft (Project-ID 281071066 - TRR 191), and the University of Heidelberg and the University of Texas at Austin where most of the work has been carried out.

\section{Preliminaries}

We gather necessary material for the main part of the text. We review higher complex structures in Subsection \ref{Sec:HCS}, the nonabelian Hodge correspondence in \ref{Sec:nonabelian-Hodge}, the Lie theoretic background in \ref{Sec:princ-nilp-elements} (especially principal nilpotent elements) and involutions on principal $G$-bundles in \ref{Sec:involutions}. Part of the material in the Lie-theoretic section seems new, the rest is well-known.

\subsection{Higher complex structures}\label{Sec:HCS}

Motivated by the aspiration to describe components of real character varieties as moduli spaces of geometric structures, Vladimir V. Fock and the third author introduced higher complex structures \cite{FT} and conjectured that they parametrize $\mathrm{PSL}_n(\mathbb{R})$-Hitchin components. These were, in turn, generalized to $\g$-complex structures where $\g$ is a complex simple Lie algebra in \cite{Tho22}. We give here a brief account, concentrating on the properties we need in the subsequent sections.

\subsubsection{Definitions}

Let $S$ be a closed orientable surface of genus at least 2. A complex structure on $S$ is equivalent to an \emph{almost complex structure}, that is, an automorphism $J$ of $TS$ satisfying $J^2=-\mathrm{id}$. The complexified tangent bundle then decomposes as
$$T^\C S=T^{1,0}S\oplus T^{0,1}S,$$
where $T^{1,0}S$ and $T^{0,1}S$ are pointwise the eigendirections of $J$. Since $T^{1,0}S$ is the complex conjugate of $T^{0,1}S$, the complex structure is uniquely encoded by $T^{1,0}S$ which is a non-real direction of $T^\C S$. This in turn is uniquely encoded by a certain ideal $I$ in $\mathrm{Sym}(T^{\mathbb{C}}S)$ generated by $T^{0,1}S$ and $(T^{1,0}S)^2$. Geometrically, $\mathrm{Sym}(T^{\mathbb{C}}S)$ is the ring of functions on $T^{*\mathbb{C}}S$ which are polynomial on fibers, and $I$ cuts out an infinitesimal thickening of the zero section of $T^\C S$ in the direction of $T^{1,0}S$. 

To describe this equivalence more explicitly, fix a reference complex structure on $S$ and local coordinates $(z,\bar{z})$. Denote by $(p,\bar{p})$ the linear coordinates on $T^{*\C}S$ corresponding to vector fields $\del_z$ and $\del_{\bar{z}}$. Any other complex structure can be then described by an ideal $I$ locally of the form 
$$I=\langle p^2, \bar{p}-\mu(z,\bar{z}) p\rangle,$$
where $\mu$ is known as the \emph{Beltrami differential} and satisfies $\lvert \mu\rvert \neq 1$. The holomorphic cotangent bundle of the new complex structure is defined by the equation $\bar{p} = \mu(z,\bar{z})p$.

\begin{definition}
    A \emph{higher complex structure of order $n$} (or of rank $n-1$) on $S$ is a special ideal $I$ in $\mathrm{Sym}(T^{\mathbb{C}}S)$ locally of the form
\begin{equation}\label{Eq:hcs-ideal}
    I=\langle p^n, -\bar{p}+\mu_2(z,\bar{z}) p+\mu_3(z,\bar{z}) p^2+...+\mu_n(z,\bar{z}) p^{n-1}\rangle,
\end{equation}
where $\mu_2=\mu$ is the usual Beltrami differential and $\mu_\ell$ for $\ell=3,...,n$ are called \emph{higher Beltrami differentials} (see \cite[Proposition 1]{FT}).
\end{definition}
 Globally, $\mu_\ell$ is a smooth section of $K^{1-\ell}\otimes \bar{K}$ where $K$ denotes the canonical bundle. A usual complex structure is a higher complex structure of order 2, i.e. of rank 1. 

An important feature of these structures is the forgetful map, which associates to a higher complex structure of order $n$ a structure of order $n-1$ by forgetting the last Beltrami differential $\mu_n$. In particular, \emph{any higher complex structure induces a complex structure on $S$}.

\begin{remark}\label{rem:orientation}
    The space of higher complex structures as defined above has two connected components. A higher complex structure induces a complex structure which in turn induces an orientation on $S$. This orientation coincides with the one induced from the reference complex structure iff $\lvert\mu_2\rvert <1$. Changing the reference complex structure to the complex conjugate one changes $\mu_2$ to $1/\bar{\mu}_2$ which is of norm strictly bigger than 1.
\end{remark}

More generally, let $\g$ be a complex simple Lie algebra.
An element $x\in\g$ is called \emph{principal nilpotent} if $\ad_x$ is nilpotent and $\dim Z(x)=\mathrm{rk}\,\g$, where $Z(x)=\{y\in \g\mid [x,y]=0\}$ denotes the centralizer (see also Section \ref{Sec:princ-nilp-elements}). 

\begin{definition}[Definition 4.1 in \cite{Tho22}]\label{def-g-C-str}
    A \emph{$\g$-complex structure} on $S$ is a $G$-conjugacy class of fields $\Phi\in\Omega^1(S,\g)$ locally of the form $\Phi=\Phi_1(z,\bar{z})dz+\Phi_2(z,\bar{z})d\bar{z}$ such that $\Phi_1$ is principal nilpotent and $\Phi_2\in Z(\Phi_1)$ commutes with $\Phi_1$ satisfying a certain inequality explained below. 
\end{definition}

There is a unique linear combination of the form $\mu_2\Phi_1-\Phi_2$ which is not principal nilpotent \cite[Proposition 2.21]{Tho22}. The inequality we impose is $\lvert\mu_2\rvert <1$. This allows to identify $\mu_2$ with a Beltrami differential of a complex structure on $S$. Therefore a $\g$-complex structure induces a complex structure on $S$.

For $\g=\mathfrak{sl}_n(\C)$ we get the notion of higher complex structure as follows: principal nilpotent elements in $\g$ form a single conjugacy orbit. So we can fix a principal nilpotent element $F\in\mathfrak{sl}_n(\C)$ and use the gauge freedom to fix $\Phi_1=F$. Using the standard representation of $\mathfrak{sl}_n(\C)$ on $\C^n$, the centralizer of $F$ in $\g$ is generated by all polynomials in $F$ without constant term. Hence the ideal of polynomials $P\in\C[p,\bar{p}]$ such that $P(\Phi_1,\Phi_2)=0$ (which makes sense since $\Phi_1$ and $\Phi_2$ commute) is of the form \eqref{Eq:hcs-ideal}.

\subsubsection{Moduli space}

For complex structures, the associated moduli space is the \emph{Teichm\"{u}ller space}, where complex structures are considered modulo diffeomorphisms of $S$ isotopic to the identity. For higher complex structures, one needs to mod-out by a larger group in order to get a finite-dimensional moduli space. We restrict exposition to the case  $\mathfrak{g}=\mathfrak{sl}_n(\C)$ here. For $\g$ of classical type, see \cite[Section 4]{Tho22}. For general complex simple $\g$, the moduli space of $\mathfrak{g}$-complex structures is not constructed yet.

\begin{definition}[Definition 3 in \cite{FT}]
    A \emph{higher diffeomorphism} of a surface $S$ is a hamiltonian diffeomorphism of $T^*S$ preserving the zero-section $S \subset T^*S$ setwise. The group of higher diffeomorphisms is denoted by $\mathrm{Ham}_0(T^*S)$.
\end{definition}

Diffeomorphisms of $T^*S$ fixing the zero-section act on the completed symmetric algebra of formal power series: 
\[\widehat{\mathrm{Sym}} (T^\C S) = \bigoplus_{k=0}^\infty \mathrm{Sym}^k (T^\C S).\]
The map from ideals in $\mathrm{Sym} (T^\C S)$ to ideals in  $\widehat{\mathrm{Sym}} (T^\C S)$ is injective on those which contain some power of the ideal $(T^\C S) \subset \mathrm{Sym} (T^\C S)$. This includes higher complex structures, so we can just as well view higher complex structures as ideals in $\widehat{\mathrm{Sym}} (T^\C S)$ where diffeomorphisms naturally act. A precise study of the space $\mathrm{Ham}_0(T^*S)$ and the action by diffeomorphisms appears in \cite[Section 7]{Nolte}. 

\begin{definition}
    The \emph{moduli space of higher complex structures of order $n$}, denoted by $\mathcal{T}^n(S)$, is the space of higher complex structures of order $n$, denoted by $\mathbb{M}^n(S)$, modulo higher diffeomorphisms. 
    
    The \emph{Fuchsian locus} in $\mathbb{M}^n(S)$ consists of those higher complex structures with trivial higher Beltrami differentials. The Fuchsian locus in $\mathcal{T}^n(S)$ is the image under the projection $\mathbb{M}^n(S)\to\mathcal{T}^n(S)$.
\end{definition}

This moduli space is finite-dimensional of dimension $(n^2-1)(2g-2)$, contractible and allows a forgetful map $\mathcal{T}^n(S)\to\mathcal{T}^{n-1}(S)$, see \cite[Theorem 2]{FT} and \cite[Theorem 1.1]{Nolte}. The Fuchsian locus inside $\mathcal{T}^n(S)$ is a copy of Teichm\"uller space, which is $\mathcal{T}^2(S)$. The main conjecture within this theory concerns the existence of a canonical diffeomorphism between $\mathcal{T}^n(S)$ and the $\mathrm{PSL}_n(\R)$-Hitchin component which is equivariant with respect to the natural mapping class group action. This conjecture has been proven for $n=3$ by Nolte in \cite{Nolte} using techniques which are special to $n=3$ and are analogous to a positive resolution to the Labourie Conjecture for $n=3$ \cite{Labourie}, but are known to fail for higher $n$ (see \cite{SaSm}). 

\medskip

Finally, we present the description of the total cotangent bundle $T^*\mathcal{T}^n(S)$. 

\begin{theorem}[Theorem 3 in \cite{FT}]
    The cotangent bundle $T^*\mathcal{T}^n(S)$ is a $\mathrm{Ham}_0(T^*S)$-equivalence class of tensors $\mu_\ell\in \Gamma(K^{1-\ell}\otimes\bar{K})$ and $t_\ell\in\Gamma(K^\ell)$ for $\ell=2,...,n$ satisfying
    \begin{equation}\label{Eq:mu-holo-cond-hcs}
    -\bar{\partial}t_k\!+\!\mu_2\partial t_k\!+\!kt_k\partial\mu_2+\sum_{l=1}^{n-k}((l\!+\!k)t_{k+l}\partial\mu_{l+2}+(l\!+\!1)\mu_{l+2}\partial t_{k+l})=0.
    \end{equation}
    We refer to this condition as the \emph{$\mu$-holomorphicity condition}.
\end{theorem}

The tensors $\mu_\ell$ are the higher Beltrami differentials from \eqref{Eq:hcs-ideal}. The tensors $t_\ell$ describe the covector.
For a trivial higher complex structure, that is, when $\mu_k=0$ for all $k\in\{2,...,n\}$, the condition simply reduces to $\bar\partial t_k=0$. Thus, $\mu$-holomorphicity can be seen as \emph{a generalization of the usual holomorphicity condition}.

\subsection{Nonabelian Hodge theory and twistor approach}\label{Sec:nonabelian-Hodge}

Building on the fundamental theorems by Narasimhan--Seshadri and Eells--Sampson, nonabelian Hodge theory provides an abundance of methods that can be used for the study of character varieties via holomorphic techniques and Higgs bundles, as well as analytic techniques from the theory of harmonic maps. We provide a brief overview of this holomorphic viewpoint and the twistor space framework to character varieties.

\subsubsection{The nonabelian Hodge correspondence}

Let $X$ be a compact K\"{a}hler manifold and let $G$ be a connected complex reductive Lie group. For $H \subseteq G$ a maximal compact subgroup of $G$, its complexification $H^{\mathbb{C}}$ is isomorphic to $G$, and the Cartan decomposition for the corresponding Lie algebras $\mathfrak{g}:=\mathrm{Lie}(G)$ and $\mathfrak{h}:=\mathrm{Lie}(H)$ reads in this case as
\[\mathfrak{g}=\mathfrak{h} \oplus i\mathfrak{h}.\]
A flat principal $G$-bundle $(P, \Theta)$ over $X$ is equivalent to a reductive representation $\rho: \pi_1(X) \to G$ via the \emph{Riemann--Hilbert correspondence}, for $\Theta$ a 1-form on $P$ with values in $\mathfrak{g}$ (a principal connection on $P$). The flatness condition means that $\Theta$ satisfies the equation 
\[d\Theta +\tfrac{1}{2}[\Theta \wedge \Theta]=0.\] 

Since $\rho$ is assumed to be reductive, Corlette's Theorem \cite{Corlette} provides the existence of a $\rho$-equivariant harmonic map from the universal cover $\widetilde{X}$ of $X$ to the associated symmetric space,
\[f: \widetilde{X}\to G/H.\]
An equivariant map from $\widetilde{X}$ to $G/H$ is the same as a reduction of structure group of the principal $G$-bundle $P$ to the maximal compact $H$. This means there is an $H$-bundle $P_H$, and an $H$-equivariant map $\iota :P_H\to P$, and the flat connection ${{\iota }^{*}}\Theta $ on $P_{H}$ now splits as
\[{{\iota }^{*}}\Theta =A+\psi,\]
where $A$ is a connection on $P_{H}$, and $\psi$ descends to a 1-form on $X$ with values in the bundle associated to $P_{H}$ via the isotropy representation $\mathrm{Ad}:H\to \mathrm{GL}(i\mathfrak{h})$. The flatness of the connection ${{\iota }^{*}}\Theta $ implies the two equations 
\begin{align*}
    {{F}_{A}}+\tfrac{1}{2}\left[ \psi \wedge\psi  \right]&=0\\
    {{d}_{A}}\psi &=0,
\end{align*}
and harmonicity of the map $f$ gives the additional equation
\[d_{A}^{*}\psi =0.\]
In the case when $G=\mathrm{SL}_n(\mathbb{C})$, a reduction to a maximal compact is the same as a unit volume hermitian metric on the associated vector bundle. In the light of this equivalence, we call the special metric provided by Corlette's Theorem, a \emph{harmonic hermitian metric}.\\
An application of the Siu--Sampson Theorem \cite[Theorem 1]{Sam} now gives that for a reduction of structure group as above, then the $(0,1)$-part ${{\bar{\partial }}_{{{P}_{H}}}}$ of the connection $A$ and the $(1,0)$-part $\varphi$ of the 1-form $\psi$ satisfy the equations
$$\bar{\partial }^2_{P_H}\varphi=0\;, \;\; \bar{\partial }_{P_H}\varphi =0\;, \;\;  [\varphi \wedge\varphi ]=0.$$ The $(0,1)$-form ${{\bar{\partial }}_{{{P}_{H}}}}$ defines a holomorphic structure on the ${{C}^{\infty }}$-principal bundle $P_{H}$; extending the structure group from $H$ to ${{H}^{\mathbb{C}}}$ we finally have:
\begin{definition}
For a compact K\"{a}hler manifold $X$ and a connected complex reductive Lie group $G$, a \emph{$G$-Higgs bundle}  over $X$ is a pair $(P,\varphi)$, where
\begin{itemize}
\item $P$ is a holomorphic principal $G$-bundle over $X$, and 
\item $\varphi \in {{\Omega}^{1,0}}\left( X,\mf{g}_P \right)$ satisfying ${{\bar{\partial }}_{{{P}_{H}}}}\varphi =0$ and $[\varphi \wedge\varphi ]=0$. 
\end{itemize}
Equivalently, we can be thinking of the Higgs field $\varphi$ as a holomorphic section $\varphi \in \mathrm{H}^0\left( X,\g_P \otimes K \right)$, where $\g_P$ denotes the adjoint bundle of $P$. 
\end{definition}
\begin{remark}
 In the case when $X$ is a compact Riemann surface, the two conditions  $\bar{\partial }^2_{P_H}\varphi=0$ and $[\varphi \wedge\varphi ]=0$ are automatically satisfied.  
\end{remark}

\begin{remark} 
When $G \subset \mathrm{GL}_n( \mathbb{C})$, a $G$-Higgs bundle can be naturally interpreted as a Higgs bundle $(E, \Phi)$ in the original sense of Hitchin \cite{Hit87}, \cite{Simpson} together with some additional structure reflecting the structure of the group $G$. In particular, when $G=\mathrm{SL}_n(\mathbb{C})$, an $\mathrm{SL}_n(\mathbb{C})$-Higgs bundle is described by a pair $(E,\Phi)$, where $E$ is a holomorphic rank $n$ vector bundle with trivial determinant and the Higgs field $\Phi$ is a holomorphic section $\Phi \in \mathrm{H}^0\left( X,\mathrm{End}(E)\otimes K \right)$ with $\mathrm{tr}(\Phi)=0$.
\end{remark}

The flatness of the connection $\Theta$ finally gives the so-called \emph{Hitchin equation} for the $G$-Higgs bundle $\left( P,\varphi  \right)$ constructed above,
\begin{equation}\label{G_Hitchin_eq}
F(A) + [\varphi\wedge \varphi^*]=0,
\end{equation}
where $\varphi^*$ denotes the hermitian adjoint coming from the Cartan involution of $\mathfrak{g}$ extended to $\g_P\otimes \Omega^1(S)$. 

So far we have described the passage from a reductive representation to a Higgs bundle; the opposite direction is provided by Simpson's Theorem \cite{Simpson} (see also \cite{GGM}), which says that for a $G$-Higgs bundle $\left( P,\varphi  \right)$ with vanishing Chern classes, there exists a reduction of structure group of $P$ such that Equation (\ref{G_Hitchin_eq}) holds if and only if $\left( P,\varphi  \right)$ is polystable.  The polystability condition is a condition appropriately extending Mumford's stability condition for vector bundles, which is implemented in order to construct the \emph{moduli space of $G$-Higgs bundles}, or Dolbeault moduli space, $\mathcal{M}_{Dol} (X,G)$ as a GIT quotient. The nonabelian Hodge correspondence for a K\"{a}hler manifold $X$ with underlying topological manifold $M$, therefore, is describing a bijection between this Dolbeault moduli space and the character variety, or Betti moduli space, $\mathcal{M}_B (M,G)$ of reductive fundamental group representations:

\begin{theorem}[Nonabelian Hodge correspondence]\label{naHc}
There is a real-analytic isomorphism $\mathcal{M}_{Dol} (X,G) \cong \mathcal{M}_B (M,G)$.
\end{theorem}

\subsubsection{The twistor approach}

We now restrict attention to a compact Riemann surface $X$ with underlying smooth surface $S$.
For a polystable Higgs bundle $\left( {{{\bar{\partial }}}_{{{E}_{H}}}},\varphi  \right)$ over $X$, Hitchin \cite{Hit87} interpreted Equation (\ref{G_Hitchin_eq}) together with the condition of holomorphicity for the Higgs field $\varphi$, $ {{{\bar{\partial }}}_{{{E}_{H}}}}(\varphi)=0$, in terms of a set of three moment maps for the action of the unitary gauge group. Following symplectic reduction techniques, the moduli space $\mathcal{M}_{Hit}$ of solutions to this set of equations was constructed as a \emph{hyperk\"{a}hler quotient}. This means that $\mathcal{M}_{Hit}$ is a $4n$-dimensional Riemannian manifold equipped with three covariant constant (with respect to the Levi-Civita connection) orthogonal automorphisms $I, J$ and $K$ of the tangent bundle $T\mathcal{M}_{Hit}$ which satisfy the quaternionic identities 
\[   {{I}^{2}}={{J}^{2}}={{K}^{2}}=IJK=-1.\]
In fact, any linear combination of the form 
\begin{equation}\label{cx_str_hyperk}
   i(a,b,c)=aI+bJ+cK 
\end{equation}
satisfies $i^2=-\mathrm{id}$ if and only if $a^2+b^2+c^2 = 1$. Taking $\lambda = (a,b,c)\in \mathbb{CP}^1$, we have a 1-parameter family of complex structures on $\mathcal{M}_{Hit}$. Then, using the Riemannian metric on $\mathcal{M}_{Hit}$, one gets a family of symplectic structures which, combined with the complex structures above, give a 1-parameter family of K\"{a}hler structures. 

In the light of \cite{HKLR87}, the \emph{twistor approach} allows one to incorporate all complex structures in the 1-parameter family described above into a single complex structure on a larger manifold, the \emph{twistor space of} $\mathcal{M}_{Hit}$. As a smooth manifold, this is defined as the product manifold
\[Z=\mathcal{M}_{Hit} \times \mathbb{CP}^1.\]
One can equip $Z$ with a complex structure as follows: at a point $(m,\lambda)$ of $Z$ it is given by the pair $(i_{\lambda}, i_0)$, where $i_0$ denotes the standard complex structure on $\mathbb{CP}^1$ and $i_{\lambda}$ is as in Equation (\ref{cx_str_hyperk}), for $\lambda = (a,b,c)\in \mathbb{CP}^1$. This defines an almost-complex structure and it is not hard to prove the integrability. The projection $p:Z\to \mathbb{CP}^1$ is holomorphic and each copy $(m, \mathbb{CP}^1)$ of the projective line is a holomorphic section of this projection, called a \emph{twistor line}. 

Using ideas of Deligne, Simpson \cite{Sim97} constructed the twistor space for $\mathcal{M}_{Hit}$ as the moduli space of $\lambda$-connections which he called \emph{Hodge moduli space} $\mathcal{M}_{Hod}$. In this broader picture, the Betti moduli space  $\mathcal{M}_B (S,G)$ and the Dolbeault moduli space $\mathcal{M}_{Dol} (X,G)$ are two special fibers of the holomorphic fiber bundle $\mathcal{M}_{Hod}\to \mathbb{CP}^1$. One possible way to describe Deligne's $\l$-connections is via what we call three-term connections, which is discussed next.

\subsubsection{Three-term connections}\label{sec:examples of 3-term conn}

For a complex reductive Lie group $G$, the works of Hitchin \cite{Hit87} and Simpson \cite{SimI}, \cite{SimII}, \cite{SimVman} provide that the cotangent bundle of the space of holomorphic $G$-bundles on a given Riemann surface $X$ can be mapped isomorphically to the space of families of connections of the form 
\[\mathcal{A}(\lambda)=\lambda^{-1}\Phi +d_A + \lambda \Phi^{*},\]
where $\lambda \in \mathbb{C}^{*}$ is a parameter, $\Phi$ a holomorphic Higgs field, $\Phi^{*}$ the hermitian conjugate with respect to a harmonic metric $h$, and $d_A$ the associated Chern connection. In the Higgs bundle setting, this is a family of \emph{flat connections}, the $(0,1)$-part of the background connection $d_A$ defines a holomorphic structure on the bundle and the Higgs field $\Phi$ defines a cotangent vector to the space of holomorphic bundles. The family $\mathcal{A}(\lambda)$ gives a family of maps from the moduli space of flat $G$-connections to itself depending on the parameter $\lambda \in \mathbb{C}^{*}$. In the limit $\lambda \to 0$ (resp. $\lambda\to \infty$) these structures tend to the moduli space of $G$-Higgs bundles on $X$ (resp. the complex conjugate $\bar{X}$) with its K\"{a}hler structure. 

Another important example of a 1-parameter family of flat connections depending on a parameter $\lambda \in \mathbb{C}^{*}$ analogous to the Hitchin family of flat connections described above was given by Fock in \cite{Fock}. These flat connections are determined by solutions of the cosh-Gordon equation and describe a candidate for the twistor space of almost-Fuchsian representations. 
In the approach of \cite{Fock}, the complex structure on the surface is a function of a background connection determined as above, and is not fixed once for all connections in the family as in Hitchin's case.

\subsection{Principal nilpotent elements}\label{Sec:princ-nilp-elements}

Let $\g$ be a complex simple Lie algebra and $G$ be a Lie group with Lie algebra $\g$. For $x\in\g$ we denote by $Z(x)=\{y\in \g\mid [x,y]=0\}$ its centralizer in $\g$. We also denote by $Z_G(x)$ the centralizer of $x$ in $G$ and by $Z(G)$ the center of $G$. 

\begin{definition}
An element of $\g$ is called \emph{regular} if the dimension of its centralizer equals the rank of the Lie algebra. A regular nilpotent element is called \emph{principal nilpotent}.
\end{definition}

A more general statement holds about the dimension of the centralizer: for any $x \in \g$, we have $\dim Z(x) \geq \rk(\g)$ (see for example Lemma 2.1.15. in \cite{Collingwood}). So the regular elements minimize this dimension.
For $\g=\mathfrak{sl}_n(\C)$, a nilpotent element is principal nilpotent iff it has maximal rank, i.e. it is of rank $n-1$.

\begin{theorem}\cite[Corollary 5.5.]{Kost}\label{Thm:one-reg-orbit}
All principal nilpotent elements are conjugate under the adjoint action of the Lie group $G$.
\end{theorem}

For a principal nilpotent element $F$, its centralizer has properties quite analogous to a Cartan subalgebra (the centralizer of a regular semisimple element):

\begin{theorem}\label{thmKost}
For $F$ a principal nilpotent element, its centralizer $Z(F)$ is abelian and nilpotent.
\end{theorem}
Using a limit argument one can show even more: for any element $x\in \g$, there is an abelian subalgebra of $Z(x)$ of dimension $\rk (\g)$; see \cite[Theorem 5.7]{Kost}.
The nilpotency of $Z(f)$ can be found in \cite[ Corollary in Section 3.7]{Steinberg}. 

The \emph{Jacobson--Morozov lemma} states that any non-zero nilpotent element $F\in\g$ can be included into an $\mathfrak{sl}_2$-subalgebra, the image of an injective homomorphism from $\mathfrak{sl}_2(\C)$ into $\g$. 
An $\mathfrak{sl}_2$-subalgebra is called \emph{principal} if it contains a principal nilpotent element. It follows from Theorem \ref{Thm:one-reg-orbit} that all principal $\mathfrak{sl}_2$-subalgebras are conjugate. Given any $\mathfrak{sl}_2$-subalgebra, the Lie algebra $\g$ splits into irreducible $\mathfrak{sl}_2$-modules. For a principal one, none of these modules is trivial and exactly one module is of dimension 3 (the principle $\mathfrak{sl}_2$-subalgebra itself).

\begin{proposition}\label{Prop:center-in-image}
    For a principal nilpotent element $F\in \g$, we have $Z(F)\subset \mathrm{Im}(\ad_F)$.
\end{proposition}
To see this, include $F$ into a principal $\mathfrak{sl}_2$-subalgebra, decompose $\g$ into irreducible $\mathfrak{sl}_2$-modules and use $\mathfrak{sl}_2$-representation theory. The important point is that there is no trivial $\mathfrak{sl}_2$-module.

\begin{proposition}\label{Prop:space-of-sl2-f}
The $\mf{sl}_2$-subalgebras containing a given principal nilpotent $F$ are acted on transitively by $Z_G(F)\cong \exp(Z(F))\times Z(G)$ with stabilizer $\exp(F)\times Z(G)$. In particular, the set of $\mf{sl}_2$-subalgebras containing $F$ is a torsor for the quotient group which is isomorphic to $\C^{\mathrm{rk}(\mf{g}) -1}$.
In particular, this space is contractible. 
\end{proposition}

The structure of the centralizer $Z_G(F)$ is described in Lemma 3.7.3 in \cite{Collingwood}. It states that for a non-zero nilpotent element $x$ the centralizer $Z_G(x)$ is a semidirect product between $\exp(Z(x)\cap \mathrm{Im}(\ad_x))$, which is $\exp(Z(F))$ for $x=F$ by Proposition \ref{Prop:center-in-image}, and the centralizer of an $\mathfrak{sl}_2$-subalgebra containing $x$, which in the principal case is $Z(G)$.

A basis $(F,H,E)$ of an $\mathfrak{sl}_2$-subalgebra, satisfying the standard relations $[H,E]=2E$, $[H,F]=-2F$ and $[E,F]=H$, is called an \emph{$\mathfrak{sl}_2$-triple}. Fix a principal $\mathfrak{sl}_2$-triple $(F,H,E)$. To this triple, Hitchin \cite{Hit92} associates a Lie algebra involution $\sigma_0:\mf{g}\to \mf{g}$ as follows. Using the decomposition of $\g$ into irreducible $\mathfrak{sl}_2$-modules, the involution $\sigma_0$ is uniquely determined by negating all highest and lowest weight vectors \cite[Proposition 6.1]{Hit92}. He also defines a compact real form $\rho_0:\mf{g}\to \mf{g}$ which extends 
\[ E \mapsto -F, \;\;\;\; H \mapsto -H, \;\;\;\; F, \mapsto -E\]
and commutes with $\sigma_0$. Finally, he shows that $\tau_0 := \sigma_0\rho_0$ is a split real structure. In fact, $\sigma_0$ can be defined up to conjugation by inner automorphisms as the product of commuting split and compact real forms. 

It will be useful to understand the set of involutions conjugate to $\sigma_0$ which negate a given principal nilpotent element. 
\begin{lemma} \label{lem:space_of_sigma}
    The collection of involutions $\sigma$ conjugate to $\sigma_0$ which negate a given principal nilpotent $F$ is acted on simply transitively by $\exp(Z(F))$. In particular, this space is contractible.
\end{lemma}
\begin{proof}
    Let $\sigma,\sigma'$ be two such involutions which negate $F$. They differ by an inner automorphism: $\sigma' = \sigma\circ \Ad_\gamma$ where $\gamma\in G$. We see that $\Ad_\gamma F = \sigma\sigma'\cdot F = F$, so $\gamma\in Z_G(F)$. Since $Z_G(F) \cong \exp(Z(F))\times Z(G)$, we see that $\exp(Z(F))$ must act transitively. From this fact, we deduce Corollary \ref{cor:sigma_negate_centralizer} below, stating that $\sigma$ negates the whole centralizer $Z(F)$. Hence it inverts $\exp(Z(F))$.
    We have to check that if $z\in \exp(Z(F))$ is different from the identity, then $\sigma\circ \Ad_z \neq \sigma$. For $X\in \mf{g}$ we have:
    \[\sigma(\Ad_z X) = \Ad_{\sigma(z)} \sigma(X) = \Ad_{z^{-1}} \sigma(X).\]
    If this was always equal to $\sigma(X)$, then $z$ would be central, which contradicts the fact that it is the exponential of a nilpotent element.
\end{proof}
\begin{corollary}\label{cor:sigma_negate_centralizer} Any $\sigma$ conjugate to $\sigma_0$ which negates $F$, actually negates all of $Z(F)$.
\end{corollary}
\begin{proof}This is true for the involution $\sigma$ of Hitchin's construction, and this property is unchanged under inner automorphisms by elements in $\exp(Z(F))$. 
\end{proof}

\begin{lemma}\label{lemma:unique-sigma-sl2}
        Let $\sigma$ be an involution conjugate to $\sigma_0$ and let $F\in\g$ be principal nilpotent with $\sigma(F)=-F$. Then there is a unique $\sigma$-invariant $\mathfrak{sl}_2$-subalgebra containing $F$.
\end{lemma}
\begin{proof}
    Proposition \ref{Prop:space-of-sl2-f} shows that $\exp(Z(F))$ acts transitively on the space of $\mathfrak{sl}_2$-subalgebras containing $F$. By Corollary \ref{cor:sigma_negate_centralizer}, we know that $\sigma$ inverts $\exp(Z(F))$. The involution $\sigma$ acts on the set of principal $\mf{sl}_2$-subalgebras containing $F$, intertwining the action of $\exp(Z(F))$ with its negative. More explicitly, let $\mf{p}\subset \mf{g}$ be a principal $\mf{sl}_2$-subalgebra containing $F$, and let $z\in \exp(Z(F))$. Then: 
    \[\sigma(\Ad_{z}\mf{p})=\Ad_{z^{-1}}\sigma(\mf{p}).\]
    It will follow from this property that there is a unique fixed point of the action by $\sigma$, which we can construct as a kind of midpoint. Let $z$ be the unique element of $\exp(Z(F))$ such that $\Ad_{z^2}(\mf{p}) = \sigma(\mf{p})$. Moreover,
    \[\sigma(\Ad_{z}(\mf{p})) = \Ad_{z^{-1}}(\sigma(\mf{p})) = \Ad_{z^{-1}}(\Ad_{z^2}(\mf{p})) = \Ad_{z}(\mf{p}).\]
    Therefore, $\Ad_{z}(\mf{p})$ is an $\mf{sl}_2$-subalgebra fixed by $\sigma$. If there were two fixed subalgebras $\mf{p},\mf{p}'$, then $\mf{p}' = \Ad_z \mf{p}$ for some $z$. Such $z$ satisfies $\sigma(z)=z$, thus is the identity.
\end{proof}

Finally, we will need the following lemma:
\begin{lemma}\label{lemma:im-incl}
    For a principal nilpotent element $F\in\g$ and an element $F'\in Z(F)$, we have $\mathrm{Im}(\ad_{F'})\subset\mathrm{Im}(\ad_{F})$.
\end{lemma}
\begin{proof}
    By the Jacobson--Morozov lemma we can complete $F$ into a principal $\mathfrak{sl}_2$-triple $(F,H,E)$. Choose a compact real form $\rho$ such that $E=-\rho(F)$. Let $E'=-\rho(F')$. Since $F'\in Z(F)$ we get $E'\in Z(E)$. The commutator $Z(E)$ is abelian, hence we have $Z(E)\subset Z(E')$. With respect to the hermitian inner product $\tr(\rho(.).)$ on $\g$, $E$ and $F$ are adjoint, and so are $E'$ and $F'$. This implies that $\mathrm{Im}(\ad_F)$ is the perp of $\ker(\ad_E)=Z(E)$, and $\mathrm{Im}(\ad_F')$ is the perp of $\ker(\ad_{E'})=Z(E')$. This concludes the proof since $Z(E)\subset Z(E')$.
    \end{proof}

\subsection{Involutions and reductions of structure group}\label{Sec:involutions}

There are three basic structures often put on a complex vector bundle: a hermitian structure, a symmetric pairing, and a real structure. In this section we explain precisely how to generalize these notions to principal $G$-bundles for connected complex simple $G$. 

Fix an antiholomorphic involution $\rho_0:G\to G$ whose fixed point locus is a maximal compact subgroup, and another commuting antiholomorphic involution $\tau_0$ whose fixed point locus is a split real subgroup. Call their composition $\sigma_0 := \rho_0\tau_0$. These involutions induce, (and are determined by) anti-linear involutions of the lie algebra $\mf{g}$ which we will refer to by the same symbols. 

\begin{remark} The involution $\sigma_0$ can equivalently be obtained by Hitchin's construction using a principal $\mathfrak{sl}_2$-triple \cite[Proposition 6.1]{Hit92}. 
\end{remark}

\begin{definition}
Let $G$ be a group, let $\epsilon_0:G\to G$ be an involution, and let $P$ be a principal $G$-bundle on a manifold $M$. An \emph{$\epsilon_0$-structure} on $P$ is an involution $\epsilon: P\to P$ such that
\[\epsilon(p.g)=p.\epsilon_0(g),\;\;\;\text{for any }\, p\in P, g\in G.\]
\end{definition}
An $\epsilon_0$-structure on a $G$-bundle $P\to M$ is equivalent to a reduction of structure group of $P$ from $G$ to the group of fixed points $G^{\epsilon_0}$. Indeed, the fixed point locus of $\epsilon$ is naturally a $G^{\epsilon_0}$-bundle. 

The special case of $\rho_0$-structures, which we will refer to as a \emph{hermitian structure}, can be understood more concretely. A hermitian structure $\rho$ on a $G$-bundle $P$ induces an involution (which we also call $\rho$) on the adjoint bundle $\mf{g}_P$. The eigenspaces of $\rho$ give a fiberwise Cartan decomposition of $\mf{g}_P$. Since $G^{\rho_0}$ is precisely the subgroup of $G$ which preserves a Cartan decomposition, the reduction to $G^{\rho_0}$ is fully specified by this involution of $\mf{g}_P$.

In the case of $\sigma_0$- and $\tau_0$-structures, we also get involutions $\sigma$ and $\tau$ of $\mf{g}_P$, but there is slightly more data in the structure. This is because the subgroup of $G$ commuting with $\tau_0$ is not just $G^{\tau_0}$, but also the center of $G$. So just specifying an involution $\tau$ of $\mf{g}_P$, which is conjugate to $\tau_0$ in each fiber, only gives a reduction to a slightly larger group. The same is true for $\sigma$. In the case $G=G_{ad}$ where the center is trivial, this is not an issue, and we may think of $\tau_0$- and $\sigma_0$-structures purely as involutions of the adjoint bundle.

\section{Fock bundles}\label{Sec:Fock-bundles}

In this section, we introduce Fock bundles, analyze the so-called Fuchsian locus and describe the variation in the Fock field.

\subsection{Fock bundles and higher complex structures}

Fix a smooth closed orientable surface $S$ with genus at least 2. Consider a complex simple Lie group $G$ with associated Lie algebra $\mathfrak{g}$. Throughout, we will fix commuting involutions $\rho_0, \tau_0$ and $ \sigma_0$ of $G$ such that $\rho_0$ is a compact real form, $\tau_0$ is a split real form, and $\sigma_0 = \rho_0\tau_0$. 
For a principal $G$-bundle $P$, denote by $\g_P$ the associated $\g$-bundle using the adjoint action of $G$ on $\g$. 

Recall the notions of a principal nilpotent element and a $\sigma_0$-structure from the previous Subsections \ref{Sec:princ-nilp-elements} and \ref{Sec:involutions}. The main notion we introduce in this article is the following:

\begin{definition}\label{defn_Fock-bundle}
    A \emph{$G$-Fock bundle} over $S$ is a triple $(P,\Phi,\sigma)$, where $P$ is a principal $G$-bundle over $S$, $\sigma$ is a $\sigma_0$-structure, and $\Phi\in \Omega^1(S,\mathfrak{g}_P)$ is a $\mathfrak{g}_P$-valued 1-form satisfying
    \begin{enumerate}
        \item $[\Phi\wedge\Phi] = 0$,
        \item $\Phi(v)(z)$ is principal nilpotent for all $z\in S$ and all non-zero vectors $v\in T_zS$.
        \item $\sigma(\Phi) = -\Phi$.
    \end{enumerate}
\noindent We shall call a field $\Phi$ defined as above a \emph{Fock field} and will often refer to the second condition above as the \emph{nilpotency condition}. An isomorphism of Fock bundles $(P,\Phi,\sigma) \to (P',\Phi',\sigma')$ is a $G$-bundle isomorphism $P \to P'$ taking $\Phi$ to $\Phi'$ and $\sigma$ to $\sigma'$. 
\end{definition}

Note that all conditions in Definition \ref{defn_Fock-bundle} are given pointwise. We stress that no term in the definition of a Fock bundle is considered to be holomorphic. 

There are two main cases one should highlight:
\begin{itemize}
    \item An $\SL_n(\C)$-Fock bundle is a vector bundle $E$ of rank $n$ with fixed volume form, equipped with a symmetric pairing $g$ (a complex bilinear non-degenerate symmetric form) and a Fock field $\Phi$ satisfying the three conditions above.
    \item For $G=G_{ad}$ the adjoint group, a $G$-Fock bundle is specified up to isomorphism by the pair $(P,\Phi)$. 
\end{itemize}

The data contained in $\sigma$ is actually very little: We will see later in Proposition \ref{Prop:sigma-dependence} that a $\sigma_0$-structure $\sigma$ which negates $\Phi$ always exists locally, and always exists globally if $G=G_{ad}$ is the adjoint group. For any $G$, any two choices of $\sigma$ are conjugate by an automorphism fixing $\Phi$, thus give isomorphic Fock bundles. That is why we will sometimes write $(P,\Phi)$ for a $G$-Fock bundle.

By \cite[Proposition 2.21]{Tho22}, we know that there is a unique complex line $L\subset T^\C_zS$ such that for all $v\in L$ the matrix $\Phi(v)(z)$ is not principal nilpotent. Note that the uniqueness needs the Lie algebra $\g$ to be simple. The nilpotency condition in Definition \ref{defn_Fock-bundle} then implies that this direction $L$ is avoiding the real locus and hence encodes a complex structure on $S$ (see Section \ref{Sec:HCS}).

\begin{proposition}
    A $G$-Fock bundle induces a complex structure on $S$.
\end{proposition} 

In the sequel, unless stated explicitly otherwise, whenever we work with complex local coordinates on $S$, we use the complex structure induced by the Fock bundle. For such a complex coordinate $z$ on $S$, we can locally write $\Phi=\Phi_1 dz+\Phi_2 d\bar{z}$. The condition $[\Phi\wedge\Phi]=0$ for a Fock field can be then written as $[\Phi_1,\Phi_2]=0$. By construction of the complex structure on $S$, we know that $\Phi_2$ is not principal nilpotent. Hence $\Phi_1$ has to be principal nilpotent. 

\medskip
We now describe more explicitly what $\mathrm{SL}_n(\mathbb{C})$-Fock bundles look like locally. We start with $n=2$. Let $E$ be a complex vector bundle of degree zero and rank $2$ over $S$ with a fixed volume form $\nu$. 

\begin{lemma}
Fix an arbitrary complex coordinate on $S$. An $\mathrm{SL}_2(\C)$-Fock field on $E$ locally is of the form $$\Phi=\begin{pmatrix} 0&0\\1&0\end{pmatrix}dz+\begin{pmatrix} 0&0\\\mu_2&0\end{pmatrix}d\bar{z},$$
where $\mu_2$ is a local complex function on $S$ with $\mu_2\bar{\mu}_2\neq 1$. 
\end{lemma}

Note that globally, $\mu_2$ is the Beltrami differential of the complex structure on $S$ induced by the Fock bundle.
\begin{proof}
    There is a local trivialization of $E$ in which $\Phi_1$ is given as in the statement of the lemma since all principal nilpotent elements are conjugate. The condition $[\Phi_1,\Phi_2]=0$ then implies the form of $\Phi_2$. Finally, the nilpotency condition implies that for any non-zero real tangent vector $v\partial+\bar{v}\bar\partial$ (where $v$ is a complex function) the combination $v\Phi_1+\bar{v}\Phi_2$ is principal nilpotent. This means that $v+\mu_2 \bar{v}\neq 0$ for all $v\neq 0$. This is equivalent to the condition $\lvert\mu_2\rvert\neq 1$.
\end{proof}

We can generalize this local description to higher rank. Let $E$ be a complex vector bundle of degree zero and rank $n$ over $S$ with a fixed volume form $\nu$. 

\begin{lemma}\label{lemma:sln-fock-field}
Fix an arbitrary complex coordinate on $S$. An $\mathrm{SL}_n(\C)$-Fock field on $E$ locally is of the form $\Phi=\Phi_1dz+\Phi_2d\bar{z}$ with 
$$\Phi_1=\sum_{i=1}^{n-1}E_{i+1,i}\;\text{ and }\; \Phi_2=\mu_2 \Phi_1+\mu_3\Phi_1^2+...+\mu_n \Phi_1^{n-1},$$ 
where the $\mu_k$ are local complex functions on $S$ with $\mu_2\bar{\mu}_2\neq 1$.
\end{lemma}

The proof is similar to the one given above; there is only one conjugacy class of principal nilpotent elements which explains the form  of $\Phi_1$. The centralizer of $\Phi_1$ is the set of polynomials in $\Phi_1$ which provides the form of $\Phi_2$. In the complex structure induced by the Fock field, we always have $\mu_2=0$.
The case when all higher Beltrami differentials vanish is the so-called \emph{Fuchsian locus} which will be analysed in Section \ref{Sec:fuchsian-locus} below.

Comparing Definition \ref{defn_Fock-bundle} of a $G$-Fock bundle to Definition \ref{def-g-C-str} of a $\mathfrak{g}$-complex structure, we immediately see the following: 
\begin{proposition}\label{Prop:link-to-g-C-str}
    Any $G$-Fock bundle induces a $\g$-complex structure on $S$. For the adjoint group $G_{ad}$, the isomorphism class of a $G_{ad}$-Fock bundle is equivalent to a $\g$-complex structure.
\end{proposition}
The second assertion follows directly from Corollary \ref{Prop:P-topology} below.
For $G=\mathrm{SL}_n(\C)$, we get a direct link between Fock bundles and the ideals describing higher complex structures of order $n$:
\begin{proposition}\label{link_hcs}
Let $(E, \Phi)$ be an $\mathrm{SL}_n(\C)$-Fock bundle over a surface $S$. The map 
    $$p:\left \{ \begin{array}{ccc}
    \mathrm{Sym}(T^{\mathbb{C}}S) & \longrightarrow &\mathrm{End}(E) \\
    v_1\cdots v_k & \mapsto &\Phi(v_1)\cdots \Phi(v_k)
    \end{array} \right.$$
is well-defined and the kernel of $p$ defines a higher complex structure.
\end{proposition}

\begin{proof}
To prove that $p$ is well-defined, one has to show that the expression $\Phi(v_1)\cdots \Phi(v_k)$ remains unchanged under permutation of $(v_1,...,v_k)$. This follows from $\Phi\wedge \Phi=0$.
The matrix viewpoint of higher complex structures analyzed in \cite[Section 4.2]{Tho22} implies that the kernel of $p$ is a higher complex structure.
\end{proof}

\begin{proposition}\label{Prop:no-automorphisms}
     A $G$-Fock bundle $(P,\Phi,\sigma)$ has no infinitesimal automorphisms. Thus, all Fock bundles are stable in this sense.
\end{proposition}

\begin{proof}
    Consider an infinitesimal gauge transformation $\eta\in\Omega^0(S,\g_P)$. In order to preserve $\Phi$, we need $\eta\in Z(\Phi)$. Since $Z(\Phi)$ is negated by $\sigma$ (see Corollary \ref{cor:sigma_negate_centralizer}), it follows that $\eta$ must be zero, because it is both $\sigma$-invariant and $\sigma$-anti-invariant.
\end{proof}

For $\SL_n(\C)$, the Fock field $\Phi$ induces a natural filtration $\mathcal{F}$ on the bundle $E$: 
\begin{proposition}\label{Prop:filtration}
    An $\SL_n(\C)$-Fock bundle has a natural increasing filtration given by $\mathcal{F}_k=\ker \Phi(v)^k$, where $v$ is a local non-vanishing vector field. The filtration is independent of the choice of $v$.
\end{proposition}
\begin{proof}The element $\Phi(v)$ is principal nilpotent. In general, the centralizer of a principal nilpotent is the subalgebra generated by it, thus if two principal
nilpotent elements commute, they define the same filtration. Independence follows from $\Phi\wedge\Phi=0$ and the fact that $Z(\Phi)$ is abelian.
\end{proof}

\subsection{Fuchsian Locus}\label{Sec:fuchsian-locus}

The Fuchsian locus in a Hitchin component, for an adjoint group, is the set of representations which factor through $\PSL_2(\R)$. Similarly, we will introduce the Fuchsian locus in the space of Fock bundles for an adjoint group as the subset induced from $\PSL_2(\C)$-Fock bundles. When $G$ has center, we call a Fock bundle Fuchsian if the associated Fock bundle for $G_{ad}$ is Fuchsian.
It will turn out that every Fock bundle with $\Phi_2 = 0$ is Fuchsian, so all Fock bundles are deformable to the Fuchsian locus. Every Fock bundle in the Fuchsian locus can be equipped with a holomorphic structure, making it into a Higgs bundle. This is the uniformizing Higgs bundle. 

\begin{proposition}
    Any $\SL_2(\C)$-Fock bundle is of the form $(E,\Phi,g)$ where $E = K^{1/2}\oplus K^{-1/2}$, where we use the complex structure on $S$ induced from the Fock bundle, $g$ is the apparent symmetric pairing in which these line subbundles are isotropic, and $$\Phi=\begin{pmatrix} 0&0\\1&0\end{pmatrix}$$
    where $1$ denotes the canonical 1-form valued in $\Hom(K^{1/2},K^{-1/2})\cong K^{-1}$.
\end{proposition}
\begin{proof}
    Let $(E,\Phi,g)$ be an $\SL_2(\C)$-Fock bundle. The symmetric pairing $g$, gives an isotropic decomposition into dual line bundles $E = L\oplus L^{-1}$. We get a decomposition
    \[\mf{sl}(E) = L^{-2} \oplus \underline{\C} \oplus L^{2}\]
    in which $\sigma:X\mapsto -X^{*_g}$ acts by $(-1,1,-1)$. The Fock field $\Phi$ must be valued in the $-1$ eigenspace $L^2 \oplus L^{-2}$. Since $\Phi$ is nilpotent, it is valued in only one of these line bundles at any given point, and since it is nowhere vanishing, it is valued in only one of the line bundles globally. Without loss of generality, suppose it is valued in $L^{-2}$. Since it is nowhere vanishing, $\Phi$ is an isomorphism from the holomorphic tangent bundle of $S$ to $L^{-2}$, so $L$ must be a square root of the canonical bundle. 
\end{proof}
As a corollary, any $\SL_2(\C)$-Fock bundle has a natural upgrade to a Higgs bundle for the induced complex structure, because $K^{1/2}\oplus K^{-1/2}$ is naturally a holomorphic vector bundle. We see that an $\SL_2(\C)$-Fock bundle is equivalent to a complex structure together with a spin structure. Similarly, a $\PSL_2(\C)$-Fock bundle is equivalent to a complex structure.
\begin{proposition}\label{Prop:psl2-fock-bundle}
    Any $\PSL_2(\C)$-Fock bundle is of the form $(P,\Phi,\sigma)$ where $P = K_* \times_{\C^*} \PSL_2(\C)$
    with adjoint bundle  $\mf{sl}_2(\C)_P = K^{-1}\oplus \underline{\C} \oplus K$, $\Phi$ is the canonical $K^{-1}$-valued $1$-form, and $\sigma$ acts by $(-1,1,-1)$ on the adjoint bundle. Here, $K_*$ denotes the $\C^*$-bundle of nonzero covectors.
\end{proposition}
In particular, $P$ is topologically trivial because $K$ has even degree. Any $\PSL_2(\C)$-Fock bundle can be upgraded to an $\SL_2(\C)$-Fock bundle by choosing a square root of $K$.

Now restrict attention to an adjoint group $G_{ad}$. Then, the principal 3-dimensional subgroup is always $\PSL_2(\C)$. This is because the $\mf{sl}_2$-representations appearing in $\mf{g}$ are always odd dimensional (cf. Section \ref{Sec:interlude} for a more detailed account). Fix a principal embedding $\PSL_2(\C)\to G_{ad}$ such that the diagram
\[
\begin{tikzcd}
  \PSL_2(\C) \arrow[r] \arrow[d,"\sigma_0"]
    & G_{ad} \arrow[d, "\sigma_0"] \\
  \PSL_2(\C) \arrow[r]
& G_{ad} \end{tikzcd}
\]
commutes. 
\begin{definition}
    If $(P,\Phi,\sigma)$ is a $\PSL_2(\C)$-Fock bundle, then the \emph{induced $G_{ad}$-Fock bundle} is the triple $(P',\Phi',\sigma')$ where 
    \begin{itemize}
        \item $P'$ is the induced bundle $P\times_{\PSL_2(\C)} G_{ad}$, or alternatively $P^\sigma \times_{\PSL_2(\C)^{\sigma_0}} G_{ad}$.
        \item $\Phi'$ is the composition of $\Phi$ with the inclusion $\mf{sl}_2(\C)_{P}\to \mf{g}_{P'}$, and
        \item $\sigma'$ is $\sigma_0$ acting on the right factor $G_{ad}$ in the second description of $P'$.
    \end{itemize}
\end{definition}
The same definition will work for general $G$ with $\PSL_2(\C)$ 
replaced by $\SL_2(\C)$ when the principal 3 dimensional subgroup of $G$ is $\SL_2(\C)$.
\begin{definition}
    A $G_{ad}$-Fock bundle is in the \emph{Fuchsian locus} if it is induced from a $\PSL_2(\C)$-Fock bundle.
    For general $G$, a $G$-Fock bundle is in the Fuchsian locus if the associated $G_{ad}$-Fock bundle is in the Fuchsian locus.
\end{definition}

\begin{example}\label{Ex:Fuchsian-locus}
    The $\SL_n(\C)$-Fock bundle induced from an $\SL_2(\C)$-Fock bundle is the pair $(E,\Phi)$ given by 
    $$E=K^{(n-1)/2}\oplus K^{(n-3)/2}\oplus ...\oplus K^{(1-n)/2} \;\; \text{ and } \;\; \Phi=\sum_{i=1}^{n-1}E_{i+1,i},$$
    where $K$ denotes the canonical bundle on $S$, $K^{1/2}$ is a choice of a square-root and $E_{i+1,i}$ is the canonical identity 1-form valued in $\mathrm{Hom}(K^{(n-k)/2},K^{(n-k-2)/2}) = K^{-1}$. Again, if we view $E$ as a smooth vector bundle, this is a Fock bundle, but if we view it as a holomorphic bundle, it becomes the uniformizing $\mathrm{SL}_n(\C)$-Higgs bundle.
\end{example}

\begin{proposition}\label{prop:fock-fuchsian-locus}
A Fock bundle $(P,\Phi,\sigma)$ is in the Fuchsian locus if and only if $\Phi_2=0$.
\end{proposition}
\begin{proof}
    A Fock field in the Fuchsian locus clearly has no $(0,1)$-part. For the converse, by Lemma \ref{lemma:unique-sigma-sl2}, we know that there is a unique $\sigma$-invariant $\mathfrak{sl}_2$-subbundle $\mathfrak{p}$ of $\g_P$ containing the image of $\Phi$, because the image of $\Phi$ is a line bundle of principal nilpotents negated by $\sigma$. 

    The normalizer in $G_{ad}$ of a principal $\mathfrak{sl}_2$-subalgebra is just the principal $\PSL_2(\C)$ it generates. To see this, first recall that the centralizer of a principal $\mf{sl}_2$ is trivial, then note that anything in the normalizer can be multiplied by an element of $\PSL_2(\C)$ to be in the centralizer. The subbundle $\mf{p}$ thus gives a reduction of structure group of $G_{ad}$ to $\PSL_2(\C)$. Call this reduction of structure group $Q\subset P$. The induced map of adjoint bundles is simply the inclusion $\mf{p}\to \mf{g}_P$. The involution $\sigma$ on $\mf{g}_P$ restricts to an involution on $\mf{p}$, and in fact is the unique extension of this involution conjugate to $\sigma_0$. We see that $(P,\Phi,\sigma)$ is the induction from $\PSL_2(\C)$ to $G$ of the Fock bundle $(Q,\Phi,\sigma|_{Q})$. 
\end{proof}

\begin{corollary}\label{Prop:P-topology}
    Let $(P,\Phi,\sigma)$ be a $G$-Fock bundle. Then $P$ is topologically trivial.
\end{corollary}
\begin{proof}
    We can deform $\Phi$ continuously to get to the Fuchsian locus where $\Phi_2 = 0$. This does not alter the topology of $P$. We then consider the induced $G_{ad}$-bundle $P_{ad}$. If $P_{ad}$ is trivial, the same follows for $P$. We know that $P_{ad}$ is induced from a $\PSL_2(\mathbb{C})$-Fock bundle. We have seen in Proposition \ref{Prop:psl2-fock-bundle} that those are topologically trivial.
\end{proof}

As with the $\SL_n(\C)$-case, $G$-Fock bundles in the Fuchsian locus are exactly those obtained by uniformizing $G$-Higgs bundles by forgetting the holomorphic structure. For any $G$, there is a finite collection of $G$-Fock bundles which induce a given $G_{ad}$-Fock bundle.

\begin{proposition}
    Let $A$ be the subgroup of the center of $G$ which is fixed by $\sigma$. There is a simply transitive action of $\mathrm{H}^1(S,A)$ on the set of $G$-Fock bundles which induce a given $G_{ad}$-Fock bundle.
\end{proposition}

The proof is once more abstract bundle theory, which we leave to the reader. Instead, we describe the case of $\SL_n(\C)$. The center of $\SL_n(\C)$ is the $n$-th roots of unity and $\sigma$ acts by inversion, so for $n$ even, $A$ is $\{1,-1\}$, whereas for $n$ odd, $A$ is trivial. This means that for $n$ even, there are $2^{2g}$ $\SL_n(\C)$-Fock bundles for every $\mathrm{PSL}_n(\C)$-Fock bundle, whereas for $n$ odd, $\mathrm{SL}_n(\C)$- and $\mathrm{PSL}_n(\C)$-Fock bundles are the same.

\subsection{Variations of Fock bundles}\label{Sec:var-Fock-fields}

Simpson \cite{Simpson} studied the infinitesimal deformation space of stable Higgs bundles via a hypercohomology group for suitable chain complexes. In a similar way, we study the variations of $G$-Fock bundles $(P, \Phi,\sigma)$. We will first forget about $\sigma$ and study the variations of Fock fields, then see what happens when we introduce $\sigma$.

\subsubsection{Variation complex without $\sigma_0$-structure}

Consider the complex $\Omega^{\bullet}(S,\g_P)$ of smooth $\g$-valued differential forms, with differential given by the adjoint action of the Fock field $\Phi$:
\begin{equation}\label{phi-cohom}
\Omega^0(S,\g_P)\xrightarrow{\ad_\Phi\;} \Omega^1(S,\g_P)\xrightarrow{\ad_\Phi\;}\Omega^2(S,\g_P).
\end{equation}

\begin{proposition}
    The complex \eqref{phi-cohom} is a chain complex.
\end{proposition}
\begin{proof}
Indeed, using the Jacobi identity and $[\Phi\wedge \Phi]=0$, we get for all $A\in \Omega^0(S,\g_P)$:
$$[\Phi\wedge [\Phi,A]]=-[\Phi\wedge [\Phi,A]]+[A,[\Phi\wedge\Phi]]=-[\Phi\wedge [\Phi,A]],$$
which implies $\ad_\Phi^2(A)=0$.
\end{proof}

We call the cohomology groups defined for this chain complex, \emph{$\Phi$-cohomology groups},  and denote them by $\mathrm{H}^{\bullet}(\Phi)$. The zero-th cohomology group $\mathrm{H}^0(\Phi)$ describes the centralizer of $\Phi$. To be more precise, the centralizers of all $\Phi(v)(z)$ for a non-zero real vector $v\in T_zS$ are all equal. This follows from $[\Phi\wedge\Phi]=0$, the nilpotency condition and the general fact that the centralizer of a principal nilpotent element is abelian. This is why we often write $Z(\Phi)\subset \Omega^0(S,\g_P)$ for any of these centralizers. We have the following:
\begin{proposition}\label{Prop:dim-phi-cohom}
    The dimensions of $\mathrm{H}^0(\Phi), \mathrm{H}^1(\Phi)$ and $\mathrm{H}^2(\Phi)$ are respectively $\mathrm{rk}(\g)$, $2\mathrm{rk}(\g)$ and $\mathrm{rk}(\g)$.
\end{proposition}
\begin{proof}
    We have $\mathrm{H}^0(\Phi)=\ker\,\ad_\Phi=Z(\Phi)$, which by the nilpotency condition is of dimension $\mathrm{rk}(\g)$.
    The natural pairing between $a\in\ker\,\ad_\Phi\subset\Omega^0(S,\g_P)$ and $b\in\Omega^2(S,\g_P)$ given by $\int_S \tr (ab)$, where $\tr$ denotes the Killing form on $\g$, descends to cohomology. This follows from the cyclicity property $\tr([a,b]c)=\tr([b,c]a)$ of the Killing form. Hence $\mathrm{H}^2(\Phi)\cong \mathrm{H}^0(\Phi)^*$ which gives $\dim \mathrm{H}^2(\Phi) =\dim \mathrm{H}^0(\Phi)= \mathrm{rk}(\g)$. Finally, the dimension of $\mathrm{H}^1(\Phi)$ can be computed via the Euler characteristic of the complex which is zero.
\end{proof}

The first $\Phi$-cohomology group describes variations of $G$-Fock fields which do not necessarily preserve the nilpotency condition:
\begin{proposition}
The first $\Phi$-cohomology group $\mathrm{H}^1(\Phi)$ describes variations of a $G$-Fock field $\Phi$ leaving the condition $[\Phi\wedge\Phi]=0$ invariant, modulo gauge transformations.
\end{proposition}

\begin{proof}
   A variation $\delta \Phi$ preserves the condition $[\Phi\wedge\Phi]=0$ if and only if $[\Phi\wedge\delta \Phi]=0$, that is, if and only if $\delta \Phi$ is a cocycle. An infinitesimal gauge transformation $\eta\in \Omega^0(S,\g_P)$ induces $\delta\Phi=[\Phi,\eta]$ which is a coboundary.
\end{proof}

We next analyze variations of Fock bundles including the nilpotency condition. The space $\Omega^1(S,\g_P)$ has a natural symplectic structure $\omega$ defined by
\begin{equation}\label{Eq:sympl-str-1-forms}
\omega(\alpha,\beta)=\int_S \tr\, \alpha\wedge\beta,
\end{equation}
where $\tr$ denotes the Killing form of $\g$.
The symplectic structure $\omega$ descends to a symplectic structure on $\mathrm{H}^1(\Phi)$. Indeed, for $\alpha,\beta\in \Omega^1(S,\g_P)$ representing classes in $\mathrm{H}^1(\Phi)$, adding a coboundary $[\gamma,\Phi]$ to $\alpha$ adds
$$\mathrm{tr}([\gamma,\Phi] \wedge\beta) = \mathrm{tr}(\gamma [\Phi\wedge\beta]) = 0.$$
The same holds when adding a coboundary to $\beta$. Hence the symplectic form only depends on the cohomology classes. Denote by $\mathrm{Var}(\Phi)$ the variations of $\Phi$ which both preserve $[\Phi\wedge\Phi]=0$ and the nilpotency condition.

\begin{proposition}\label{Prop:Phi-variation}
    The variations of Fock fields $\mathrm{Var}(\Phi)$ are described by the direct sum $$\mathrm{Var}(\Phi)\cong \mathrm{Im}(\ad_\Phi)\oplus Z(\Phi)\bar{K},$$
    where $Z(\Phi)\subset \Omega^0(S,\g_P)$ denotes the kernel of $\ad_\Phi$ and $\bar{K}$ the conjugated canonical bundle.
\end{proposition}
\begin{proof}
    Using the complex structure on $S$ induced from the Fock bundle, we can locally write $\Phi=\Phi_1 dz+\Phi_2d\bar{z}$, where $\Phi_1$ is principal nilpotent. Since all principal nilpotent matrices are conjugate, any variation preserving the nilpotency condition is equivalent modulo $\mathrm{Im} \ad_{\Phi}$ to a variation which changes $\Phi_2$ without changing $\Phi_1$. Such variation is of the form $Z(\Phi_1)d\bar{z}$ since $[\Phi\wedge\delta\Phi]=0$.
\end{proof}

\begin{corollary}
    For a $G$-Fock bundle $(P, \Phi)$ over $S$, $\mathrm{Var}(\Phi)$ forms an isotropic subspace in $(\Omega^1(S,\g_P), \omega)$. Modulo $\mathrm{Im}(\ad_{\Phi})$, it descends to an isotropic subspace of $\mathrm{H}^1(\Phi)$.
\end{corollary}
\begin{proof}
Take two variations of the Fock field $\Phi$  given by $\delta_1\Phi=\ad_\Phi(\eta_1)+C_1d\bar{z}$ and $\delta_2\Phi=\ad_{\Phi}(\eta_2)+C_2d\bar{z}$, with $C_1,C_2\in Z(\Phi)$ and $\eta_1,\eta_2\in\Omega^0(S,\g_P)$. One now checks that the symplectic form $\omega(\delta_1\Phi,\delta_2\Phi)$ defined by \eqref{Eq:sympl-str-1-forms} vanishes. Indeed, since $\ad_\Phi$ squares to zero, we get $\mathrm{tr} \ad_\Phi(\eta_1)\wedge \ad_{\Phi}(\eta_2)=0$. Moreover, $C_1 d\bar{z}\wedge C_2d\bar{z}=0$ and the wedge product of the crossed terms also vanishes since $C_1,C_2\in Z(\Phi)$.
\end{proof}
Note that pointwise, $\mathrm{Var}(\Phi)$ is a Lagrangian in $T^*S\otimes \g$ (where we choose a volume form at a point to identify $\Lambda^2(TS)$ with $\C$).

\subsubsection{Introduction of $\sigma_0$-structures}\label{sec_sigma_split_coh} 

We first analyse the possible $\sigma_0$-structures $\sigma$ one can put on a pair $(P,\Phi)$ where $\Phi$ satisfies the first two conditions of Definition \ref{defn_Fock-bundle}.

\begin{proposition}\label{Prop:sigma-dependence}
    For any such pair $(P,\Phi)$, there 
    locally exist $\sigma_0$-structures making it into a Fock bundle. If $G=G_{ad}$, they exist globally.
\end{proposition}

\begin{proof}
Recall that $\Phi_1$ is valued in principal nilpotent elements. By Lemma \ref{lem:space_of_sigma} there is a contractible space of choices of involution of  each fiber of the adjoint bundle which negate the image of $\Phi_1$. We can thus choose $\bar{\sigma}:{\mf{g}_P}\to \mf{g}$ which negates $\Phi_1$ globally. By Lemma $\ref{cor:sigma_negate_centralizer}$, $\bar{\sigma}$ will also negate the centralizer of $\Phi_1$, so must negate $\Phi_2$ as well. When $G$ is adjoint, $\bar{\sigma}$ uniquely determines a $\sigma_0$ structure $\sigma:P\to P$.

When $G$ is not adjoint, there is slightly more data to a $\sigma_0$-structure. The involution $\bar{\sigma}$ gives a reduction to $\pi^{-1}(G_{ad}^{\sigma})$ where $\pi:G\to G_{ad}$ is the projection. A $\sigma_0$-structure on the other hand is a reduction to $G^{\sigma}$ which is the identity component of this group. Reductions of a principal bundle to its identity component are sections of a bundle with discrete fibers, so they always exist locally.
\end{proof}

\begin{proposition}\label{Prop:var-sigma}
    The infinitesimal variations of $\sigma_0$-structures negating $\Phi$ are described by $\mathrm{H}^0(\Phi)$.  
\end{proposition}
\begin{proof}
    Consider a $\sigma_0$-structure $\sigma$ satisfying $\sigma(\Phi)=-\Phi$. Since all $\sigma_0$-structures are conjugate, any infinitesimal variation $\delta\sigma$ is described by $\xi\in\g^{-\sigma}$ and given by $\delta\sigma(x)=[\xi,x]$ for $x\in\g$.
    The variations negating $\Phi$ have to satisfy $\delta\sigma(\Phi)=0$, hence $\xi\in\ker\ad_\Phi$, i.e. $\xi\in \mathrm{H}^0(\Phi)$. 
\end{proof}

Actually a stronger statement is true: by Lemma \ref{lem:space_of_sigma} the exponential of $\mathrm{H}^0(\Phi)=Z(\Phi)$ acts simply transitively on the space of choices of $\sigma_0$-structure. In particular, all $\sigma_0$ structures are conjugate.

We can now use the $\sigma_0$-structure $\sigma$ in the $\Phi$-cohomology.
Define $\Omega^{k,\sigma}(S,\g_P)$ to be the space of $\sigma$-invariant $\g_P$-valued $k$-forms, and similarly $\Omega^{k,-\sigma}(S,\g_P)$ to be the space of $\sigma$-anti-invariant forms. Since $\sigma(\Phi)=-\Phi$, the $\Phi$-cohomology complex splits as a direct sum of two complexes
$$\Omega^{0,\pm \sigma}(S,\g_P) \xrightarrow{\ad_\Phi\;} \Omega^{1,\mp \sigma}(S,\g_P) \xrightarrow{\ad_\Phi\;} \Omega^{2,\pm \sigma}(S,\g_P),$$
and we may define $\mathrm{H}^{k,\sigma}(\Phi)$ and $\mathrm{H}^{k,-\sigma}(\Phi)$ as the subspaces represented by $\sigma$-invariant and respectively $\sigma$-anti-invariant elements of these complexes. Equivalently, we can define them as the $\sigma$-invariant and $\sigma$-anti-invariant parts of $\mathrm{H}^k(\Phi)$:
$$\mathrm{H}^k(\Phi)\cong \mathrm{H}^{k,\sigma}(\Phi)\oplus \mathrm{H}^{k,-\sigma}(\Phi).$$

Moreover, one sees that the Lie bracket respects the $\sigma$-grading, because $\sigma$ commutes with the Lie bracket. 

\begin{proposition}\label{Prop-no-sigma-cohom}
We have $\mathrm{H}^{k,\sigma}(\Phi)=0$, for $k=0,1,2$.
\end{proposition}
\begin{proof}
One needs to show that $\sigma$ acts by $-1$ on $\Phi$-cohomology. This is a pointwise statement.
For $k=0$, the 0-th cohomology is simply the center $Z(\Phi)$, on which $\sigma$ acts by $-1$. The natural pairing between 0-forms and 2-forms is $\sigma$-invariant. Hence $\sigma$ also acts by $-1$ on $\mathrm{H}^2(\Phi)$.

For $k=1$, the centralizer $Z(\Phi)\bar{K}$ descends to a Lagrangian subspace in $\mathrm{H}^1(\Phi)$. This subspace is negated by $\sigma$, so is its complement under the symplectic pairing between 1-forms. This describes the entire space $\mathrm{H}^1(\Phi)$.
\end{proof}

In view of Proposition \ref{Prop-no-sigma-cohom}, we have $\mathrm{H}^k(\Phi)=\mathrm{H}^{k,-\sigma}(\Phi)$, hence $[\mathrm{H}^k(\Phi),\mathrm{H}^\ell(\Phi)]\subset \mathrm{H}^{k+\ell,\sigma}(\Phi)=0$. We thus have the following:
\begin{corollary}\label{coboundaries}
For representatives $\alpha,\beta$ of cohomology classes $[\alpha] \in \mathrm{H}^k(\Phi)$ and  $[\beta] \in \mathrm{H}^\ell(\Phi)$, the bracket $[\alpha, \beta]$ is a coboundary.
\end{corollary}

\section{Connections associated to Fock bundles}\label{Sec:can-connection}

In nonabelian Hodge theory, one associates to a
stable Higgs bundle a Hermitian--Yang--Mills metric in a gauge-equivariant way, and to a
polystable Higgs bundle is associated a Hermitian--Yang--Mills metric up to scale on
each stable factor. Subsequently, these metrics are associated to flat bundles over the Riemann surface. In this section, we will associate to Fock bundles a certain family of connections. We analyze the question under which conditions we can choose these connections to be flat. To begin with, we define a compatible connection on a Fock bundle:

\begin{definition}\label{Phi_comp_conn}
On a $G$-Fock bundle $(P,\Phi, \sigma)$, a connection $d_A$ is called \emph{$\Phi$-compatible} if $d_A\Phi=0$.
\end{definition}

Note that $d_A\Phi=0$ is an affine equation in the connection matrix $A$, thus easily provides the existence of $\Phi$-compatible connections. Indeed, local solutions can be found by choosing a local trivialization of $P$ in which the centralizer of $\Phi$ is constant. In this trivialization, $d\Phi$ will be valued in the centralizer of $\Phi$, thus it will be in the image of $\ad_{\Phi}$ (see Proposition \ref{Prop:center-in-image}), allowing us to find a matrix valued $1$-form $A_0$ with $d\Phi + [A_0\wedge\Phi]=0$. These local solutions can be now patched together using a partition of unity. We thus have the following result, which is analogous to the construction of the Chern connection for a given Dolbeault operator and a hermitian metric:

\begin{proposition}\label{Prop:comp-connections}
There exist $\Phi$-compatible connections on any given $G$-Fock bundle.
\end{proposition}

\subsection{Three-term connections}\label{Sec:3-term-connections}

In Section \ref{sec:examples of 3-term conn} we reviewed two sorts of families of 3-term connections from the works of Hitchin \cite{Hit87} and Fock \cite{Fock} subject to a parameter $\lambda \in \mathbb{C}^{*}$. The form of these families comes from the general theory of hyperk\"{a}hler manifolds and their twistor space \cite{HKLR87}. Within this framework, we consider a family of connections on a principal $G$-bundle $P$ over a surface $S$ of the form 
\begin{equation}
 \mathcal{A}(\lambda)= \lambda^{-1}\Phi+d_A+\lambda \Psi,
\end{equation}
where $\lambda\in \mathbb{C}^*$ is a parameter, $d_A$ is a fixed background connection on $P$ and $\Phi,\Psi\in\Omega^1(S,\mathfrak{g}_P)$.
The curvature of a connection in this family is a Laurent polynomial in $\l$ given by:
\begin{align}\label{curvature_3term}
F(\mathcal{A}(\lambda))&=[\mathcal{A}(\lambda)\wedge\mathcal{A}(\lambda)] \nonumber \\
& = \lambda^{-2}[\Phi\wedge\Phi]+\lambda^{-1}d_A(\Phi)+ F(A)+[\Phi\wedge\Psi]+\lambda d_A(\Psi)+\lambda^2[\Psi\wedge\Psi],
\end{align}
where $F(A)$ denotes the curvature of the fixed background connection $d_A$.

In order to have a family $\mathcal{A}(\lambda)$ of flat connections, i.e. flat for all values of $\lambda$, all five coefficients of the Laurent polynomial in \eqref{curvature_3term} have to vanish. Note that if $\Phi$ and $\Psi$ are Fock fields, then the terms $\lambda^{-2}[\Phi\wedge\Phi]$ and  $\lambda^2[\Psi\wedge\Psi]$ in (\ref{curvature_3term}) do vanish. If, in addition, the background connection $d_A$ is a connection compatible with both $\Phi$ and $\Psi$, then the curvature $F(\mathcal{A}(\lambda))$ becomes independent of the parameter $\lambda$.
In this situation, the flatness of $\mathcal{A}(\lambda)$ is equivalent to the equation
\begin{equation}\label{Fock equation}
F(A)+[\Phi\wedge\Psi] = 0.
\end{equation}
Note the similarity of this equation to the form of the well-known Hitchin equation from \cite{Hit87}.

In Section \ref{Sec:canonical-connection} below, we will show the existence of a connection $d_A$ compatible with both Fock fields $\Phi$ and $\Psi$. For this to be possible, the two fields have to satisfy a certain transversality condition.

\subsection{Transversality}

We next define the notion of transverse Fock fields.
\begin{definition}
For a pair of Fock fields $\Phi, \Psi \in \Omega^1(S,\mathfrak{g}_P)$, we call $\Phi$ and $\Psi$ \emph{transverse} if $\Omega^1(S,\g_P)=\ker(\ad_\Phi)\oplus \mathrm{Im}(\ad_\Psi)$ and $\Omega^1(S,\g_P)=\ker(\ad_\Psi)\oplus \mathrm{Im}(\ad_\Phi)$.
\end{definition}
Figure \ref{Fig:transversality} illustrates the transversality condition; note that the figure shows the decomposition of $\Omega^1(S,\g_P)$ as a direct sum, not as a union.

\begin{figure}[h!]
\centering
\includegraphics[height=4.5cm]{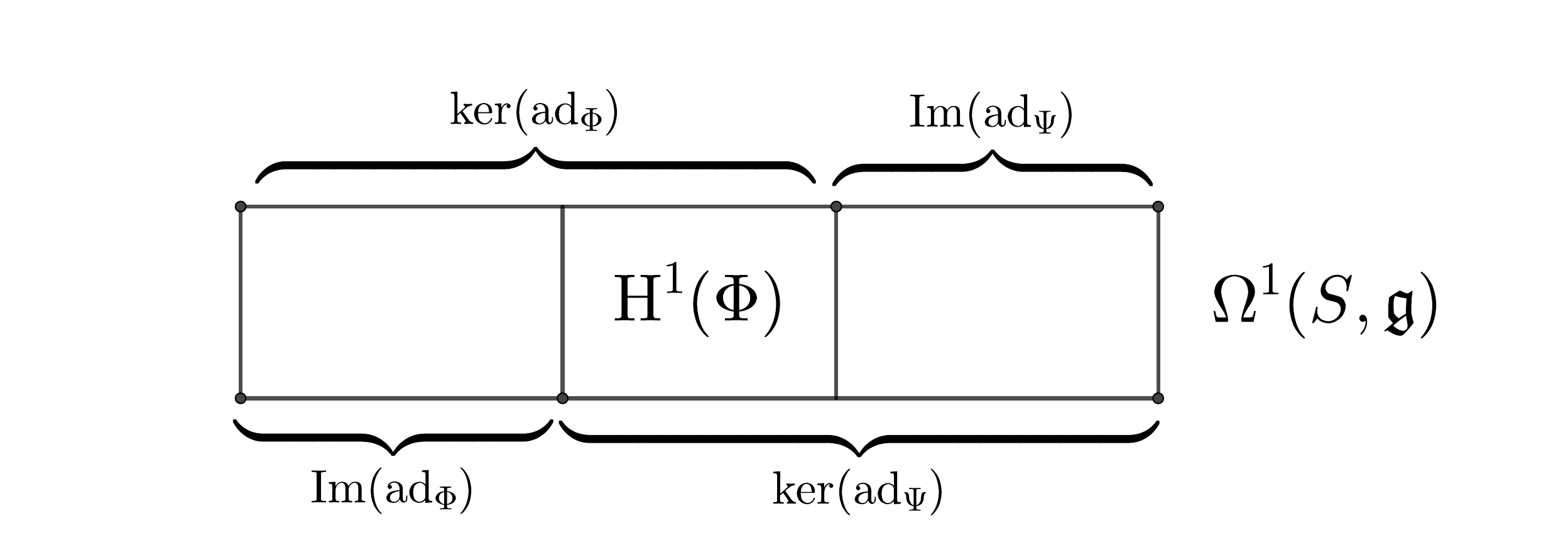}
\caption{Decomposition of $\Omega^1(S,\g_P)$ for transverse $\Phi$ and $\Psi$.}\label{Fig:transversality}
\end{figure}

\begin{proposition}\label{Prop:repr-for-cohom}
    For transverse $\Phi$ and $\Psi$, the cohomology $\mathrm{H}^1(\Phi)$ can be identified with $\ker(\ad_\Phi)\cap\ker(\ad_\Psi)$.
\end{proposition}
\begin{proof}
    Let $[\alpha]\in \mathrm{H}^1(\Phi)$ be represented by $\alpha\in \ker(\ad_\Phi)$. Using the transversality condition $\Omega^1(S,\mathfrak{g}_P)=\ker(\ad_\Psi)\oplus \mathrm{Im}(\ad_\Phi)$, we see that the projection onto $\ker(\ad_\Psi)$ changes $\alpha$ only by a coboundary. Hence we can use this projection as a preferred representative for the class $[\alpha]$.
\end{proof}

Moreover, in the setting of $\sigma$-splitted $\Phi$-cohomology from Section \ref{sec_sigma_split_coh}, we get the following corollary:
\begin{corollary}\label{Coro:sigma-decompo-transverse}
    For a pair of transverse Fock fields $\Phi, \Psi \in \Omega^1(S,\mathfrak{g}_P)$, one has: 
\begin{equation*}
\Omega^1(S,\mathfrak{g}_P)^\sigma=\mathrm{Im}(\ad_\Phi)^\sigma\oplus \mathrm{Im}(\ad_\Psi)^\sigma.
\end{equation*}
\end{corollary}
This follows from the fact that $\mathrm{H}^{1,\sigma}(\Phi)=0$ (see Proposition \ref{Prop-no-sigma-cohom}) and the identification of $\ker(\ad_\Phi)\cap \ker(\ad_\Psi)$ with $\mathrm{H}^1(\Phi)$.

The transversality condition is an open condition. In the unitary setting (where the bundle $P$ is equipped with a hermitian structure $\rho$), we will see that the Fock field $\Psi=\Phi^{*}$ can be made transverse to $\Phi$, by a suitable diagonal gauge transformation.

\subsection{Canonical connection from transverse Fock fields}\label{Sec:canonical-connection}

We now associate a canonical connection $d_A$ to a given pair of transverse Fock fields $\Phi$ and $\Psi$. This allows to get a family of 3-term connections $\l^{-1}\Phi+d_A+\l\Psi$. The existence and uniqueness of $d_A$ brings to mind the Chern connection on a holomorphic bundle with hermitian structure.

\begin{theorem}\label{Thm-filling-in}
Let $(P,\Phi,\sigma)$ be a $G$-Fock bundle equipped with a second Fock field $\Psi$, transverse to $\Phi$. Then there is a unique $\sigma$-invariant connection $d_A$ on $P$ which is compatible with both $\Phi$ and $\Psi$, i.e. solving the equations 
\[d_A\Phi=d_A\Psi=0.\] 
\end{theorem}
In the sequel, we call $d_A$ the \emph{canonical connection} and refer to the construction of $d_A$ as the \emph{filling-in procedure} (since it gives the middle term in the 3-term connection). The associated 3-term connection $\Phi+d_A+\Psi$ will be called the \emph{canonical 3-term connection}. 

\begin{proof}
\emph{Existence}: By Proposition \ref{Prop:comp-connections}, we first find a compatible connection $d_{A_0}$ on $P$, i.e. such that $d_{A_0}\Phi = 0$.
The existence of $d_{A_0}$ relies on the fact that for appropriate local trivializations, $d\Phi$ lies in the image of $\ad_\Phi$.
Note that we can choose $d_{A_0}$ to be $\sigma$-invariant since $\sigma(\Phi)=-\Phi$.

\noindent The same argument as for $\Phi$ shows that $d_{A_0}\Psi=d\Psi+[A_0\wedge\Psi]$ lies in the image of $\ad_\Psi$. Hence there is a section $R\in\Omega^1(S,\g_P)$ such that $d_{A_0}\Psi=[R\wedge\Psi]$. Since both $\Psi$ and $d_{A_0}\Psi$ are $\sigma$-anti-invariant, we can choose $R$ to be $\sigma$-invariant. By Corollary \ref{Coro:sigma-decompo-transverse} (since $\Phi$ and $\Psi$ are transverse) we can decompose $R=R_1+R_2$ where $R_1\in\mathrm{Im}(\ad_\Phi)^\sigma$ and $R_2\in\mathrm{Im}(\ad_\Psi)^\sigma$. Since $\ad_\Psi^2=0$, we have $[R\wedge\Psi]=[R_1\wedge\Psi]$.

We claim that $d_A:=d_{A_0}-R_1$ is a solution. Indeed, $d_A\Psi=0$ by definition of $R_1$ and $d_A\Phi=-[R_1\wedge\Phi]=0$ since $R_1\in\mathrm{Im}(\ad_\Phi)$ and $\ad_\Phi^2=0$. Finally, both $d_{A_0}$ and $R_1$ are $\sigma$-invariant. Therefore $d_A$ is indeed a solution.

\emph{Uniqueness}: The space of solutions to the equations $d_{A}\Phi=d_{A}\Psi=0$ is an affine subspace of the space of connections. If $d_A$ is one such solution, then all solutions are of the form $d_A+C$, where $C\in \ker(\ad_{\Phi})\cap \ker(\ad_{\Psi})$. Lemma \ref{Prop-no-sigma-cohom} shows that the involution $\sigma$ acts by $-1$ on the $\Phi$-cohomology, so acts by $-1$ on $\ker(\ad_{\Phi})\cap \ker(\ad_{\Psi})$. Since $d_A+C$ and $d_A$ are $\sigma$-invariant, $C$ has to be $\sigma$-invariant. Therefore $C=0$. 
\end{proof}

\subsection{Unitary case}\label{Sec:unitary-case}

We equip a $G$-Fock bundle $(P,\Phi,\sigma)$ with a compatible hermitian structure $\rho$. Ideally we would like to apply the construction of a canonical connection from Theorem \ref{Thm-filling-in} for the pair of Fock fields $\Phi$ and $\Psi=\Phi^{*}:=-\rho(\Phi)$. However, the construction can not be used directly since the Fock fields $\Phi$ and $\Phi^{*}$ need not be necessarily transversal. We will see that we can always conjugate $\Phi$ so that $\Phi$ and $\Phi^{*}$ become transversal.

\begin{definition}
    A hermitian structure $\rho$ on a $G$-Fock bundle $(P,\Phi,\sigma)$ is called \emph{compatible} if the involutions $\rho$ and $\sigma$ of $\g_P$ commute. 
\end{definition}
In the case of an $\mathrm{SL}_n(\C)$-Fock bundle $(E,\Phi,g)$, this is equivalent to the existence of a real vector bundle $E^{\mathbb{R}}$ equipped with a real metric $g^{\mathbb{R}}$ such that the complexification of $E^{\mathbb{R}}$ gives the bundle $E$, the complex linear extension of the real metric $g^{\mathbb{R}}$ gives $g$ and the hermitian extension gives the hermitian form $h$.

\begin{proposition}
    For any $G$-Fock bundle $(P,\Phi, \sigma)$, there is a compatible hermitian structure.
\end{proposition}
\begin{proof}
    The statement is pointwise, so we can reduce to a problem of Lie algebra involutions on $\g$. The proposition directly follows from the fact that for any Lie algebra involution (here $\sigma$), there is a compact real form which commutes with it \cite[Theorem 6.16]{knapp}. 
\end{proof}

\begin{proposition}
    Let $(P,\Phi,\sigma)$ be a $G$-Fock bundle equipped with a compatible hermitian structure $\rho$. Then $(P,\Phi^*,\sigma)$ is again a $G$-Fock bundle.
\end{proposition}
\begin{proof}
     Applying the hermitian conjugation to $[\Phi\wedge\Phi]=0$ we get the same condition for $\Phi^{*}$. Since the set of principal nilpotent matrices is invariant under hermitian conjugation, $\Phi^*$ satisfies the nilpotency condition. Finally, $\Phi^{*}$ is negated by $\sigma$: $$\sigma(\Phi^{*})=-\sigma\circ\rho(\Phi) = -\rho\circ\sigma(\Phi)=\rho(\Phi)=-\Phi^{*}.$$
\end{proof}

In order to use the filling-in procedure of Theorem \ref{Thm-filling-in}, we need $\Phi$ and $\Phi^{*}$ to be transverse.
For that, note that the spaces $\mathrm{Im}(\ad_{\Phi})$ and $\ker(\ad_{\Phi^{*}})$ are orthogonal complements with respect to the (pseudo-)hermitian form on $T^*S\otimes\g_P$ given by 
\begin{equation*}
\langle \alpha, \beta\rangle d\bar{z}\wedge dz=\mathrm{tr}(\alpha^{*}\wedge \beta),
\end{equation*}
where $\tr$ denotes the Killing form on $\g$.
These are indeed orthogonal complements because $\ad_\Phi$ and $\ad_{\Phi^{*}}$ are adjoint with respect to the inner product on the entire space $\Lambda^\bullet T^*S\otimes\g_P$ given by this same formula. Note that we chose a volume form to define this, but different choices of volume form will lead to metrics which differ by a scalar, so the notions of adjoint, and orthogonal complement do not depend on this choice. This hermitian form is not positive definite though, in fact, the norm of $a\,dz+b\,d\bar{z}$ is given by $\mathrm{tr}(a^{*} a)-\mathrm{tr}(b^{*} b)$. 

On $\Omega^1(S,\g_P)$, the pseudo-hermitian form is given by the formula
\begin{equation}\label{Eq:herm-form}
    \langle \alpha,\beta\rangle = -\tfrac{i}{2}\int_S\tr \alpha^*\wedge\beta.
\end{equation}
Recall that any subspace $W$ of a vector space equipped with an indefinite form is complementary to a subspace $W^\perp$, precisely whenever the form is non-degenerate when restricted to $W$. This motivates the following:

\begin{definition}\label{Def:pos-Higgs-field}
    For a $G$-Fock bundle $(P,\Phi)$ with hermitian form $\rho$, the Fock field $\Phi$ is called \emph{positive} (with respect to $\rho$) if $\mathrm{Im}(\ad_\Phi)\subset \Omega^1(S,\g_P)$ is positive-definite with respect to the pseudo-hermitian form in \eqref{Eq:herm-form}. The space of positive Fock fields is denoted by $\mathcal{P}$. We also refer to $\rho$ being a positive hermitian structure (on the Fock bundle $(P,\Phi)$).
\end{definition}

Note that for a positive Fock field $\Phi\in\mathcal{P}$, the fields $\Phi$ and $\Phi^{*}$ are transverse. Note further that the Fuchsian locus, as described in Section \ref{Sec:fuchsian-locus}, is included in $\mathcal{P}$ since $\mathrm{Im}(\ad_{\Phi})$ consists only of $(1,0)$-forms. 

\begin{lemma}\label{lemma:diag-gauge-to-positive}
For every Fock field $\Phi\in \Omega^1(S,\g_P)$ equipped with a hermitian form $\rho$, there is a diagonal gauge transformation $a\in \mathrm{Aut}(P,\sigma)$ such that $a^{-1} \Phi a$ is positive.
\end{lemma}

\begin{proof}
Consider a $G$-Fock bundle $(P,\Phi,\sigma)$. Decompose $\Phi=\Phi_1+\Phi_2$. By Lemma \ref{lemma:unique-sigma-sl2},  there is a unique $\sigma$-invariant $\mathfrak{sl}_2$-subbundle $\mathfrak{p}\subset \g_P$ obtained by completing $\Phi_1$ into a principal $\mathfrak{sl}_2$-triple. Let $H$ be a non-vanishing section of the line subbundle $\mathfrak{l}\subset \mathfrak{p}$ consisting of the semisimple elements of the principal $\mathfrak{sl}_2$-triples (acting on the nilpotent part by scaling). Note that $\mathfrak{l}$ is a degree zero line bundle, so this section indeed exists. 

Consider the gauge transformation $a_t=\exp(tH)$. Since $H$ is $\sigma$-invariant, we have $a_t\in\mathrm{Aut}(P,\sigma)$. The element $H$ allows to decompose the bundle $\g_P\cong \oplus_{k\in\mathbb{Z}}\, \g_{P,k}$ where elements in $\g_{P,k}$ are eigenvectors of $\ad_H$ with eigenvalue $k$. We know that $\Phi_1\in \g_{P,-2}$ since $\Phi_1$ and $H$ are part of the same $\mathfrak{sl}_2$-triple. We have $\Phi_2\in Z(\Phi_1)$, i.e. $\Phi_2$ is a linear combination of lowest weight vectors. In addition, $\Phi_2$ is not regular since we work in the complex structure induced from $(P,\Phi)$. This implies that $\Phi_2\in\oplus_{k\leq -4}\g_{P,k}$ since in the decomposition of $\g_P$ into irreducible $\mathfrak{sl}_2$-modules (via adjoint action by $\mathfrak{p}$), only one module is of dimension 3.

Therefore the quantity $e^{2t}a_t^{-1}\Phi a_t$ converges to $\Phi_1$ for $t\to\infty$. Since $\Phi_1$ is clearly positive and positive Fock fields form an open set, we conclude.
\end{proof}

For transverse $\Phi$ and $\Phi^{*}$ we can use the filling-in procedure from Theorem \ref{Thm-filling-in} to get a canonical connection $d_A$.
\begin{proposition}
    For a Fock bundle $(P,\Phi,\sigma)$ equipped with a compatible positive hermitian structure $\rho$, then the canonical connection $d_A$ is unitary.
\end{proposition}

\begin{proof}
One can first check that both $d_A$ and $\rho(d_A)$ are $\sigma$-invariant solutions to the equations $d_A(\Phi)=0$ and $d_A(\Phi^{*})=0$. Indeed, since $\sigma$ and $\rho$ commute, $\rho(d_A)$ is $\sigma$-invariant. Applying $\rho$ to the equation $d_A\Phi=0$ gives $\rho(d_{A})(\Phi^{*})=0$. Similarly, applying $\rho$ to $d_A(\Phi^{*})=0$ gives $\rho(d_A)(\Phi)=0$. Hence $\rho(d_A)$ is also a solution. By uniqueness from Theorem \ref{Thm-filling-in}, we get $\rho(d_A)=d_A$, thus the connection $d_A$ is unitary.
\end{proof}

An important example of positive Fock fields is the Fuchsian locus:
\begin{proposition}\label{Prop:chern-like-connection}
If $(P,\Phi,\sigma)$ lies in the Fuchsian locus, then there is a compatible positive hermitian structure $\rho$. In addition, the canonical connection $d_A$ is the Chern connection.
\end{proposition}
\begin{proof}
    From Proposition \ref{prop:fock-fuchsian-locus}, we know that a Fock bundle in the Fuchsian locus comes from some uniformizing $G$-Higgs bundle $(V,\Phi)$. From \cite{Hit92} we know that $V$ can be equipped with a compatible hermitian structure $\rho$ and $\sigma_0$-structure $\sigma$. Since $\Phi\in\Omega^{1,0}(S,\g_V)$, the image of $\ad_\Phi$ is positive definite for the form defined in \eqref{Eq:herm-form}, i.e. $\Phi$ is positive.
    
\noindent  From nonabelian Hodge theory, we know that the Chern connection $d_A$ satisfies $d_A\Phi=0$ and $\sigma(d_A)=d_A$ \cite[Equation (7.2)]{Hit92}. Hence the Chern connection is the result of the filling-in procedure by the uniqueness part of Theorem \ref{Thm-filling-in}.
\end{proof}

\subsection{Main open question}

We have seen that on a $G$-Fock bundle $(P,\Phi,\sigma)$ with compatible positive hermitian structure $\rho$, there is a canonical 3-term connection $\Phi+d_A+\Phi^{*}$. Our main open question asks when we can choose $\rho$ to get a \emph{flat connection}. We give partial answers and a reformulation in symplectic terms.

Instead of varying $\rho$, we can equivalently fix the hermitian structure $\rho$ and vary the gauge class of the Fock field.
\begin{question}\label{main-conj}
Let $(P, \Phi_0, \sigma)$ be a $G$-Fock bundle over $S$ equipped with a compatible hermitian structure $\rho$. When does there exist a hermitian endomorphism-valued function \[\eta\in \Omega^0(S,\g_P^{\sigma, -\rho})\] such that $\Phi = e^{-\eta} \Phi_0 e^{\eta}$ is positive, and the corresponding 3-term connection $\Phi + d_A + \Phi^{*}$ is flat? The latter means that $\rho$ satisfies the equation 
\begin{equation}\label{Eq:Hitchin-like-3}
F(A) + [\Phi\wedge\Phi^{*}] = 0.
\end{equation}
\end{question}

This question is appealingly similar to nonabelian Hodge theory associating a flat connection to the gauge equivalent class of any polystable Higgs bundle. In our setting, Fock bundles are automatically stable in the sense of Proposition \ref{Prop:no-automorphisms}.
Apart from the similarity with nonabelian Hodge theory, we have some partial answers:

\begin{enumerate}
\item There is a class of Fock bundles for which existence of a solution to Equation \eqref{Eq:Hitchin-like-3} is guaranteed, namely, for Fock bundles in the Fuchsian locus from Section \ref{Sec:fuchsian-locus}. 
Proposition \ref{Prop:chern-like-connection} tells us that $\Phi$ and $\Phi^{*}$ are transverse and that the Chern connection is the result from our construction of 3-term connections. Therefore, the nonabelian Hodge correspondence gives the existence of a hermitian structure.

\item The linearization of Equation \eqref{Eq:Hitchin-like-3} is an elliptic isomorphism. This is developed below in Section \ref{Sec:elliptic}. We can then conclude that the subset of Fock fields for which there exists a solution is open.

\item An interesting class of examples is \emph{harmonic higher complex structures} introduced in \cite{Nolte}. Fix a complex structure on $S$ giving a hyperbolic metric $g_S$ (via uniformization). A higher complex structure is called \emph{harmonic} if its induced complex structure on $S$ coincides with the fixed one and if 
\begin{equation}\label{harmonicity_cond}
\bar\partial(\bar{\mu}_k g_S^{k-1})=0,
\end{equation}
for all $k\in\{3,...,n\}$.  The gauge-theoretic meaning of condition (\ref{harmonicity_cond}) can be given as follows: 
denote by $h_S$ the hermitian structure on $V=K^{(n-1)/2}\oplus...\oplus K^{(1-n)/2}$ induced from the hyperbolic metric on $S$. Denote again by $E$ the underlying smooth complex vector bundle of $V$. Then a Fock field $\Phi$, with $(1,0)$-part fixed to be the Fuchsian Fock field in Example \ref{Ex:Fuchsian-locus}, corresponds to a harmonic higher complex structure (via Proposition \ref{link_hcs}) if and only if $\bar\partial(\Phi^{*_{h_S}})=0$. Indeed, the hermitian metric $h_S$ is diagonal and is given by $h_S=\mathrm{diag}(g_S^{(1-n)/2},g_S^{(3-n)/2},...,g_S^{(n-1)/2})$. The non-zero entries of $(\Phi^{0,1})^{*_{h_S}}$ are given by $\bar{\mu}_k g_S^{k-1}$. 

One of the main results of \cite{Nolte} is that every equivalence class of higher complex structures (modulo higher diffeomorphisms) has a harmonic representative (unique up to the action of usual diffeomorphisms isotopic to the identity). In Section \ref{Sec:higher-diffeos} below, we give the gauge-theoretic meaning of higher diffeomorphisms as special gauge transformations. Hence it seems that we can reduce Question \ref{main-conj} to harmonic higher complex structures, in which the elliptic condition $\bar\partial(\Phi^{*_{h_S}})=0$ holds. We are very grateful to Alexander Nolte for his insight in this description. 

\item In the case of an $\mathrm{SL}_n(\C)$-Fock bundle $(E,\Phi,g)$ equipped with a hermitian structure $h$, we have two filtrations $\mathcal{F}$ and $\mathcal{F}^*$ associated to $(E,\Phi)$ and $(E,\Phi^{*})$ respectively by Proposition \ref{Prop:filtration}. They are transverse since $\mathcal{F}_k$ is $h$-orthogonal to $\mathcal{F}_{n-k}^*$. Therefore by considering $L_k=\mathcal{F}_k\cap \mathcal{F}_{n-k}^*$, we get a line decomposition of $E$ for which $\Phi$ is lower triangular, $\Phi^*$ is upper triangular and $h$ is diagonal. This line decomposition helps in providing estimates.

\item We can reformulate Question \ref{main-conj} in symplectic terms. The Atiyah--Bott symplectic form on the space of all connections gives a presymplectic form on the space of appropriate Fock bundles. The gauge group action gives a moment map which is precisely given by $F(A)+[\Phi\wedge\Phi^{*}]$ in this case. Question \ref{main-conj} then becomes equivalent to an infinite-dimensional symplectic reduction. The difficulty here comes from the fact that the presymplectic form has degenerate directions, in particular, the orbits of the higher diffeomorphism action. Along the complex gauge orbits, the form is non-degenerate though. This is work in progress. 
\end{enumerate}

\subsection{Reality}\label{Flat_Fock_vs_Hitchin_comp}

We show that for Fock bundles admitting a flat canonical 3-term connection, this connection always has monodromy in the split real form.

\begin{proposition}\label{Prop:real-monodr-1} Let $(P,\Phi, \sigma)$ be a $G$-Fock bundle equipped with a compatible positive hermitian structure $\rho$. Suppose the canonical 3-term connection $\nabla=\Phi+d_A+\Phi^{*}$ associated to the tuple $(P,\Phi,\sigma,\rho)$ is flat. Then the monodromy of $\nabla$ lies in the split real form of the Lie algebra $\mathfrak{g}$.
\end{proposition}

\begin{proof}
We simply show that $\Phi+d_A+\Phi^{*}$ is invariant under the involution $\tau=\sigma\rho=\rho\sigma$. Recall that $\sigma(\Phi)=-\Phi$. Hence
$$\tau(\Phi)=\rho\circ\sigma(\Phi)=\rho(-\Phi)=\Phi^{*}.$$

\noindent Since the canonical connection $d_A$ is unitary, we have $\rho(d_A)=d_A$. Moreover, the connection $d_A$ is also $\sigma$-invariant, thus we get
$\tau(d_A)=\rho(\sigma(d_A))=d_A$.
Therefore, $\Phi+d_A+\Phi^{*}$ is $\tau$-invariant. By \cite[Proposition 6.1]{Hit92}, the monodromy of the connection $\nabla$ lies in the split real form of $\g$.
\end{proof}

The following corollary is immediate:

\begin{corollary}\label{Prop:real-monodr-2}
The connected component of Fock bundles admitting a unique solution to Equation \eqref{Eq:Hitchin-like-3} containing the Fuchsian locus, maps to the Hitchin component.
\end{corollary}

\section{Neighborhood of the Fuchsian locus}\label{Sec:ellipticity}

We next show that in an open neighborhood of the Fuchsian locus, one can always associate a flat  3-term connection to a given $G$-Fock bundle equipped with a compatible hermitian structure; this provides a partial answer to Question \ref{main-conj}. Throughout the whole section, let $(P,\Phi, \sigma)$ be a $G$-Fock bundle equipped with a compatible positive hermitian structure $\rho$.

\subsection{Linearized Equation}\label{Sec:elliptic}

We show that the differential of the map from positive Fock fields to curvature $F(A) + [\Phi\wedge \Phi^{*}]$, is an elliptic operator of order two and, in fact, is an isomorphism of appropriate Sobolev spaces. We compute the change in curvature as we infinitesimally conjugate $\Phi$ by a hermitian endomorphism. 

Let $(P,\sigma,\rho)$ be a principal $G$-bundle equipped with $\sigma_0$-structure $\sigma$ and compatible hermitian structure $\rho$. Define $F$ to be the map $\mathcal{P}\to\Omega^2(S,\g_P)$ given by the curvature of the canonical 3-term connection:
\begin{equation}\label{Eq:curv-operator}
    F(\Phi):=F(A)+[\Phi\wedge\Phi^*].
\end{equation}

\begin{lemma} 
Let $\Phi_t$ be a smooth path of positive Fock fields on a $G^{\sigma,\rho}$-bundle with $\Phi_0 = \Phi$, and whose derivative at $t=0$ is\[\dot{\Phi} :=  [\Phi,\eta]\] for some $\eta\in \Omega^0(S,\mf{g}_P^{\sigma,-\rho})$. Let $A_t$ be the path of connections obtained by the filling-in procedure of Theorem \ref{Thm-filling-in}. Then we have
\[\dot{A} = Q d_A \eta,\]
where $Q$ is the involution on $\Omega^{1}(S,\mf{g}_P)^\sigma\cong \mathrm{Im}(\ad_\Phi)^\sigma\oplus\mathrm{Im}(\ad_{\Phi^*})^\sigma$ defined by $-1$ on $\mathrm{Im}(\ad_\Phi)^\sigma$ and by $1$ on $\mathrm{Im}(\ad_{\Phi^*})^\sigma$. Note that we used Corollary \ref{Coro:sigma-decompo-transverse} here.
\end{lemma}

\begin{proof}
    It suffices to check that the derivatives of $d_A\Phi$ and $d_A\Phi^{*}$ vanish at $t=0$. One sees
    \[\frac{d}{dt} d_A\Phi = d_A \dot\Phi + [\dot{A}\wedge\Phi]=-[\Phi\wedge d_A \eta] + [\dot{A}\wedge\Phi] = [\Phi\wedge(\dot{A} -d_A\eta)]\] and similarly
    \[\frac{d}{dt} d_A\Phi^{*} = [\Phi^{*}\wedge(\dot{A} +d_A\eta)].\]
    The fact that $\ad_{\Phi}^2=0$ and $\ad_{\Phi^{*}}^2=0$, implies that $\dot{A} = Q d_A \eta$ indeed solves both of these equations. 
\end{proof}

We can now easily calculate the variation of curvature as we move the Fock field in the direction of $\eta$. 

\begin{theorem}\label{Thm:ellipticity}
The derivative of the map $F$ from \eqref{Eq:curv-operator} is the differential operator $L:\Omega^0(S,\mf{g}_P^{\sigma,-\rho}) \to \Omega^2(S,\mf{g}_P^{\sigma,\rho})$ given by 
\[L\eta = d_A Q d_A \eta + [[\Phi,\eta]\wedge\Phi^{*}] - [\Phi\wedge [\Phi^{*},\eta]]. \]
The operator $L$ is elliptic, and extends to a continuous linear isomorphism of Sobolev spaces 
\[H_{l+2}(S,\mf{g}_P^{\sigma,-\rho})\to H_{l}(S,\Lambda^2T^*S\otimes\mf{g}_P^{\sigma,\rho}),\]
for all $l\in\mathbb{N}$, where $H_l$ denotes the Sobolev space with $l$ square integrable derivatives.
\end{theorem}

\begin{proof} 
To compute the symbol of $L$, we only need to look at the highest order term $d_A Q d_A$. This is a composition of three operators, so its symbol is the composition of three symbols. The symbol $\sigma_{d_A}$ of $d_A$ is given by the formula
\[\sigma_{d_A}(\alpha)\eta = \alpha\wedge\eta.\]
This is a map $T^*S\times \Lambda^\bullet T^*S\otimes \mf{g}_P\to \Lambda^\bullet T^*S\otimes \mf{g}_P$. Note that this does not actually depend on the connection $d_A$. Since $Q$ is an operator of order zero, the symbol of $Q$ is just itself. Putting these together, we get the symbol of $L$
\[\sigma_{L}(\alpha)\eta = \alpha \wedge Q (\alpha \eta).\]
The operator $L$ is elliptic if this is an isomorphism from $\mf{g}_P^{\sigma,-\rho}$ to $\Lambda^2T^*S \otimes \mf{g}_P^{\sigma,\rho}$ for all non-zero $\alpha$. This is implied if $-\frac{i}{2}\tr(\eta \alpha \wedge Q(\alpha \eta))$ is a definite form in $\eta$ for all non-zero $\alpha$. This is indeed true, because if we write $\alpha\eta = \xi + \xi^{*}$ where $\xi$ is in $\mathrm{Im}(\ad_\Phi)$, we get
\[-\tfrac{i}{2}\tr(\eta \alpha \wedge Q(\alpha \eta)) = -\tfrac i2\tr((\xi + \xi^{*}) \wedge (-\xi + \xi^{*})) = i \tr(\xi^{*}\wedge\xi)=-2\lVert \xi \rVert^2,\]
which is negative-definite on $\mathcal{P}$ by Definition \ref{Def:pos-Higgs-field}. 

Now we prove that $L$ induces an isomorphism $H_1\to H_{-1}$. The essential point is that the bilinear form $\int_S \tr(\eta L\eta)$ is equivalent to the standard $H_1$-norm (defined using any choice of metrics). Recall the pseudo-hermitian product given by Equation \eqref{Eq:herm-form}. We use integration by parts, and conjugation invariance of trace to write 
\begin{align*}
-\tfrac{i}{2}\int_S \tr(\eta L\eta)  &=  -\tfrac{i}{2}\int_S \tr(\eta d_A Q d_A \eta) + \tr(\eta[[\Phi,\eta]\wedge\Phi^{*}]) - \tr(\eta[\Phi\wedge [\Phi^{*},\eta]])\\
&=  -\tfrac{i}{2}\int_S -\tr(d_A \eta \wedge Q d_A \eta) + \tr([\eta,\Phi^{*}]\wedge[\Phi,\eta]) - \tr([\eta,\Phi] \wedge [\Phi^{*},\eta])\\
&=-\tfrac{i}{2}\int_S 2 \tr(\pi_{\mathrm{Im}\,\ad_\Phi}(d_A \eta)^{*} \wedge \pi_{\mathrm{Im}\,\ad_\Phi}(d_A \eta)+ 2\tr([\eta,\Phi^{*}]\wedge[\Phi,\eta])\\
&= 2\lVert \pi_{\mathrm{Im}\,\ad_\Phi}(d_A\eta) \rVert^2 +2\lVert [\Phi,\eta]\rVert^2.
\end{align*}
 
    This is a norm on $\Omega^1(S,\mathfrak{g}^{\sigma,-\rho})$ which is the integral of the sum of a positive definite norm of the first derivatives of $\eta$, and a positive definite norm of $\eta$ pointwise. This means it is equivalent to the standard $H_1$-norm. Recall that $H_{-1}(S,E\otimes \Lambda^2T^*S)$ can be defined as the topological dual to $H_1(E)$ for any vector bundle $E$. The norms being equivalent means that $L$ induces an isomorphism $H_1\to H_{-1}$. This implies that the induced maps $H_l \to H_{l-2}$ are all isomorphisms for $l\in\mathbb{N}, l\geq 1$.
\end{proof}

\subsection{Application of the implicit function theorem}\label{Sec:impl-fct-thm}

In this section we use standard techniques from nonlinear elliptic PDE to show that for Fock fields with small enough $(0,1)$-part, i.e. near the Fuchsian locus, there is a hermitian structure such that the canonical connection is flat. We start by recalling the implicit function theorem for Banach manifolds:

\begin{theorem}[Implicit function theorem]
Let $W\subset X\times Y$ be an open subset of a product of differentiable Banach manifolds. Let $f:W \to Z$ be a continuously differentiable function to a third differentiable Banach manifold. Suppose that $f(x,y)=z$, and that $df_{(x,y)}$ restricts to an isomorphism from $T_y Y$ to $T_zZ$. Then there is a function $g$ from a neighborhood $U$ of $x$ to $Y$ satisfying
\[f(x,g(x)) = 0\]
for all $x\in U$. Furthermore, $g$ is continuously differentiable. 
\end{theorem}

We will apply this where $X$ is the space of Fock fields for a principal bundle $P$ equipped with $\sigma$ and $\rho$, $Y$ is the space of sections of $\mf{g}_P^{\sigma,-\rho}$, $W\subset X\times Y$ is the subset of $(\phi,\eta)$ where $e^{-\eta}\Phi e^\eta\in \mathcal{P}$, $Z$ is the space of two-forms valued in $\mf{g}^{\sigma,\rho}$, and $f$ is the map taking $(\phi,\eta)$ to the curvature of the connection associated to the conjugated Fock field
$$f: \left \{ \begin{array}{ccc}
W &\longrightarrow & \Omega^2(S, \mf{g}_P^{\sigma,\rho}) \\
(\phi,\eta) &\mapsto & F(e^{-\eta}\Phi e^{\eta})
\end{array} \right.$$
where $F$ is the map from Equation \eqref{Eq:curv-operator}.
We need to equip all these spaces with differentiable Banach manifold structures such that the criteria of the implicit function theorem are satisfied. To do this we should understand the derivative of the map from Fock fields to curvature.

Let $\dot\Phi\in \mathrm{Var}(\Phi)^{-\sigma}$ be a tangent vector at $\Phi$ to the space of Fock fields. Let $\dot A$ be the corresponding variation in canonical connection. The derivative of $d_A\Phi$ and $d_A\Phi^*$ must be zero, giving us two equations on $\dot{A}$:
\[d_A\dot\Phi + [\dot{A}\wedge \Phi] = 0, \text{ and}\]
\[d_A\dot\Phi^* + [\dot{A}\wedge \Phi^*] = 0\]
By transversality, we know that $ad_\Phi$ restricts to an isomorphism from $\mathrm{Im}(\ad_{\Phi^*})^{\sigma}\subset \Omega^1(S,\g_P)$ to $\mathrm{Im}(\ad_{\Phi})^{-\sigma}\subset \Omega^2(S,\g_P)$. Let $Y:\mathrm{Im}(\ad_{\Phi})^{-\sigma} \to \mathrm{Im}(\ad_{\Phi^*})^{\sigma}$ denote the inverse of this map. We can express $\dot{A}\in\Omega^1(S,\g_P)^{\sigma,\rho}$ in terms of $Y$, using Corollary \ref{Coro:sigma-decompo-transverse}.

\begin{lemma} 
The variation of the connection $d_A$ is given by
\[\dot{A} = - Y(d_A\dot\Phi) + Y(d_A\dot\Phi)^*.\]
\end{lemma}

It follows that the total change in the curvature $F(\Phi)$ is
\[dF_{\Phi}(\dot\Phi) = -d_A Y(d_A\dot\Phi) + d_A Y(d_A\dot\Phi)^* + [\dot\Phi \wedge \Phi^*] + [\Phi \wedge \dot\Phi^*].\]

All we will use here is that this is a second order operator which varies continuously (in the operator norm topology) with respect to variations of $\Phi$ in the $C^0$-topology. This is because the linear maps $Y$ depend continuously on $\Phi$.

Since $dF_\Phi$ is a second order operator, it extends to a continuous map $H_k(\mathrm{Var}(\Phi)) \to H_{k-2}(\Lambda^2T^*S\otimes\mf{g}^{\rho,\sigma})$ where $H_k$ denotes the Sobolev--Hilbert space of sections with $k$ square-integrable derivatives, where $k\in\mathbb{N}$ with $k\geq 2$. The Sobolev embedding theorem tells us that in two dimensions, $H_k$ is in $C^0$. If we equip $\mathcal{P}$ with the $H_k$-topology, and $\Omega^2(S\mf{g}^{\sigma,\rho})$ with the $H_{k-2}$-topology, then $F$ is continuously differentiable. A Sobolev inequality tells us that multiplication $H_k \times H_k\to H_k$ is continuous in dimension 2, so the map $(\Phi,\eta)\mapsto e^{-\eta}\Phi e^{\eta}$ is continuous (and continuously differentiable). Finally, by Theorem \ref{Thm:ellipticity} we know that the partial derivative of $f$ with respect to $\eta$ is an isomorphism. All the axioms of the implicit function theorem are satisfied, so we get the following:

\begin{proposition}\label{prop:openness}
    Let $k\geq 2$. The set of Fock fields $\Phi$ of class $H_k$, such that there exists $\eta \in H_k(S,\mf{g}_P^{\sigma,-\rho})$ with flat canonical connection associated to $e^{-\eta}\Phi e^\eta$, is open. Furthermore, the map from $\Phi$ to this solution $\eta$ is differentiable.
\end{proposition}

Note that solutions $\eta$ for a given $\Phi$ are isolated, but at this point we have not proven them to be unique.

Since the operator $L$ is elliptic in the variable $\eta$, and continuously differentiable with respect to the $C^0$-topology (thus also the $H_k$-topology,) it can be approximated by its linearization for small continuous fluctuations of $\eta$. Standard elliptic regularity arguments then imply that if $\Phi$ is smooth, then $\eta$ must be as well.

\subsection{Map to the Hitchin component}

Now we restrict to $\mathrm{SL}_n(\C)$ to get a link from a neighborhood of the Fuchsian locus in the moduli space of higher complex structures $\mathcal{T}^n(S)$ to the Hitchin component.

Proposition \ref{prop:openness} gives us a map from a neighborhood $U\subset \mathbb{M}^n(S)$ of the Fuchsian locus in the space of higher complex structures to the Hitchin component. In Section \ref{Sec:higher-diffeos} we will see that this map is locally constant along higher diffeomorphism orbits. Unfortunately, this is not quite enough to conclude that we have a canonical map from a neighborhood of the Fuchsian locus in $\mathcal{T}^n(S)$ to the Hitchin component, because we can not guarantee that the intersection of $U$ with higher diffeomorphism orbits is connected.

Luckily there is a convenient slice of the higher diffeomorphism action, namely the harmonic higher complex structures defined in \cite{Nolte}. Let $\mathcal{H}\mathbb{M}^n(S)$ denote the space of harmonic  higher complex structures of order $n$ on $S$. A harmonic higher complex structure amounts to a complex structure, and list of tensors $(\mu_{3},...,\mu_{n})$ with $\mu_k$ a section of $K^{1-n}\otimes \bar{K}$ satisfying $\bar\partial(\bar{\mu}_k g_S^{k-1})=0$, where $g_S$ denotes the hyperbolic metric on $S$ associated to the complex structure. 
Theorem 8.2 in \cite{Nolte} states that any higher complex structure modulo higher diffeomorphism allows a harmonic representative, unique up to isotopy: $$\mathbb{M}^n(S)/\mathrm{Ham}_0(T^*S) \cong \mathcal{H}\mathbb{M}^n(S)/\mathrm{Diff}_0(S).$$
There is a natural action of $\R^*$ on higher complex structures by
\[t\cdot(\mu_{3},...,\mu_{n}) = (t\mu_3, ..., t^{n-2}\mu_n)\]
coming from scalar multiplication by $t$ in $T^*S$. This action clearly preserves $\mathcal{H}\mathbb{M}^n(S)$. Let $\mathcal{H}\mathbb{M}^{n}_{*}(S) \subset \mathcal{H}\mathbb{M}^n(S)$ denote the subset for which there is a path of solutions to (\ref{Eq:Hitchin-like-3}), for $\{(t\mu_{3},...,t^{n-2}\mu_{n}) \vert \,\,   t\in [0,1]\}$ starting at the Fuchsian solution. Proposition \ref{prop:openness} gives that this path of solutions is unique if it exists, so we have a natural map from $\mathcal{H}\mathbb{M}^{n}_{*}(S)$ to the Hitchin component via evaluating at the end of the path. 

Proposition \ref{prop:openness} also implies that $\mathcal{H}\mathbb{M}^{n}_{*}(S)$ is open. Also, note that $\mathcal{H}\mathbb{M}^{n}_{*}(S)$ is necessarily $\mathrm{Diff}_0(S)$-invariant because we can pull-back solutions by diffeomorphisms. We can define $\mathcal{T}^n_{*}(S)\subset \mathcal{T}^n(S)$ to be the quotient of $\mathcal{H}\mathbb{M}^{n}_{*}(S)$ by $\mathrm{Diff}_0(S)$. The map from $\mathcal{H}\mathbb{M}^{n}_{*}(S)$ to the Hitchin component factors through $\mathcal{T}^n_{*}(S)$ because diffeomorphisms isotopic to the identity cannot change monodromy. Therefore, we have the following result:
\begin{theorem} \label{Thm:neighborhood}
    There is an open neighborhood of the Fuchsian locus, namely $\mathcal{T}^n_{*}(S)$, which has a canonical map to the Hitchin component.
\end{theorem}
Of course it is the nature of the map that is interesting, not its mere existence. We come back to this theorem in \Cref{Sec:comparison-nonab-Hodge}, where we prove that it gives a local diffeormorphism.

\section{Higher diffeomorphisms}\label{Sec:higher-diffeos}

In this section, we exhibit a gauge-theoretic interpretation for infinitesimal higher diffeomorphisms. These can be described by special $\lambda$-dependent gauge transformations acting on a family of flat 3-term connections $\lambda^{-1}\Phi+d_A+\lambda\Psi$. For $\Psi=\Phi^{*}$, this action does not change the associated point in the Hitchin component since a gauge transformation does not change the monodromy.

\subsection{Special $\lambda$-dependent gauge transformations}

Consider a $G$-Fock bundle $(P,\Phi,\sigma)$ equipped with a second Fock field $\Psi$.
Assume the existence of a connection $d_A$ on $P$ such that 
$$\mathcal{A}(\lambda)=\lambda^{-1}\Phi+d_A+\lambda\Psi$$
is flat for all $\lambda \in \mathbb{C}^*$.
The space of all compatible triples $(\Phi,\Psi,d_A)$ on $(P,\sigma)$ such that the associated connection is flat is denoted by $\mathcal{C}^{fl}(P,\sigma)$.

Consider a 3-term infinitesimal gauge transformation
$$\eta(\lambda)=\l^{-1}\eta_{-1}+\eta_0+\l\eta_1 \;\text{ where }\; \eta_{-1},\eta_0,\eta_1 \in \Omega^0(S,\g_P).$$
The action on $\mathcal{A}(\l)$ is given by
\begin{align}\label{Eq:3-term-gauge-action}
    \eta(\l).\mathcal{A}(\l) &= d\eta(\l)+[\mathcal{A}(\l),\eta(\l)] \nonumber\\
    &= \l^{-2}[\Phi,\eta_{-1}]+\l^{-1}(d_A(\eta_{-1})+[\Phi,\eta_0])+d_A(\eta_0)+[\Phi,\eta_1]+[\Psi,\eta_{-1}] \nonumber\\
    & \;\;\;+ \l(d_A(\eta_1)+[\Psi,\eta_0])+\l^2[\Psi,\eta_1]. 
\end{align}

The action of $\eta_0$ being the usual complex gauge action, we consider $\eta_0=0$ here to concentrate on the new $\l$-dependent gauge action. In order to preserve the space of 3-term connections, the lowest and highest term in $\l$ have to vanish. This gives that $[\Phi,\eta_{-1}]=0$ and $[\Psi,\eta_{1}]=0$. Hence $\eta_{-1}\in Z(\Phi)$ and $\eta_1\in Z(\Psi)$.
The variation of $\Phi$ is then given by
\begin{equation}\label{variation-phi}
   \delta \Phi = d_A(\eta_{-1}). 
\end{equation}
Similarly we get $\delta\Psi =d_A(\eta_1)$ and $\delta A = [\eta_{-1},\Psi]+[\eta_1,\Phi]$.
Note that the variation $\delta A$, coming from the gauge transformation, is the same as the one induced from the filling-in procedure (Theorem \ref{Thm-filling-in}), since gauge transformations preserve curvature.

\begin{definition}
    A \emph{special $\l$-dependent gauge} of type $(P,\Phi,\Psi)\in\mathcal{C}^{fl}(P,\sigma)$ is an element of $\Omega^0(S,\g_P\otimes \C[\lambda^{\pm 1}])$ of the form $\l^{-1}\eta_{-1}+\l\eta_1$ such that $\eta_{-1}\in Z(\Phi)$ and $\eta_1\in Z(\Psi)$.
\end{definition}
Note that the notion of special $\l$-dependent gauge depends explicitly on $\Phi$ and $\Psi$. Considering Fock fields up to these transformations gives equivalence classes, but this equivalence relation does not come from a group action.

In the unitary setting, where $(P,\Phi)$ is equipped with a compatible hermitian structure $\rho$ and $\Psi=\Phi^{*}=-\rho(\Phi)$, we ask for $\eta_1=\eta_{-1}^*$ in order to preserve $\rho$.

\begin{proposition}
Variations coming from special $\lambda$-dependant gauge transformations are tangent to $\mathcal{C}^{fl}(P,\sigma)$.
\end{proposition}

\begin{proof}
    Let $\mathcal{A}(\l)\in\mathcal{C}^{fl}(P,\sigma)$. Since gauge transformations preserve curvature, the flatness of $\mathcal{A}(\l)$ remains. We have also already seen that the space of 3-term connections is preserved. It remains to show that $\Phi$ and $\Psi$ stay Fock fields. It is enough to prove it for $\Phi$. Let us show that $\sigma$ negates the variation $\delta \Phi=d_A(\eta_{-1})$. Since $\eta_{-1}\in Z(\Phi)$, it is negated by $\sigma$. The canonical connection $d_A$ being $\sigma$-invariant, we get that $\delta\Phi$ is negated by $\sigma$.

    The only remaining point is that the nilpotency condition is preserved. This is a bit more delicate and needs Lemma \ref{lemma:im-incl}. Consider a local chart in which $\Phi=Fdz+\Phi_2d\bar{z}$, where $F$ is a fixed principal nilpotent element. We have to show that $\delta\Phi=d_A(\eta_{-1})\in \mathrm{Var}(\Phi)$. By Proposition \ref{Prop:Phi-variation} we know that $\mathrm{Var}(\Phi)\cong \mathrm{Im}(\ad_\Phi)\oplus Z(\Phi)\bar{K}$. The (1,0)-part of $d_A(\eta_{-1})$ is $\partial \eta_{-1}+[A_1,\eta_{-1}]$. We have $\partial \eta_{-1}\in Z(F)$ and $Z(F)\subset \mathrm{Im}(\ad_F)$ since $Z(F)$ describes the space generated by lowest weight vectors. From $[\eta_{-1},\Phi]=0$, we get $\eta_{-1}\in Z(F)$, hence $[A_1,\eta_{-1}]\in \mathrm{Im}(\ad_{\eta_{-1}})\subset \mathrm{Im}(\ad_F)$ by Lemma \ref{lemma:im-incl}. Therefore, locally there is $R\in\Omega^0(S,\g_P)$ such that $[R,F]=\partial \eta_{-1}+[A_1,\eta_{-1}]$.

    We show that $d_A(\eta_{-1})-[R,\Phi]\in Z(\Phi)\bar{K}$. By definition of $R$ this difference has no (1,0)-part. It is sufficient to prove $\delbar \eta_{-1}+[A_2,\eta_{-1}]-[R,\Phi_2]\in Z(F)$. We compute
\begin{equation}\label{Eq:nilp-aux}
[F,\delbar \eta_{-1}+[A_2,\eta_{-1}]-[R,\Phi_2]] = [\eta_{-1},[A_2,F]]-[\Phi_2,[R,F]],
\end{equation}
where we used $[F,\eta_{-1}]=0$ (hence also $[F,\delbar \eta_{-1}]=0$), the Jacobi identity and $[F,\Phi_2]=0$. Now the identity $d_A\Phi=0$ gives
$$[A_2,F]=\delbar \Phi_2+[A_1,\Phi_2].$$
Together with $[R,F]=\partial \eta_{-1}+[A_1,\eta_{-1}]$ (from definition of $R$), we continue Equation \eqref{Eq:nilp-aux}:
$$[\eta_{-1},[A_2,F]]-[\Phi_2,[R,F]]=[\eta_{-1},\delbar \Phi_2]-[\Phi_2,\partial \eta_{-1}]=0,$$
where some cancellation happened and we use that $[\eta_{-1},\Phi_2]=0$.
\end{proof}
 
\begin{remark}
    At first glance it might be surprising that the action is only defined on the space of \emph{flat} 3-term connections. In fact, the condition $d_A\Phi=0$ is needed in order to get a variation $\delta\Phi$ which preserves $[\Phi\wedge\Phi]=0$, i.e. which satisfies $[\Phi\wedge\delta\Phi]=0$. Indeed,  since $\delta\Phi=d_A(\eta_{-1})$ and $[\Phi,\eta_{-1}]=0$, we get $$[\Phi\wedge\delta\Phi]=[\Phi\wedge d_A(\eta_{-1})]=[d_A(\Phi),\eta_{-1}]=0.$$
The condition $F(A)+[\Phi\wedge\Psi]=0$ assures the preservation of the condition $d_A(\Phi)=0$:
$$\delta(d_A(\Phi))= d_A(d_A(\eta_{-1}))+[([\eta_{-1},\Psi]+[\eta_1,\Phi]) \wedge \Phi]=[F(A)+[\Phi\wedge\Psi], \eta_{-1}]=0.$$
\end{remark}

The variation $\delta \Phi=d_A(\eta_{-1})$ seems to depend on the connection $d_A$. We show that this is not really the case:
\begin{proposition}
    Up to usual gauge transformations, the variation $\delta\Phi=d_A(\eta_{-1})$ does not depend on the $\Phi$-compatible connection $d_A$.
\end{proposition}
The main ingredient of the proof is Corollary \ref{coboundaries} about the vanishing of brackets in $\Phi$-cohomology.
\begin{proof}
   If $\lambda^{-1}\Phi+ d_B+\l\Psi$ is another flat 3-term connection, then we prove that there exists a usual gauge transformation $R$ such that
$d_A(\eta_{-1})-d_B(\eta_{-1}) = [R,\Phi].$
This means that the variation of $\Phi$ in \eqref{variation-phi} does not depend on $d_A$ modulo gauge transformations.

Note that $d_A\Phi=0$ and $d_B\Phi=0$ implies $[A-B,\Phi]=0$, and also $d_A(\eta_{-1})-d_B(\eta_{-1}) = [A-B,\eta_{-1}]$.
Hence in the language of $\Phi$-cohomology, $A-B$ is a 1-cocycle and $\eta_{-1}$ is a 0-cocycle. 
We have seen in Corollary \ref{coboundaries} that the bracket of cohomology classes of Fock fields is always a coboundary. This gives the existence of $R$. 
\end{proof}

\subsection{First variation formula}\label{Sec:higher-diff-action-via-gauge}

Now we restrict to the case of $\mathrm{SL}_n(\C)$-Fock bundles $(E,\Phi,g)$. We prove that the induced infinitesimal action of special $\l$-gauge transformations on the higher complex structure is given by the infinitesimal action by higher diffeomorphisms. 

Here we concentrate on the action on $\Phi$ and consider a $\l$-dependent gauge transformation $\l^{-1}\xi$ (we slightly change notation to the previous subsection where $\xi=\eta_{-1}$), where $\xi\in\Omega^0(S,\mathfrak{sl}(E))$ with $[\Phi,\xi]=0$. This implies that $\xi$ is a polynomial in $\Phi$, evaluated in various vector fields. Recall the variation of $\Phi$ from Equation \eqref{variation-phi}:
$$\delta \Phi = d_A\xi.$$

\begin{theorem}\label{hamiltonian-first-variations-coincide}
The variation on a Fock field $\Phi$ induced by an infinitesimal gauge transformation $\lambda^{-1} \xi$ with $\xi=\Phi(v_1)\cdots \Phi(v_k)$ is equivalent to the infinitesimal action of the Hamiltonian $H=v_1\cdots v_k$ on the higher complex structure induced by $\Phi$.
\end{theorem}

The proof is a direct computation using local coordinates and the condition $d_A\Phi=0$.
\begin{proof} The statement is local. Hence consider a coordinate system $(z,\bar{z})$ on $S$ and a gauge fixing in which the Fock field $\Phi$ is $\Phi=F \, dz+Q(F)d\bar{z}$, where $F$ is a fixed principal nilpotent element and $Q$ is a polynomial without constant term. We can decompose $\xi=\Phi(v_1)\cdots \Phi(v_k)$ into monomials in $F$. So let us suppose that $\xi=w_k(z,\bar{z})F^k$. We then get
$$\delta \Phi = d_A\xi = d\xi+[A,\xi] = \left(\partial w_k F^k+w_k[A_1,F^k]\right)dz+\left(\delbar w_k F^k+w_k[A_2,F^k]\right)d\bar{z},$$
where we used $A=A_1dz+A_2d\bar{z}$. With $\Phi=\Phi_1dz+\Phi_2d\bar{z}$, we get
\begin{equation}\label{var-phi-1-2}
\delta \Phi_1 = \partial w_k F^k+w_k[A_1,F^k] \;\;\text{ and }\;\; \delta \Phi_2 = \delbar w_k F^k+w_k[A_2,F^k].
\end{equation}
From $\Phi_2=Q(F)=\mu_2 F+\mu_3F^2+...+\mu_nF^{n-1}$, we get
$$\delta \Phi_2 = \sum_{k=2}^n \delta \mu_k \, F^{k-1} + \mu_2\delta F+\mu_3(\delta F \,F+F\delta F)+...+\mu_n\sum_{\ell=0}^{n-2} F^\ell\delta F\, F^{n-2-\ell}.$$
Hence using Equation \eqref{var-phi-1-2} one has
\begin{equation}\label{aux-1}
\sum_{k=2}^n \delta \mu_k F^{k-1} = \delbar w_k F^k+w_k[A_2,F^k]-\sum_{m=2}^n\mu_m \sum_{\ell=0}^{m-2}F^\ell\left(\partial w_k F^k+w_k[A_1,F^k]\right)F^{m-2-\ell}.
\end{equation}
The flatness condition now gives
\begin{align*}
0=d_A\Phi &= -\delbar \Phi_1+\partial\Phi_2 +[A_1,\Phi_2]+[\Phi_1,A_2] \\
&=\sum_{m=2}^n \partial\mu_m F^{m-1}+[A_1,Q(F)]-[A_2,F].
\end{align*}
We may thus deduce 
\begin{equation}\label{aux-2}
[A_2,F^k]=\sum_{\ell=0}^{k-1}F^\ell[A_2,F]F^{k-1-\ell}=kF^{k-1}\sum_{m=2}^n \partial\mu_m F^{m-1}+\sum_{m=2}^n\sum_{j=0}^{k-1}\mu_m F^j[A_1,F^{m-1}]F^{k-1-j}.
\end{equation}
The last part of the expression needs some more manipulation:
\begin{align*}
\sum_{j=0}^{k-1} F^j[A_1,F^{m-1}]F^{k-1-j} &= \sum_{j=0}^{k-1}\sum_{\ell=0}^{m-2}F^{j+\ell}[A_1,F]F^{m-2-\ell}F^{k-1-j}\\
&= \sum_{\ell=0}^{m-2}F^\ell[A_1,F^k]F^{m-2-\ell}.
\end{align*}
Combining this last equation with Equations \eqref{aux-2} and \eqref{aux-1}, we get
$$\sum_{k=2}^n \delta \mu_k F^{k-1} = \delbar w_k \, F^k+kw_kF^{k-1}\sum_{m=2}^n \partial\mu_m F^{m-1}-\partial w_k\, F^k\sum_{m=2}^{n}(m-1)\mu_mF^{m-2}.$$
This gives exactly the first variation formula for higher complex structures. In order to see this, compare to the Poisson bracket expression (see \cite[Section 3.3]{FT}):
\begin{align*}\sum_{k=2}^n \delta \mu_k \, p^{k-1} &= \{w_kp^k,-\bar{p}+\mu_2 p+\mu_3 p^2+...+\mu_np^{n-1}\} \\
&= \delbar w_k \,p^k+kw_kp^{k-1}\sum_{m=2}^n\partial\mu_m\, p^{m-1}-\partial w_k\, p^{k}\sum_{m=2}^n (m-1)\mu_m p^{m-2}.
\end{align*}
This completes the proof of the theorem.
\end{proof}

\begin{remark}
For a higher complex structure with Beltrami differentials $\mu_k=0$, for all $k\in\{2,...,n\}$, the argument above can be simplified to 
$$\delta \mu_2\, F+...+\delta\mu_n\, F^{n-1} = \delbar w_k\, F^k+w_k[A_2,F^k]$$
and the flatness condition provides that $[A_2,F^k]=\delbar \Phi_1=0$.
\end{remark}

\subsection{Constant monodromy along paths of higher diffeomorphisms}

In the previous subsection, we lifted infinitesimal higher diffeomorphisms to 3-term gauge transformations. Now, we show that under the right circumstances, a path of flat 3-term connections inducing higher diffeomorphic higher complex structures must be obtained from a path of 3-term gauge transformations. This shows that within the subset of higher complex structures where canonically associated flat 3-term connections exist, the monodromy is locally constant along higher diffeomorphism orbits. We prove this statement in the setting with hermitian structure, though it surely has an analogue for more general 3-term connections. 

\begin{proposition}\label{prop:constancy} 
Suppose $\lambda^{-1}\Phi_t + d_{A_t} + \lambda\Phi_t^{*}$ is a differentiable family of flat 3-term connections on a vector bundle $E$ with compatible symmetric pairing and hermitian metric, and $\Phi_t$ is a family of positive $\mathrm{SL}_n(\C)$-Fock fields inducing a path of higher diffeomorphic higher complex structures. Then the $3$-term connections $\lambda^{-1}\Phi_t + d_{A_t} + \lambda\Phi_t^{*}$ are all gauge equivalent. 
\end{proposition}

\begin{proof}
The fact that the higher complex structures are higher diffeomorphic gives us a time dependant Hamiltonian $h_t:T^*S \to \C$ which satisfies 
\[\frac{d}{dt} \Phi_t = d_{A_t}(h_t(\Phi_t)) + [\Phi_t,\eta_t],\]
where $\eta_t$ is a path in $\Omega^0(S,\mf{g}_P^\sigma)$. This follows from Theorem \ref{hamiltonian-first-variations-coincide}. Decompose $\eta_t$ into its hermitian and anti-hermitian parts $\eta_t=\eta_t^{-\rho}+\eta_t^{\rho}$.
If $\eta_t^{-\rho} = 0$, then we are done because the path of three term infinitesimal gauge transformations 
\[\lambda^{-1}h_t(\Phi_t) + \eta_t + \lambda h_t(\Phi_t)^{*}\]
induces the correct time derivative of $\Phi_t$, thus must induce the correct derivative of the whole 3-term family. 

At each time $t$, the derivative of curvature $F_{A_t} + [\Phi_t\wedge\Phi_t^{*}]$ with respect to a change in $\Phi_t$ is a linear operator. We have just argued that the derivative of curvature if we change $\Phi$ by $d_{A_t}(h_t(\Phi_t)) + [\Phi_t,\eta_t^{\rho}]$ is zero, so the only thing that could be changing curvature is $[\Phi_t,\eta_t^{-\rho}]$. The change in curvature induced by a hermitian gauge transformation of $\Phi$ is the operator $L$ from Theorem \ref{Thm:ellipticity}. Since we are assuming our connections to be flat, we have $L(\eta_t^{-\rho})=0$. This implies $\eta_t^{-\rho} = 0$ because $L$ has zero kernel.
\end{proof}

If we assume unique solutions to Equation \eqref{Eq:Hitchin-like-3} for all Fock bundles, then Proposition \ref{prop:constancy} tells us that the map from Fock bundles to the Hitchin component factors through the projection to $\mathcal{T}^n(S)$. Note the importance of the higher diffeomorphism group being connected.

\subsection{Comparison with nonabelian Hodge along the Fuchsian locus}\label{Sec:comparison-nonab-Hodge}

In this subsection, we show the map we have defined from an open neighborhood of the Fuchsian locus to the Hitchin component is a local diffeomorphism along the Fuchsian locus by comparing with the map from nonabelian Hodge theory. 

\begin{theorem} The map from Theorem \ref{Thm:neighborhood} is a local diffeomorphism in some neighbourhood of the Fuchsian locus. 
\end{theorem}

There is an identification of harmonic Fock bundles with Higgs bundles given by
\[\Phi_1 + \Phi_2 \mapsto \Phi_1 + \Phi_2^{*_h}.\]
Here, we have the bundle $E = K^{(n-1)/2} \oplus \cdots \oplus K^{(1-n)/2}$, $h$ is the diagonal Hermitian metric induced from the hyperbolic metric $g$, $\Phi_1$ is the standard holomorphic, principal-nilpotent section of $\End(E)\otimes K$, and $\Phi_2 = \mu_3 \Phi_1^2 + \dots + \mu_n \Phi_1^{n-1}$ where $\mu_k$ is a section of $K^{-k+1}\otimes\bar{K}$ satisfying $\delbar(g^{k-1} \mu_k) = 0$. This last harmonicity condition means exactly that $\Phi_2^{*_h}$ is holomorphic, so $\Phi_1+\Phi_2^{*_h}$ is indeed a Higgs field. 
This identification of harmonic Fock bundles with Higgs bundles in the Hitchin section lets us compare the nonabelian Hodge correspondence with the partially defined map from $\mathcal{T}^n(S)\to \Hit^n(S)$ we defined using  Proposition \ref{prop:openness}. The two maps $\mathcal{T}^n(S) \supset U \to \Hit^n(S)$ clearly agree on the Fuchsian locus where $\Phi_2=0$. We will show that their first derivatives agree transverse to the Fuchsian locus. Since the nonabelian Hodge map is known to give a diffeomorphism, this shows that our map is a local diffeomorphism in a neighborhood of the Fuchsian locus. 

Consider the family of Fock fields $\Phi(t) = \Phi_1 + t \Phi_2$, for $t\in \R$. By Proposition \ref{prop:openness}, in a neighborhood of $t=0$ there is a gauge transformation $g(t) = I + t g_1 + \mathcal{O}(t^2)$ such that $g(t)^{-1}\Phi(t)g(t)$ is a path of Fock fields satisfying the flatness equation. 

Similarly, by the nonabelian Hodge correspondence, there is a family $f(t) = I + tf_1 + \mathcal{O}(t^2)$ of gauge transformations which transform the family of Higgs fields $\phi(t) := \Phi_1+t\Phi_2^{*_h}$ to solutions to the Hitchin equation. 

The main observation is that $g_1$ and $f_1$ above are both zero. This is because any non-zero $g_1$ or $f_1$ would give rise to a non-zero order-1 term in the curvature, namely
\[2 \del_A\delbar_A g_1 + [[\Phi_1, g_1]\wedge\Phi_1^{*_h}] - [\Phi_1\wedge[\Phi_1^{*_h}, g_1]],\]
where $A$ is the Chern connection on $E$, in the Fock bundle case, and in fact the same formula in the Higgs bundle case.
This is the same computation as in Theorem \ref{Thm:ellipticity}. 

Since the first order terms of $g$ and $f$ vanish, it follows that the derivative with respect to $t$ of both resulting flat connections on $E$ is simply $\Phi_2 + \Phi_2^{*_h}$. It follows that the first derivatives of the two monodromy representations, and thus the two points in the Hitchin component, are equal.

\section{Covectors}\label{Sec:covectors}

In the previous sections we have obtained a Mapping Class Group equivariant map from an open neighborhood of the Fuchsian locus in the moduli space of higher complex structures of order $n$ to the $\mathrm{PSL}(n, \mathbb{R})$-Hitchin component. To study the complex character variety, in this final section we relax the requirement that the canonical connection $d_A$ be $\sigma$-invariant. We will see that $d_A$ is then determined by cotangent vectors to the space of higher complex structures. This also leads to a gauge-theoretic meaning of the $\mu$-holomorphicity condition.

\subsection{Definition of covectors}

The setting of our investigations is now the following. Consider a $G$-Fock bundle $(P,\Phi,\sigma)$ on the surface $S$ equipped with a hermitian structure $\rho$. \emph{The main difference with the previous sections is that $\rho$ is not supposed to commute with $\sigma$.}
We assume $\Phi\in\mathcal{P}$ to be a positive Fock field (see Definition \ref{Def:pos-Higgs-field}). This implies that $\Phi$ and $\Phi^{*}=-\rho(\Phi)$ are transverse. 

We have seen in Proposition \ref{Prop:Phi-variation} that the tangent space to the space of Fock bundles modulo usual gauge transformations is given by $Z(\Phi)\bar{K}$. This motivates the following definition of a covector:
\begin{definition}
    A \emph{covector} for a $G$-Fock bundle $(P,\Phi)$ is an element of $(Z(\Phi)\bar{K})^*$.
\end{definition}
We will give several realizations of covectors. For $\mathrm{SL}_n(\C)$, we will see in Section \ref{Sec:mu-holo-meaning} below the link to the cotangent bundle to higher complex structures.

\begin{lemma}\label{lemma:contraction} 
For any $\Phi\in  \mathcal{P}$ we have:
\[\mf{g}_P\otimes T^*S = \mathrm{Im}(\ad_\Phi) \oplus \mathrm{Im}(\ad_{\Phi^{*}}) \oplus Z(\Phi) \bar{K} \oplus Z(\Phi^*) K.\]
\end{lemma}
The same statement holds for $\Omega^1(S,\g_P)$, where we use $\mathrm{Im}(\ad_\Phi)$ both for the subbundle and its sections. The proof uses a fixed point argument coming from a contraction.
\begin{proof}
We do the proof at a point $p\in S$, so we may as well choose an identification of $(\mf{g}_P)_p$ with $\mf{g}$, and a local coordinate so we get a basis $dz,d\bar{z}$ of $T^*S\otimes \C$. Both vector spaces in the lemma have the same dimension, namely $2\mathrm{dim}(\g)$. All we have to show is that there is a direct sum on the right hand side.

Suppose we have a sum \[[\Phi,A] + [\Phi^{*},B] + C d\bar{z} + D dz = 0,\] where $A,B,C,D\in \mathfrak{g}$ and $C,D$ are in the centralizers of $\Phi$ and $\Phi^{*}$ respectively. We want to show that the only solution to this equation is the trivial solution $[\Phi,A]=[\Phi^{*},B]=C d\bar{z}=D dz=0$. Split our linear relation into $dz$ and $d\bar{z}$ parts:
\[ [\Phi_1, A] + [\Phi_2^*, B] + D = 0,\]
\[ [\Phi_1^*, B] + [\Phi_2, A] + C = 0.\]
Note that $\mf{g}$ is an orthogonal direct sum of $\mathrm{Im}(\ad_{\Phi_1})$ and $Z(\Phi_1^*)$, and likewise for $\mathrm{Im}(\ad_{\Phi_1^*})$ and $Z(\Phi_1)$. This means, for example, that $[\Phi_1, A]$ is $-\pi_{\mathrm{Im}(\Phi_1)}([\Phi_2^*, B])$, where $\pi_{\mathrm{Im}(\Phi_1)}$ is orthogonal projection onto $\mathrm{Im}(\ad_{\Phi_1})$. Let $c_\Phi:\mathrm{Im}(\ad_{\Phi_1})\to \mathrm{Im}(\ad_{\Phi_2})$ denote the map which takes $[\Phi_1, A]$ to $[\Phi_2, A]$ for any $A$. Define $c_{\Phi^{*}}$ analogously. We see that $[\Phi_1,A]$ has to be a fixed point of the map
\[(-\pi_{\mathrm{Im}(\Phi_1)}) \circ c_{\Phi^{*}} \circ (-\pi_{\mathrm{Im}(\Phi^{*}_1)}) \circ  c_{\Phi}.\]
Recall that the positivity condition on $\Phi\in\mathcal{P}$ provides that $\|[\Phi_1, A]\|^2 > \|[\Phi_2, A]\|^2$ for all $A\in \mf{g}$, meaning precisely that $c_\Phi$ (and thus $c_{\Phi^{*}}$) decreases norm.

Orthogonal projections do not increase norm, so this composition must decrease norm. This implies that $[\Phi_1,A]$ must be zero. We immediately get that $[\Phi,A]=[\Phi^{*},B]=0$. The centralizers of $\Phi$ and $\Phi^{*}$ have trivial intersection because they are nilpotent preserving opposite flags, so $C$ and $D$ are zero as well.
\end{proof}

An equivalent way to phrase this lemma is as follows: for positive Fock field $\Phi$ we have
\begin{equation}
\Omega^1(S,\g_P)= \mathrm{Var}(\Phi)\oplus \mathrm{Var}(\Phi^{*}),
\end{equation}
where we used $\mathrm{Var}(\Phi)=\mathrm{Im}(\ad_\Phi)\oplus Z(\Phi)\bar{K}$ from Proposition \ref{Prop:Phi-variation}.

Recall from Section \ref{Sec:var-Fock-fields} the symplectic form $\omega=\int_S \tr\alpha\wedge\beta$ on $\Omega^1(S,\g_P)$. Both subspaces $\mathrm{Var}(\Phi)$ and $\mathrm{Var}(\Phi^*)$ are maximal isotropic with respect to $\omega$. Pointwise they give a decomposition of $T^*S\otimes \g$ into two Lagrangian subspaces. Therefore, pairing with $\omega$ gives the duality $\mathrm{Var}(\Phi)\cong \mathrm{Var}(\Phi^{*})^*$.

The space of covectors is related to the cohomology group $\mathrm{H}^1(\Phi)$. From $\Omega^1(S,\g_P)=\mathrm{Var}(\Phi)\oplus\mathrm{Var}(\Phi^{*})$ and $\mathrm{Var}(\Phi)\subset \ker\,\ad_\Phi$, we get $\ker\,\ad_\Phi=\mathrm{Var}(\Phi)\oplus (\ker\,\ad_\Phi\cap \mathrm{Var}(\Phi^{*}))$. Hence from $\mathrm{Var}(\Phi)=Z(\Phi)\bar{K}\oplus \mathrm{Im}(\ad_\Phi)$ we deduce
\begin{equation}\label{Eq:h1-aux}
\mathrm{H}^1(\Phi)\cong Z(\Phi)\bar{K}\oplus (\ker\,\ad_\Phi\cap \mathrm{Var}(\Phi^{*})).
\end{equation}
The latter can be identified with the space of covectors:
\begin{proposition}
    The symplectic pairing induces an isomorphism between $(Z(\Phi)\bar{K})^*$ and $\ker\,\ad_\Phi\cap\mathrm{Var}(\Phi^{*})$.
\end{proposition}
\begin{proof}
    Both spaces have the same dimension $\mathrm{rk}(\g)$. This follows from $\dim \mathrm{H}^1(\Phi)=2\mathrm{rk}(\g)$ (see Proposition \ref{Prop:dim-phi-cohom}) and Equation \eqref{Eq:h1-aux}. Given a non-zero element $x\in \ker\,\ad_\Phi\cap\mathrm{Var}(\Phi^{*})$, there is $z\in\Omega^1(S,\g_P)$ such that $\omega(x,z)\neq 0$ since $\omega$ is non-degenerate. The symplectic pairing between $\ker\,\ad_\Phi\cap\mathrm{Var}(\Phi^{*})$ and $\mathrm{Var}(\Phi^{*})\oplus \mathrm{Im}(\ad_\Phi)$ is trivial, hence we can reduce to $z\in Z(\Phi)\bar{K}$ by Lemma \ref{lemma:contraction}. Therefore, the map $x\mapsto \omega(x,.)$ from $\ker\,\ad_\Phi\cap\mathrm{Var}(\Phi^{*})$ to $(Z(\Phi)\bar{K})^*$ is injective, thus an isomorphism.
\end{proof}
Combining this proposition with Equation \eqref{Eq:h1-aux}, we get
$$\mathrm{H}^1(\Phi)\cong Z(\Phi)\bar{K}\oplus (Z(\Phi)\bar{K})^*.$$
This shows that $\mathrm{H}^1(\Phi)$ describes variations of $\Phi$ and a covector modulo gauge transformations.

\subsection{Canonical connection with covectors}

We want to equip $P$ with a unitary $\Phi$-compatible connection $d_A$, generalizing the filling-in procedure from Theorem \ref{Thm-filling-in}, to get a canonical 3-term connection $\Phi+d_A+\Phi^{*}$.

Denote by $\Pi$ the space of all unitary connections which are $\Phi$-compatible (solutions to $d_A\Phi=0$). Note that such connections are automatically $\Phi^{*}$-compatible.
Since $\Phi$ and $\Phi^{*}$ are transverse, $\Pi$ is non-empty and is modeled on $\rho$-invariant elements of $\ker\,\ad_\Phi\cap\ker\,\ad_{\Phi^{*}}$.
We start by analysing the vector space $(\ker\,\ad_\Phi\cap\ker\,\ad_{\Phi^{*}})^\rho$ and show that it is a realization of the space of covectors. 

\begin{proposition}
    The symplectic pairing gives an isomorphism
$$\ker(\ad_\phi)\cap\ker(\ad_{\Phi^{*}})\cong (Z(\Phi)\bar{K}\oplus Z(\Phi^{*})K)^*.$$
\end{proposition}
\begin{proof}
Both spaces have the same dimension. Indeed the right hand side is of dimension $\dim Z(\Phi)+\dim Z(\Phi^{*})=2\mathrm{rk}(\g)$. The left hand side can be identified with $\mathrm{H}^1(\Phi)$ which is of the same dimension by Proposition \ref{Prop:dim-phi-cohom}.   

The symplectic pairing between $\ker\,\ad_\Phi\cap\ker\,\ad_{\Phi^{*}}$ and $\mathrm{Im}(\ad_\Phi)\oplus\mathrm{Im}(\ad_{\Phi^{*}})$ is trivial by the cyclicity of the Killing form. Since $\omega$ is non-degenerate and using Lemma \ref{lemma:contraction}, the map from $\ker\,\ad_\phi\cap\ker\,\ad_{\Phi^{*}}$ to $(Z(\Phi)\bar{K}\oplus Z(\Phi^{*})K)^*$ given by $x\mapsto \omega(x,\cdot)$ is injective. Therefore it is an isomorphism.
\end{proof}

Since $\rho$ exchanges $Z(\Phi)\bar{K}$ with $Z(\Phi^{*})K$, the $\rho$-invariant part of $(Z(\Phi)\bar{K}\oplus Z(\Phi^{*})K)^*$ maps isomorphically to any of the two factors. Thus:

\begin{corollary}\label{coro:covector-para}
    The symplectic pairing induces an isomorphism between $(Z(\Phi)\bar{K})^*$ and $(\ker\,\ad_\Phi\cap\ker\,\ad_{\Phi^{*}})^\rho$.
\end{corollary}
Using this identification, we sometimes call an element of $(\ker\,\ad_\Phi\cap\ker\,\ad_{\Phi^{*}})^\rho$ a \emph{covector}.
In the case of $\mathrm{SL}_n(\C)$, the centralizer $Z(\Phi)=Z(\Phi_1)$ is generated by the powers of $\Phi_1$. Hence a covector $t\in (\ker\,\ad_\Phi\cap\ker\,\ad_{\Phi^{*}})^\rho$ is uniquely described by $(t_k=\tr \,\Phi_1^{k}t^{1,0})_{1\leq k\leq n-1}$.
This is similar to the formula of covectors for higher complex structures \cite[Proposition 4.2]{thomas-flat-conn}.

\medskip
Now we study the affine space $\Pi$ of unitary $\Phi$-compatible connections. 
For a connection $d_A\in\Pi$, we put $A^{-\sigma}=\tfrac{1}{2}(d_A-\sigma(d_A))\in \Omega^1(S,\g_P^{-\sigma})$. It is important to notice that $A^{-\sigma}\in \ker(\ad_\Phi)$. Indeed, this is the $\sigma$-invariant part of $d_A\Phi=0$. Note also that in general, since $\rho$ and $\sigma$ do not necessarily commute, $A^{-\sigma}\notin \Pi$.

Again we use the symplectic pairing with $Z(\Phi)\bar{K}$ and define the map
$$\psi: \left \{ \begin{array}{ccc}
\Pi &\rightarrow & (Z(\Phi)\bar{K})^* \\
d_A &\mapsto & \int_S\tr(A^{-\sigma}\,\cdot)
\end{array} \right.$$

\begin{proposition}\label{Prop:affine-iso}
    The map $\psi$ is an affine map covering the linear map from Corollary \ref{coro:covector-para}. Hence it is an isomorphism.
\end{proposition}
\begin{proof}
    For $d_A,d_B\in\Pi$, we have $d_A-d_B=A-B\in (\ker\,\ad_\Phi\cap\ker\,\ad_{\Phi^{*}})^\rho$. In particular $A-B\in \ker\,\ad_\Phi$. Since $A^{-\sigma}\in\ker\,\ad_\Phi$ we also have $A^{-\sigma}-B^{-\sigma}\in \ker\,\ad_\Phi$. By Proposition \ref{Prop-no-sigma-cohom} there is no $\sigma$-invariant $\Phi$-cohomology, so $[A-B]=[A^{-\sigma}-B^{-\sigma}]\in \mathrm{H}^1(\Phi)$. Therefore 
$$\psi(d_A)-\psi(d_B)=\int_S\tr((A^{-\sigma}-B^{-\sigma})\,\cdot)=\int_S\tr((A-B)\,\cdot),$$
where we used that the pairing with $Z(\Phi)\bar{K}$ does only depend on the cohomology class (since it pairs to zero with $\mathrm{Im}\,\ad_\Phi$).
\end{proof}

\begin{proposition}\label{Prop:covector}
Let $\mathcal{A}$ denote the space of all connections on $P$. The map $\mathcal{A}\times Z(\Phi)\bar{K}\to \C$ given by $(d_A,x)\mapsto \int_S\tr(A^{-\sigma}x)$ is independent of the choice of $\sigma$ (among $\sigma_0$-structure negating $\Phi$), and is gauge invariant.
\end{proposition}

\begin{proof}
For the first claim, we know from Proposition \ref{Prop:var-sigma} that an infinitesimal variation of a $\sigma_0$-structure negating $\Phi$ is described by $\mathrm{H}^0(\Phi)$. More precisely, a variation $\delta\sigma$ is described by $\delta\sigma(x)=[\xi,x]$ for all $x\in \g$ where $\xi\in \mathrm{H}^0(\Phi)=Z(\Phi)$. The change of $A^{-\sigma}=\frac{1}{2}(A-\sigma(A))$ is given by $\delta A^{-\sigma}=-\frac{1}{2}[\xi,A] \in \mathrm{Im}(\ad_\xi)$. Since $Z(\Phi)$ is abelian we get
$\int_S\tr([\xi,A]x)=\int_S\tr([x,\xi]A)=0$.

For the second claim, consider a gauge transformation $\eta\in\Aut(P)$. The action on $\Phi$ and $A^{-\sigma}$ are given by $\eta.\Phi=\eta\Phi \eta^{-1}$ and $\eta.A^{-\sigma}=\eta A^{-\eta.\sigma}\eta^{-1}$ where $\eta.\sigma$ is the pull-back of $\sigma$ along $\eta$. Note that $\eta.\sigma$ is compatible with $\eta.\Phi$. We distinguish two cases.

If $\eta.\sigma=\sigma$, i.e. $\eta\in\Aut(P,\sigma)$, the pairing stays unchanged since $Z(\eta\Phi\eta^{-1})=\eta Z(\Phi)\eta^{-1}$ and the Killing form is Ad-invariant. If $\eta.\Phi=\Phi$, we conclude by the first part since $\eta.\sigma$ is then $\Phi$-compatible.
Both cases generate all gauge transformations since $\g^\sigma$ is a maximal Lie subalgebra of $\g$.
\end{proof}

From all this, we get a canonical base point in $\Pi$, generalizing the filling-in procedure from Theorem \ref{Thm-filling-in}.

\begin{theorem}\label{Thm:gen-filling-in}
    Consider a $G$-Fock bundle $(P,\Phi,\sigma)$ equipped with a hermitian structure $\rho$. If $\Phi\in\mathcal{P}$ is positive, then there is a unique unitary $\Phi$-compatible connection $d_{A_0}$ such that $A_0^{-\sigma}\in\mathrm{Var}(\Phi)$. This is independent of the choice of $\sigma$. 
\end{theorem}
\begin{proof}
    We start from uniqueness: if $d_{A_0}\in\Pi$ is such that $A_0^{-\sigma}\in\mathrm{Var}(\Phi)$, then $\psi(d_{A_0})=0$ since $Z(\Phi)\bar{K}\subset \mathrm{Var}(\Phi)$ which is isotropic. By Proposition \ref{Prop:affine-iso}, $\psi$ is an isomorphism, which gives uniqueness.

    To show existence, consider $d_{A_0}=\psi^{-1}(\{0\})$. Then $\int_S\tr(A_0^{-\sigma}\,\cdot)$ vanishes both on $Z(\Phi)\bar{K}$ and on $\mathrm{Im}(\ad_\Phi)$ since $A_0^{-\sigma}\in\ker\,\ad_\Phi$. Hence $A_0^{-\sigma}\in\mathrm{Var}(\Phi)^{\perp_\omega}=\mathrm{Var}(\Phi)$.

    Finally, the independence of the choice of $\sigma$ follows directly from Proposition \ref{Prop:covector}.
\end{proof}

A more conceptual formulation of the previous theorem is given by:
\begin{corollary}
The space of unitary connections $d_A$ on $P$ such that $d_A \Phi=0$ is naturally the dual space to $T_\Phi( \mathcal{P})/\mathrm{Im}(\ad_\Phi)$.
\end{corollary}

\begin{proof}
By definition, we have $T_\Phi(\mathcal{P})=\mathrm{Var}(\Phi)\cong \mathrm{Im}(\ad_\Phi)\oplus Z(\Phi)\bar{K}$. Corollary \ref{coro:covector-para} shows that $(\ker\,\ad_\Phi\cap\ker\,\ad_{\Phi^{*}})^\rho$, which is identified with the space of unitary $\Phi$-compatible connections via Theorem \ref{Thm:gen-filling-in}, is dual to $Z(\Phi)\bar{K}$. 
\end{proof}

The generalized filling-in procedure of Theorem \ref{Thm:gen-filling-in} allows to describe a point $d_A\in\Pi$ in two different, but equivalent ways: First using the canonical base point $d_{A_0}\in \Pi$ we can write
$$d_A=d_{A_0}+t$$
where $t\in(\ker\,\ad_\Phi\cap\ker\,\ad_{\Phi^{*}})^\rho$ is a covector. The second description simply decomposes $d_A$ into its $\sigma$-invariant and $\sigma$-anti-invariant part:
$$d_A=d_{A^\sigma}+A^{-\sigma}.$$ 
Note that in general, neither $d_{A^\sigma}$ nor $A^{-\sigma}$ are in $\Pi$. Proposition \ref{Prop:affine-iso} shows that using the pairing with $Z(\Phi)\bar{K}$ to parametrize $d_A\in\Pi$ gives the same result whether we use $t$ or $A^{-\sigma}$. This is why we also refer to $A^{-\sigma}$ as being the covector.

\medskip
We formulate an extension of our Question \ref{main-conj} including the data of a covector. Let $(P, \Phi)$ be a $G$-Fock bundle over $S$ equipped with a hermitian structure $\rho$ such that $\Phi\in\mathcal{P}$ is positive. Denote by $\mu$ the $\g$-complex structure on $S$ induced by $\Phi$ and let $t\in(Z(\Phi)\bar{K})^*$ be a covector. From Theorem \ref{Thm:gen-filling-in} we know that there is a unique unitary $\Phi$-compatible connection $d_A$ described by $t$. We need the following definition.
\begin{definition}
    A covector $t$ is called \emph{$\mu$-holomorphic} if the associated connection $d_A$ has curvature $F(A)\in\mathrm{Im}(\ad_\Phi)\subset\Omega^2(S,\g_P)$.
\end{definition}

We will see below in Theorem \ref{Thm:mu-holo} that this condition coincides with the usual $\mu$-holomorphicity from Equation \eqref{Eq:mu-holo-cond-hcs} in the case of $\mathrm{SL}_n(\C)$. For now, we show that this notion only depends on $\mu$ and $t$:
\begin{proposition}
   The condition $F(A)\in\mathrm{Im}(\ad_{\Phi})$ only depends on $t$ and the $\g$-complex structure $\mu$ induced by $\Phi$.
\end{proposition}
\begin{proof}
    The condition does clearly only depend on the isomorphism class of the Fock bundle since both $\Phi$ and $F(A)$ get conjugated under a gauge transformation. To show independence from the choice of $\rho$, note that changing $\rho$ changes $d_A$ by a coboundary term: the new middle term $d_{A'}$ is given by $d_{A'}=d_A+[C,\Phi]$ for some $C\in \Omega^0(S,\g_P)$ (see Theorem \ref{Thm-filling-in}). The curvature change is then
    $$F(A')=F(A)+d[C,\Phi]+[A\wedge [C,\Phi]]+\tfrac{1}{2}[[C,\Phi]\wedge[C,\Phi]].$$
    Modulo $\mathrm{Im}(\ad_{\Phi})$, the second term equals $[C,d\Phi]$, the third terms equals $[C,[A\wedge \Phi]]$ (using the Jacobi identity) and the last term vanishes (using again the Jacobi identity and $(\ad_\Phi)^2=0$). Therefore $F(A')-F(A)\equiv [C,d\Phi+[A\wedge\Phi]]=0 \,\mod\,\mathrm{Im}(\ad_\Phi)$ using $d_A\Phi=0$.
\end{proof}

We can now formulate the extension of our main open question:
\begin{question}\label{main-conj-covec}
Let $(P, \Phi)$ be a $G$-Fock bundle with a hermitian structure $\rho$ such that $\Phi\in\mathcal{P}$ is positive. Denote by $\mu$ the induced $\g$-complex structure on $S$ and consider a covector $t\in(Z(\Phi)\bar{K})^*$. 

For which $\mu$-holomorphic $t$ there is a gauge transformation $\eta\in \Omega^0(S,\g_P^{-\rho})$ such that the associated 3-term connection associated to $e^{-\eta}\Phi e^{\eta}$ and covector $t$ is flat?
\end{question}

Note that for the case when $G=\mathrm{PSL}_2(\mathbb{C})$, solutions exist for small $t$ by the work of Donaldson \cite{Donaldson} and Trautwein \cite{Traut} on the space of almost-Fuchsian representations. In general, we expect that a tubular neighborhood of the Fuchsian locus in the complex character variety can be parameterized in this way.

\subsection{$\mu$-holomorphicity}\label{Sec:mu-holo-meaning}

We are now giving the gauge-theoretic meaning of the $\mu$-holomorphicity condition \eqref{Eq:mu-holo-cond-hcs}. For that we consider an $\mathrm{SL}_n(\C)$-Fock bundle $(E,\Phi,g)$ with hermitian metric $h$ and assume that $\Phi$ is positive. Denote by $\sigma$ and $\rho$ the involutions on $\mathfrak{sl}(E)$ induced by $g$ and $h$. Denote by $\mu$ the induced higher complex structure on $S$. Finally, let $A^{-\sigma}$ be a covector. This describes a 3-term connection $\Phi+d_A+\Phi^*$ by Theorem \ref{Thm:gen-filling-in}. 

Recall from the previous subsections that the pairing with $Z(\Phi)\bar{K}$ parametrizes covectors. Since $Z(\Phi)=Z(\Phi_1)$ is generated by powers of $\Phi_1$, we consider
\begin{equation}\label{formula-t}
t_k=\tr \Phi_1^{k-1}A^{-\sigma}_1
\end{equation}
where $2\leq k\leq n$ and $A_1^{-\sigma}$ denotes the $(1,0)$-part of $A^{-\sigma}$.

\begin{theorem}\label{Thm:mu-holo}
Let $(E,\Phi,g)$ be an $\SL_n(\C)$-Fock bundle with hermitian structure $h$ and positive $\Phi$. Denote by $\mu=(\mu_k)_{2\leq k\leq n}$ the induced higher complex structure and let $t=(t_k)_{2\leq k\leq n}$ be a covector data \eqref{formula-t}. Then the $\mu$-holomorphicity condition
$$
(-\bar{\partial}\!+\!\mu_2\partial\!+\!k\partial\mu_2)t_{k}+\sum_{l=1}^{n-k}((l\!+\!k)\partial\mu_{l+2}+(l\!+\!1)\mu_{l+2}\partial)t_{k+l}=0,
$$
for all $k\in\{2,3,...,n\}$, is equivalent to the condition
\begin{equation}\label{Eq:mu-holo-equiv1}
F(A)\in\mathrm{Im}(\mathrm{ad}_\Phi)\subset\Omega^2(S,\mathfrak{sl}(E)),
\end{equation}
where $F(A)$ is the curvature of the unique unitary $\Phi$-compatible connection described by the covector $t$. Using the cohomological description from Section \ref{Sec:var-Fock-fields}, Equation \eqref{Eq:mu-holo-equiv1} is equivalent to
$$[F(A)]=0 \in \mathrm{H}^2(\Phi).$$
\end{theorem}

We can make the condition \eqref{Eq:mu-holo-equiv1} more symmetric in $\Phi$ and $\Phi^*$. Since $d_A$ is unitary, its curvature is $\rho$-invariant. Hence $F(A)\in\mathrm{Im}(\ad_{\Phi})$ is equivalent to $F(A)\in\mathrm{Im}(\ad_{\Phi})\cap\mathrm{Im}(\ad_{\Phi^*})$. 

From Proposition \ref{Prop-no-sigma-cohom}, we know that there is no $\sigma$-invariant $\Phi$-cohomology. Hence it is sufficient to show that $[F(A)^{-\sigma}]=0 \in\mathrm{H}^2(\Phi)$. Decomposing $d_A=d_{A^\sigma}+A^{-\sigma}$, we get
$$F(A)=F(A^\sigma)+A^{-\sigma}\wedge A^{-\sigma}+d_{A^\sigma}(A^{-\sigma}).$$
Therefore $[F(A)]=[d_{A^\sigma}(A^{-\sigma})] \mod \mathrm{H}^2(\Phi)$.

The symplectic pairing on $\Phi$-cohomology induces an isomorphism $\mathrm{H}^2(\Phi)\cong \mathrm{H}^0(\Phi)^*$. Thus, an element in $\mathrm{H}^2(\Phi)$ is trivial if and only if its pairing with $\mathrm{H}^0(\Phi)=Z(\Phi)$ is trivial.
The centralizer $Z(\Phi)$ is generated (as a vector space) by the elements $\Phi_1^k$ for $1\leq k\leq n$. We will see that the condition 
\begin{equation}\label{imp-eq}
\tr \Phi_1^k d_{A^\sigma}(A^{-\sigma})=0
\end{equation}
gives the $\mu$-holomorphicity condition for $t_{k+1}$. 
Before going to the general case, consider the example for $n=2$.

\begin{example}
Fix an arbitrary complex structure on $S$ and fix a standard basis $(F,H,E)$ in $\mathfrak{sl}_2$. We work in a gauge where $\Phi=F \, dz+\mu_2F\, d\bar{z}$. By Equation \eqref{formula-t}, we know that $A^{-\sigma}=t_2E\, dz+\mu_2t_2 E \,d\bar{z}$. Since the only $\sigma$-invariant part of $\mathfrak{sl}_2$ is spanned by $H$, we can put $A^\sigma=aH\, dz+bH\, d\bar{z}$ for some local functions $a$ and $b$.
From $d_{A^\sigma}(\Phi)=0$, we get $2b-2a\mu_2+\partial\mu_2=0$. This yields
\begin{align*}
\tr Fd_{A^\sigma}(A^{-\sigma}) &= -\delbar t_2+\partial(\mu_2t_2)+2a\mu_2t_2-2bt_2 \\
&= -\delbar t_2+\mu_2 \partial t_2+2\partial\mu_2\, t_2,
\end{align*}
which is the $\mu$-holomorphicity condition for $n=2$.
\end{example}

\subsubsection{Interlude: Natural basis from a principal $\mathfrak{sl}_2$-triple}\label{Sec:interlude}

Consider a complex simple Lie algebra $\g$. By a theorem of Kostant (see Theorem \ref{Thm:one-reg-orbit}), we know that there is a unique principal $\mathfrak{sl}_2$-triple in $\g$ up to conjugation. Fix such a triple $(E,H,F)$.
It induces two decompositions of $\g$. First, by weights of $\ad_H$:
$$\g \cong \bigoplus_{k\in\mathbb{Z}} \g_k, \;\; \text{ where } \;\; \g_k=\{g\in \g\mid [H,g]=k g\}.$$
Second, by the action with the bracket, $\g$ becomes an $\mathfrak{sl}_2$-module which can be decomposed into irreducible representations:
$$\g \cong \bigoplus_{i\in\mathbb{N}} n_iV_{i},$$
where $V_i$ is the irreducible representation of $\mathfrak{sl}_2$ of dimension $2i+1$ and $n_i\in \mathbb{N}$ the multiplicities.

In the sequel, we only consider $\g=\mathfrak{sl}_n(\C)$. Then, we know that $n_i=1$ for $1\leq i\leq n-1$ and $n_i=0$ otherwise. Using both decompositions, we get
\begin{equation}\label{line-decompo}
\g\cong \bigoplus_{k,i} \g_k\cap V_{i},
\end{equation}
which is a line decomposition; see also Figure \ref{Fig:g-decompo}.

\begin{figure}[h!]
\centering
\includegraphics[height=4cm]{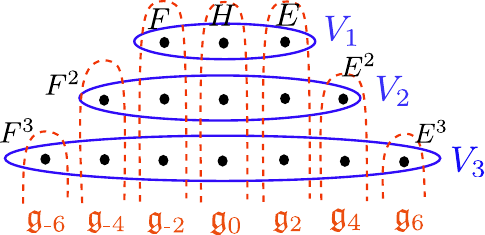}

\caption{Line decomposition of $\mathfrak{sl}_n$ by a principal triple}\label{Fig:g-decompo}
\end{figure}

All irreducible representations of $\mathfrak{sl}_2$ are highest weight representations. This means that for a given irreducible representation $V$, there is a highest weight vector $v\in V\backslash\{0\}$ with $E.v=0$. Then acting successively with $F$ generates all of $V$. 

In our setting, the highest weight vector of $V_i$ is given by $E^i$. Hence a basis adapted to the line decomposition \eqref{line-decompo} is given by
\begin{equation}\label{def-G}
G_{i,j}=\ad_F^{i-j}(E^i) \in V_i\cap \g_{2j},
\end{equation}
where $i\in \{1,...,n-1\}$ and $j\in \{-i,-i+1,...,i-1,i\}$.

A nice property of this line decomposition is its behavior under the trace:
$$\tr(G_{i,j}G_{k,\ell}) = 0, \;\text{ if }\; (k,\ell)\neq (i,-j).$$
In terms of Figure \ref{Fig:g-decompo}, the proposition says that the trace of a product of two elements of the basis is only non-zero if the two corresponding dots lie symmetric with respect to the middle axis. More details can be found in \cite{Malek} (in particular Section 4), where the Lie bracket is computed in the basis $(G_{i,j})$ using a graphical calculus.

\subsubsection{Proof of Theorem \ref{Thm:mu-holo}}

The proof of Theorem \ref{Thm:mu-holo} is a nice but lengthy computation, using principally the Lie theory of the decomposition \eqref{line-decompo}. Since Theorem \ref{Thm:mu-holo} is a local statement, we can work in local coordinates. To do that, we fix an arbitrary complex structure on $S$.

The general setting is the following: using a gauge transformation, we can fix 
$$\Phi(z,\bar{z})=F\, dz+ Q(F)\,d\bar{z}\;\;\text{ where } Q(F)=\textstyle\sum_{k=2}^n \mu_k(z,\bar{z}) F^{k-1}.$$
Further put $A^{-\sigma}=B\,dz+C\, d\bar{z}$. From $[A^{-\sigma}\wedge \Phi]=0$ we have 
\begin{equation}
    [F,C]=[Q(F),B].
\end{equation}
Finally, we put $A^\sigma=M_1\, dz+M_2\, d\bar{z}$ where $M_i\in \g^\sigma$. Since $\g^\sigma\subset \mathrm{Im}(\ad_F)$, we can write $M_i=[F,N_i]$ with suitable $N_i\in \g$, where $i=1,2$.

Recall from Equation \eqref{imp-eq} that we want to compute $$\tr F^k d_{A^\sigma}(A^{-\sigma})= \tr(F^k dA^{-\sigma}) + \tr (F^k[A^\sigma\wedge A^{-\sigma}]).$$

Using the above notations and the definition $t_{k+1}=\tr F^kB$ we get
\begin{align}\label{Eq-proof-mu-holo-0}
\tr(F^kd_{A^\sigma}(A^{-\sigma})) &= \tr(F^k(\partial C-\delbar B)) + \tr(F^k([M_1,C]+[B,M_2])) \nonumber\\
&= -\delbar t_{k+1}+\tr(F^k\partial C)+\tr M_1[C,F^k]+\tr M_2[F^k,B].
\end{align}
We analyze the different terms separately.

\medskip
\noindent \underline{Step 1:}
We compute
\begin{align*}
\tr M_1[C,F^k] &=\tr [F,N_1][C,F^k] = \tr N_1[[C,F^k],F] \\
&= \tr N_1[[C,F],F^k]= \tr N_1[[B,Q(F)],F^k] \\
&= \tr N_1[[B,F^k],Q(F)] = \tr [B,F^k][Q(F),N_1]
\end{align*}
where we used $[C,F]=[B,Q(F)]$, cyclicity of the trace and the Jacobi identity.

Combined with $M_2=[F,N_2]$ we get
\begin{equation}\label{Eq-proof-mu-holo-1}
\tr M_1[C,F^k]+\tr M_2[F^k,B] = \tr [F^k,B]X
\end{equation}
where $X=[N_1,Q(F)]+[F,N_2]$.

\medskip
\noindent \underline{Interlude:}
From the flatness equation $d_{A^\sigma}(\Phi)=0$, we get an important relation for $X$:
\begin{align*}
0=d_{A^\sigma}(\Phi)&= \partial Q(F)+[M_1,Q(F)]+[F,M_2]\\
&= \partial Q(F)+[[F,N_1],Q(F)]+[F,[F,N_2]] \\
&= \partial Q(F)-[[N_1,Q(F)]+[F,N_2],F],
\end{align*}
where we used the Jacobi identity. Therefore:
\begin{equation}\label{eq-for-X}
\partial Q(F)=[X,F].
\end{equation}

We can explicitly solve Equation \eqref{eq-for-X} in $X$. From the line decomposition of $\mathfrak{sl}_n{\C}$ (see Figure \ref{Fig:g-decompo}), we see that we should look for a solution of the form $X=\sum_{\ell=2}^n \alpha_\ell \partial\mu_\ell\,[F^{\ell-1},E]$. Then we get
\begin{align*}
\sum_{\ell=2}^n \partial\mu_\ell\, F^{\ell-1} = \partial Q(F) &= [X,F]=\sum_{\ell=2}^n \alpha_\ell\partial\mu_\ell[[F^{\ell-1},E],F]\\
&= \sum_{\ell=2}^n \alpha_\ell\partial\mu_\ell (2\ell-2)F^{\ell-1}
\end{align*}
using the Jacobi identity and $F^{\ell-1}\in \g_{-2\ell+2}$. Therefore we find $\alpha_\ell=\frac{1}{2\ell-2}$, so
\begin{equation}\label{sol-X}
X=\sum_{\ell=2}^n \frac{\partial\mu_\ell}{2\ell-2}\,[F^{\ell-1},E].
\end{equation}
We see that we found $X$ up to an element of $\ker\,\ad_F$. This choice does not matter since we compute the symplectic pairing with $Z(\Phi)$.

\medskip
\noindent \underline{Step 2:}
Now we come back to Equation \eqref{Eq-proof-mu-holo-1} using the explicit expression \eqref{sol-X} for $X$:
\begin{equation}\label{Eq-proof-mu-holo-2}
\tr [F^k,B]X = \sum_{\ell=2}^n \frac{\partial\mu_\ell}{2\ell-2} \tr[F^k,B][F^{\ell-1},E].
\end{equation}
Using cyclicity of the trace and Jacobi identity, we get
\begin{align}\label{Eq:proof-mu-holo-16}
\tr [F^k,B][F^{\ell-1},E] &= \tr[F,\sum_{j=0}^{k-1}F^jBF^{k-1-j}][F^{\ell-1},E] \nonumber\\
&= \sum_{j=0}^{k-1}\tr F^j BF^{k-1-j}[[F^{\ell-1},E],F] \nonumber\\
&= \sum_{j=0}^{k-1}\tr F^j BF^{k-1-j}[-H,F^{\ell-1}] \nonumber\\
&= (2\ell-2)\sum_{j=0}^{k-1}\tr F^j BF^{k-1-j}F^{\ell-1} \nonumber\\
&= (2\ell-2)k\tr BF^{k+\ell-2} \nonumber\\
&= (2\ell-2)kt_{k+\ell-1},
\end{align}
where we used $F^{\ell-1}\in \g_{2-2\ell}$ and the definition of $t_{k+\ell-1}$.

Therefore Equation \eqref{Eq-proof-mu-holo-2}, using Equation \eqref{Eq-proof-mu-holo-1}, becomes
\begin{equation}\label{Eq-proof-mu-holo-3}
\tr M_1[C,F^k]+\tr M_2[F^k,B] = \tr [F^k,B]X = \sum_{\ell=2}^n k (\partial\mu_\ell) t_{k+\ell-1}.
\end{equation}

\medskip
\noindent \underline{Step 3:}
The remaining term in Equation \eqref{Eq-proof-mu-holo-0} can now be computed:
\begin{align*}
\tr F^k\partial C = \partial \tr(F^kC) &= \frac{1}{2k}\partial \tr([F^k,H]C)\\
&= \frac{1}{2k}\partial \tr(F^k[[E,F],C])\\
&= \frac{1}{2k}\partial \tr(F^k([[C,F],E]+[[E,C],F]))\\
&= \frac{1}{2k}\partial \tr(F^k[[B,Q(F)],E])\\
&= \frac{1}{2k}\partial\left(\sum_{\ell=2}^n \tr(\mu_\ell F^k[[B,F^{\ell-1}],E])\right)\\
&= \frac{1}{2k}\partial\left(\sum_{\ell=2}^n \mu_\ell\tr([F^k,B][F^{\ell-1},E])\right).
\end{align*}
Using $\tr [F^k,B][F^{\ell-1},E] =(2\ell-2)kt_{k+\ell-1}$ (Equation \eqref{Eq:proof-mu-holo-16}), we conclude
\begin{equation}\label{Eq-proof-mu-holo-4}
\tr F^k\partial C = \sum_{\ell=2}^n(\ell-1)\partial(\mu_\ell t_{k+\ell-1}).
\end{equation}

Putting Equations \eqref{Eq-proof-mu-holo-0}, \eqref{Eq-proof-mu-holo-3} and \eqref{Eq-proof-mu-holo-4} together, we get that the condition $\tr(F^kd_{A^\sigma}A^{-\sigma})=0$ is equivalent to
$$-\delbar t_{k+1}+\sum_{\ell=2}^n\left((k+\ell-1)\partial\mu_\ell+(\ell-1)\mu_\ell\partial\right)t_{k+\ell-1}=0,$$
which is exactly the $\mu$-holomorphicity condition.

\begin{remark}
    It is surprising that in the proof above, we never used integration by parts. The statement is somehow pointwise true (without being a pointwise statement).
\end{remark}

\subsection{Higher diffeomorphism action on covectors}\label{Sec:higher-diffeos-on-covectors}

We extend the gauge-theoretic implementation of the action of higher diffeomorphisms on flat Fock bundles in Section \ref{Sec:higher-diff-action-via-gauge} to the setting of covectors. In the case of $\mathrm{SL}_n(\C)$, we recover the variation formula of the covector data $t_k$ from the theory of higher complex structures.

Consider a $G$-Fock bundle $(P,\Phi,\sigma)$ with a hermitian structure $\rho$ (not necessarily commuting with $\sigma$) such that $\Phi$ is positive. Use $\sigma$ to decompose a unitary $\Phi$-compatible connection $d_A=d_{A^{\sigma}}+A^{-\sigma}$, where $A^{-\sigma}$ is the covector.

Recall from Section \ref{Sec:higher-diffeos} that a higher diffeomorphism is a special $\l$-dependent gauge transformation $\l^{-1}\xi+\l\xi^{*}$ where $\xi\in Z(\Phi)$. The variations of $\Phi$ and $d_A$ are given by $\delta\Phi=d_A\xi$ and $\delta A=[\xi^{*},\Phi]+[\xi,\Phi^{*}]$.

In the framework of covectors, we define the action of higher diffeomorphisms in the same way. The only problem is that the variation $\delta \Phi=d_A\xi$ does not preserve $\sigma(\Phi)=-\Phi$ since $\xi\in Z(\Phi)$ is $\sigma$-anti-invariant but $d_A$ is not $\sigma$-invariant. Thus, higher diffeomorphisms also change $\sigma$.

There is a way out to keep the same $\sigma$: we can use a usual gauge transformation $\eta$ to make $\delta\Phi$ anti-invariant under $\sigma$. The infinitesimal gauge transformation $\l^{-1}\xi+\eta+\l\xi^{*}$ induces the variation
$$\delta\Phi=d_{A^\sigma}(\xi)+[A^{-\sigma},\xi]+[\eta,\Phi].$$
The first term is $\sigma$-anti-invariant, we wish to choose $\eta$, so that the rest two terms will vanish. Then, the action on $\Phi$ is simply given by $\delta\Phi=d_{A^\sigma}(\xi)$.

\begin{proposition}
    There exists $\eta\in\Omega^0(S,\g_P^{-\sigma})$ such that $[A^{-\sigma},\xi]+[\eta,\Phi]=0$.
\end{proposition}
\begin{proof}
    This is an application of Corollary \ref{coboundaries} stating that the bracket between cohomology classes is a coboundary. We have $\xi\in Z(\Phi)=\mathrm{H}^0(\Phi)$ and $A^{-\sigma}\in\text{ker}\,\text{ad}_{\Phi}$. Since $\xi\in Z(\Phi)$ the bracket $[A^{-\sigma},\xi]$ only depends on the class $[A^{-\sigma}]\in \mathrm{H}^1(\Phi)$. In addition, the bracket is a coboundary which gives the existence of $\eta$. Since $\Phi, \xi$ and $A^{-\sigma}$ are $\sigma$-anti-invariant, so is $\eta$.
\end{proof}

\medskip
Now we restrict attention to the case of an $\mathrm{SL}_n(\C)$-Fock bundle $(E,\Phi,g)$ and compute explicitly the action on the covector data. Put $\xi=\Phi(v_1)\cdots \Phi(v_k)$ which is associated to the Hamiltonian $H=v_1\cdots v_k$. Consider the infinitesimal gauge transformation $\l^{-1}\xi+\eta+\l\xi^{*}$, where $\eta$ is explicitly given by
\begin{equation}\label{def-eta-0}
\eta = A^{-\sigma}(v_1)\Phi(v_2)\cdots \Phi(v_k)+...+\Phi(v_1)\cdots \Phi(v_{k-1})A^{-\sigma}(v_k).
\end{equation}
Note that $\sigma(\eta)=-\eta$ and that $\eta=0$ if $t_k=0$ for all $k$ (recall that $t_k=\tr \Phi_1^{k-1}A_1^{-\sigma}$).

\begin{proposition}
The section $\eta$ is well-defined, i.e. it does not depend on the order of the $v_i$.
\end{proposition}
\begin{proof}
This will follow from $\Phi\wedge\Phi=0$ and $[A^{-\sigma}\wedge \Phi]=0$. Let us compute what happens if we exchange $v_1$ and $v_2$. From $[A^{-\sigma}\wedge \Phi](v_1,v_2)=0$ we get $$[A^{-\sigma}(v_1),\Phi(v_2)]=[A^{-\sigma}(v_2),\Phi(v_1)].$$
Hence $$A^{-\sigma}(v_1)\Phi(v_2)+\Phi(v_1)A^{-\sigma}(v_2) = A^{-\sigma}(v_2)\Phi(v_1)+\Phi(v_2)A^{-\sigma}(v_1).$$
Since the $\Phi(v_i)$ pairwise commute, $\eta$ will not change under the exchange between $v_1$ and $v_2$. In a similar way, this is proved for any exchange between $v_i$ and $v_{i+1}$. Hence $\eta$ is invariant under permutation.
\end{proof}

\begin{proposition}
We have $[A^{-\sigma},\xi]+[\Phi,\eta]=0$.
\end{proposition}
\begin{proof}
Since the equation to prove is additive, we can assume the Hamiltonian to be $H=v_1\cdots v_k$. Then, the first term is 
$$[A^{-\sigma},\Phi(v_1)\cdots \Phi(v_k)] = \sum_{i=0}^{k}\Phi(v_1)\cdots \Phi(v_{i-1})[A^{-\sigma},\Phi(v_i)]\Phi(v_{i+1})\cdots \Phi(v_k).$$
Using $\Phi\wedge\Phi=0$ the second term is
$$\sum_{i=0}^{k}\Phi(v_1)\cdots \Phi(v_{i-1})[\Phi,A^{-\sigma}(v_i)]\Phi(v_{i+1})\cdots \Phi(v_k).$$
The flatness identity $[A^{-\sigma},\Phi]=0$ implies $[A^{-\sigma},\Phi(v_i)]+[\Phi,A^{-\sigma}(v_i)]=0$, for every $i$,  which concludes the proof.
\end{proof}

As a consequence, we get
\begin{equation}\label{higher-diff-action-extended-2}
\delta \Phi = d_{A^\sigma}(\xi).
\end{equation}
In particular, the property $\sigma(\Phi)=-\Phi$ is preserved.

The variation of the middle term $d_A$ is given by (see also Equation \eqref{Eq:3-term-gauge-action}): $$\delta A=d_A\eta+[\xi,\Phi^{*}]+[\xi^{*},\Phi].$$
Thus, the variation of $A^{-\sigma}$, which is important for the covector data, is given by
\begin{equation}\label{higher-diff-action-extended-1}
\delta A^{-\sigma} = d_{A^\sigma}(\eta)+[\xi,(\Phi^{*})^\sigma]+[(\xi^{*})^{\sigma},\Phi].
\end{equation}

We can now state the variation formula of the differentials $t_k$.
\begin{proposition}\label{Prop:var-covectors}
The variation of $t_k$ under a Hamiltonian $H=w_\ell p^{\ell-1}$ is given by 
$$\delta t_k = (k+\ell-2)t_{k+\ell-2}\partial w_\ell+(\ell-1)w_\ell\partial t_{k+\ell-2}.$$
\end{proposition}

This is in accordance to the computations from the classical perspective on higher complex structures. Before giving the proof which is a direct computation, let us see an example:
\begin{example}
For $n=2$, the $\mu$-holomorphicity condition reads $(-\delbar+\mu_2\partial+2\partial\mu_2)t_2=0$. A Hamiltonian $H=w_2p$ induces 
\begin{align*}
\delta \mu_2 &= (\delbar-\mu_2\partial+\partial\mu_2)w_2,\\
\delta t_2 &= w_2\partial t_2+2t_2 \partial w_2.
\end{align*}
One can check that the $\mu$-holomorphicity is preserved. 
\end{example}
Note that the $\mu$-holomorphicity condition is preserved under the higher diffeomorphism action. This is because it is a gauge action, hence it preserves the flatness of the 3-term connection.

\begin{proof}[Proof of Proposition \ref{Prop:var-covectors}]
We start in a gauge in which $\Phi_1=Fdz$ is a fixed principal nilpotent element ($F$ will only change under the higher diffeomorphism action). Recall the definition $t_k=\tr(F^{k-1}A^{-\sigma}_1)$ and the notation $B=A^{-\sigma}_1$, from which we get 
$$\delta t_k=\tr\left(\delta\left(F^{k-1}\right)B + F^{k-1}\delta B\right).$$
From Equations \eqref{higher-diff-action-extended-1} and \eqref{higher-diff-action-extended-2}, we get
$$\delta\left(F^{k-1}\right) = \sum_{i=0}^{k-2}F^i (\delta F)F^{k-i-2} = \sum_{i=0}^{k-2}F^i (\partial\xi+[A^\sigma_1,\xi])F^{k-2-i}$$
and
$$\delta B = \partial\eta+[A^\sigma_1,\eta]+[\xi,(\Phi_2)^{*,\sigma}]+[\xi^{*,\sigma},F].$$
The last two terms of $\delta B$ will not contribute to the variation $\delta t_k$. Indeed, $\tr F^{k-1}[\xi^{*,\sigma},F]=0$ by cyclicity of the trace and $\tr F^{k-1}[\xi,(\Phi_2)^{*,\sigma}]=0$, since $[\xi,F^{k-1}]=0$.

\smallskip
Using $\xi=w_\ell F^{\ell-1}$ and $\eta=w_\ell\sum_{j=0}^{\ell-2}F^j B F^{\ell-2-j}$, we get
$$\delta\left(F^{k-1}\right) = (k-1)\partial w_\ell \,F^{k+\ell-3} + w_\ell\sum_{i=0}^{k-2}F^i[A^\sigma_1,F^{\ell-1}]F^{k-2-i}$$
and 
$$\partial\eta+[A^\sigma_1,\eta] = \partial w_\ell \sum_{j=0}^{\ell-2}F^jBF^{\ell-2-j}+w_\ell\sum_{j=0}^{\ell-2}F^j (\partial B) F^{\ell-2-j}+w_\ell[A^\sigma_1,\sum_{j=0}^{\ell-2}F^jBF^{\ell-2-j}].$$
Therefore using $t_{k+\ell-2}=\tr F^{k+\ell-3}B$ we get
\begin{align*}
\delta t_k &= \tr \delta(F^{k-1})B+\tr F^{k-1}\delta B\\
&= (k-1)\partial w_\ell t_{k+\ell-2}+\partial w_\ell \sum_{j=0}^{\ell-2}t_{k+\ell-2}+w_\ell\sum_{j=0}^{\ell-2}\partial t_{k+\ell-2} \\
& \; +w_\ell \sum_{i=0}^{k-2}\tr\left(F^i[A^\sigma_1,F^{\ell-1}]F^{k-2-i}B\right)+w_\ell\sum_{j=0}^{\ell-2}\tr\left(F^{k-1}[A^\sigma_1,F^jBF^{\ell-2-j}]\right)\\
&= \left((k+\ell-2)\partial w_\ell+(\ell-1)w_\ell\partial\right)t_{k+\ell-2}.
\end{align*}
The last line comes from the cancellation of two terms. Using cyclicity we get:
\begin{align*}
\tr\sum_{i=0}^{k-2}F^i[A^\sigma_1,F^{\ell-1}]F^{k-2-i}B &= \tr A^\sigma_1 \sum_{i=0}^{k-2}[F^{\ell-1},F^{k-2-i}BF^i] \\
&= \tr A^\sigma_1 \sum_{i=0}^{k-2}\sum_{j=0}^{\ell-2}F^{k-2-i}F^j[F,B]F^{\ell-2-j}F^i.
\end{align*}
Similarly
$$\tr\sum_{j=0}^{\ell-2}F^{k-1}[A^\sigma_1,F^jBF^{\ell-2-j}] = \tr A^\sigma_1 \sum_{j=0}^{\ell-2}\sum_{i=0}^{k-2}F^jF^{k-2-i}[B,F]F^iF^{\ell-2-j}.$$
Hence the two terms cancel out.
\end{proof}

\section*{Statement of no conflict of interest}

On behalf of all authors, the corresponding author states that there is no conflict of interest. 

\section*{Data availability statement}

The present manuscript has no associated data.

\bigskip
\noindent\small{\textsc{Department of Mathematics, University of Patras}\\ University Campus, Patras 26504, Greece
	}\\
\emph{E-mail address}:  \texttt{gkydonakis@math.upatras.gr}

\bigskip
\noindent\small{\textsc{Department of Mathematics, Yale University}\\
PO Box 208283, New Haven, CT 06520-828, USA
}\\
\emph{E-mail address}:  \texttt{charles.reid@yale.edu}

\bigskip
\noindent\small{\textsc{Institut Camille Jordan, Universit\'{e} Lyon 1}\\
	21 Av Claude Bernard, 69100 Villeurbanne, France}\\
\emph{E-mail address}: \texttt{athomas@math.univ-lyon1.fr}

\end{document}